\documentclass[12pt]{amsart}
\usepackage{amssymb,pb-diagram}
\usepackage{amsmath,amscd}
\usepackage{graphicx}
\usepackage{enumerate}

\usepackage{chngcntr}
\counterwithin{figure}{section}

\newtheorem{thm}{Theorem}[section]

\newtheorem{lem}[thm]{Lemma}
\newtheorem{cor}[thm]{Corollary}
\newtheorem{prop}[thm]{Proposition}

\theoremstyle{definition}

\newtheorem{defi}[thm]{Definition}

\theoremstyle{remark}

\newtheorem{rem}[thm]{Remark}

\newtheorem*{remnn}{Remark}

\newcommand{\JK}{\operatorname{JK}}
\newcommand{\FR}{\operatorname{FR}}
\newcommand{\Wh}{\operatorname{Wh}}
\newcommand{\LM}{\operatorname{LM}}
\newcommand{\im}{\operatorname{im}}

\newcommand{\id}{\operatorname{id}}

\newcommand{\Z}{\mathbb{Z}}
\newcommand{\N}{\mathbb{N}}
\newcommand{\R}{\mathbb{R}}
\newcommand{\Q}{\mathbb{Q}}

\newcommand{\e}{\epsilon}
\newcommand{\hra}{\hookrightarrow}
\newcommand{\imra}{\looparrowright}
\newcommand{\sra}{\twoheadrightarrow}
\newcommand{\ira}{\rightarrowtail}
\newcommand{\ra}{\longrightarrow}

\title{The group of disjoint 2-spheres in 4-space}
\author{Rob Schneiderman and Peter Teichner}
\address{Lehman College, City University of New York}
\email{robert.schneiderman@lehman.cuny.edu}
\address{University of California, Berkeley and Max Planck Institute for Mathematics, Bonn }
\email{teichner@mac.com}
\thanks{}

\begin{document}


\begin{abstract}
We compute the group $\LM_{2,2}^4$ of link homotopy classes of link maps of two 2-spheres into 4-space. It turns out to be free abelian, generated by geometric constructions applied to the Fenn--Rolfsen link map and detected by two self-intersection invariants introduced by Kirk in this setting. 
As a corollary, we show that any link map with one topologically embedded component is link homotopic to the unlink. 

Our proof introduces a new basic link homotopy, which we call a {\em Whitney homotopy}, that shrinks an embedded {\em Whitney sphere} constructed from four copies of a Whitney disk. Freedman's disk embedding theorem is applied to get the necessary embedded Whitney disks, after constructing sufficiently many {\em accessory spheres} as algebraic duals for immersed Whitney disks. To construct these accessory spheres and immersed Whitney disks we use the algebra of metabolic forms over the group ring $\Z[\Z]$, and introduce a number of new 4-dimensional constructions, including maneuvers involving the boundary arcs of Whitney disks.

\end{abstract}

\maketitle

\section{Introduction and Statements of Results}

A {\em link map}  is a continuous map that sends connected components of the source {\em disjointly} into the target. A {\em link homotopy} is a homotopy through link maps. 

So even if the components start off as disjoint embeddings, they are allowed to self-intersect (but not intersect each other) during a link homotopy. 
John Milnor  \cite{M1} had initiated the study of link homotopy for classical links in $3$-space as a way to measure ``linking modulo knotting''. His invariants still play a central role in trying to understand some important open problems for topological 4-manifolds. 
In a certain sense, they also describe our current $4$-dimensional link homotopy computation, see Corollary~\ref{cor:JK}. 

The fundamental sets $\LM_{p,q}^n$ of spherical link maps $S^p\amalg S^q\to S^n$ up to link homotopy are computed in many high-dimensional cases \cite{HK,Ko2,Sc}, for example the linking number gives $\LM_{1,n-2}^{n}\cong\Z$ for all $n\geq 3$. 
However, due to the usual difficulties of extracting geometric information from algebraic invariants of surfaces in dimension four, there has been no significant progress towards understanding $\LM_{2,2}^4$ since 1988 when Paul Kirk \cite{Ki} constructed a surjection $\LM_{2,2}^4\sra\bigoplus_\N \Z$.
Our main Theorem~\ref{thm:main} below states that Kirk's invariant is in fact injective which leads to
the following more concise characterization:
\begin{thm}\label{thm:free}
Consider the commutative ring $R:=\Z[z_1,z_2] / (z_1z_2)$. Then $\LM_{2,2}^4$ is a free $R$-module of rank one, freely generated by the Fenn--Rolfsen link map $\FR$. In particular, the abelian group $\LM_{2,2}^4$ is free with basis $\{\FR, z_1^n\cdot \FR, z_2^n\cdot \FR \ | \ n\in\N\}$.
\end{thm}

Here the addition in $\LM_{2,2}^4$ is given by connected sum and the inverse link map arises from reflection of $S^4$ (which is part of our theorem). The action of $z_1$ on $(f_1,f_2)$ comes from doubling $f_1$ and then tubing the two copies together along an arc that represents the meridian to $f_2$; the action of $z_2$ is defined similarly by doubling $f_2$. See Section~\ref{sec:thm-free-proof} for details and Figure~\ref{fig:whitehead-null-htpy} for a picture of the Fenn--Rolfsen link map as the Jin--Kirk construction on the Whitehead link: $\FR=\JK(\Wh)$. 

It is extremely rare that a 4-dimensional problem has such a simple answer and that very explicit operations lead to all possible examples. Even more is true in this setting, one actually just needs the following self-intersection invariant to distinguish all 2-component spherical link maps: 
First turn any link map $(f_1,f_2):S^2\amalg S^2\to S^4$ into a {\em generic immersion} and then consider the Hurewicz maps
$$
\pi_1(S^4 \smallsetminus \im(f_i))\sra H_1(S^4 \smallsetminus \im(f_i)) \cong H^2(\im(f_i)) \cong\Z.
$$
Even though the fundamental group of the complement changes during a homotopy of $f_i$, the first homology stays constant by Alexander duality and the fact that $H^2$ is infinite cyclic for the image of any generic immersion.
One can thus consider Wall's intersection invariant 
$\lambda$ with values in the group ring $\Z[\Z]$ to obtain
two well-defined link homotopy invariants
$\lambda(f_i,f_i)\in \Z[\Z]$. This was first observed by Kirk in \cite{Ki}, hence we shall refer to these two Laurent polynomials as the {\em Kirk invariants} 
\[
\sigma_i(f_1,f_2) := \lambda(f_i,f_i)\in \Z[\Z].
\]

One might expect the existence of higher-order invariants for spheres in such a 4-dimensional setting but our main result says that the Kirk invariants classify 
$\LM_{2,2}^4$:
\begin{thm}\label{thm:main}
If $(f_1,f_2): S^2 \amalg S^2 \to S^4$ is a link map with vanishing Kirk invariants $\sigma_1(f_1,f_2)=0=\sigma_2(f_1,f_2)$, then $(f_1,f_2)$ is link homotopically trivial. 
\end{thm}

Fix a (multiplicative) generator $x$ of $\Z$ so that $\Z[\Z] = \Z[x^{\pm 1}]$ is the ring of Laurent polynomials. 
Here and throughout the paper we'll often need the Laurent polynomial 
\[
z := (1-x)\cdot (1-x^{-1}) = (1-x) + (1-x^{-1}) \in \Z[x^{\pm 1}].
\]
In Section~\ref{sec:proof-thm-image}, we'll derive from Theorem~\ref{thm:main} our main computation:
\begin{thm}\label{thm:image}
The Kirk invariants $(\sigma_1, \sigma_2)$ give a short exact sequence of abelian groups
\[
0 \ra \LM^4_{2,2} \ra  z\cdot \Z[z]\oplus z\cdot \Z[z] \ra \Z \ra 0
\]
On the right we use the addition map to $\Z = z\cdot \Z[z]/z^2\cdot \Z[z]$. 
\end{thm}

With slightly different notation, Kirk computed the cokernel of his invariants in \cite{Ki}, our contribution is the injectivity of his invariants
and the relation to Theorem~\ref{thm:free}. In Section~\ref{sec:thm-free-proof} we will embed $\Z[z]$ into the ring $R$, from Theorem~\ref{thm:free}, by sending $z$ to $z_1+z_2$ which will translate between these results.


\subsection{Seeing 2--spheres in 4--space via classical links}\label{sec:JK}
Let $\mathcal L$ denote the set of links in $\R^3$ with two components, both of which are unknotted and have linking number~$0$. There is an additive map (under connected sum)
\[
\JK:\mathcal L \ra \LM_{2,2}^4
\]
which seems to go back to Fenn, Rolfsen, Levine, Jin and Kirk. We shall refer to the map $\JK$ as the {\em Jin--Kirk construction} and it can be described as follows: For $L=(l_1,l_2)\in\mathcal L$ observe that $S^3\smallsetminus l_i$ are homotopy equivalent to $S^1$ because the knots $l_i$ are trivial, so each component is null-homotopic in complement of the other by triviality of the linking number. Thus one can take the track of a null-homotopy of $l_2$ in the complement of $l_1$ and then a spanning disk for $l_1$ to describe a link map $D^2 \amalg D^2 \to D^4$ bounded by $L$, see Figure~\ref{fig:whitehead-null-htpy}. Performing the same construction with the roles of $l_1$ and $l_2$ switched and then taking the union with the previous construction along $L$ yields a link map $\JK(L):S^2\amalg S^2 \to S^4$. Note that this construction involves the choices of two null-homotopies but it gives a well defined link homotopy class because $\pi_2(S^3\smallsetminus l_i)\cong\pi_2(S^1)=0$.

\begin{figure}[ht!]
         \centerline{\includegraphics[scale=.275]{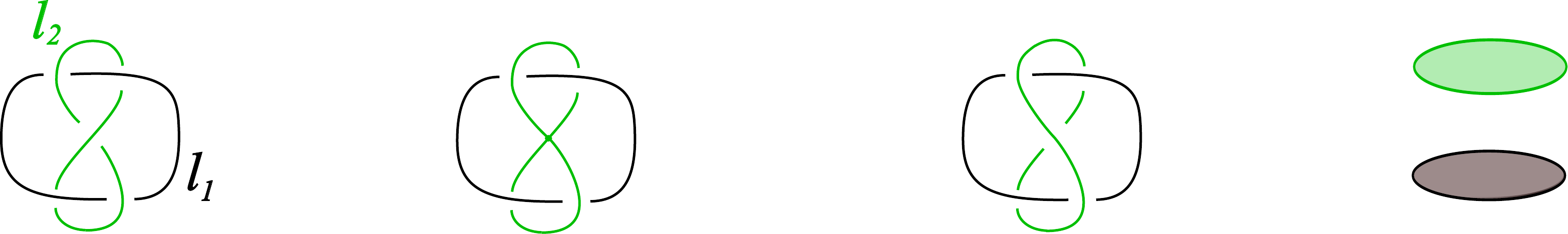}}
         \caption{One half of the \emph{Fenn-Rolfsen link map} $\FR=\JK(\Wh)$, the other half uses the symmetry of the Whitehead link $\Wh$.}
         \label{fig:whitehead-null-htpy}
\end{figure}
Figure~\ref{fig:whitehead-null-htpy} shows a \emph{movie} of one half $D^2 \amalg D^2 \to D^4$ of the Fenn--Rolfsen link map $\FR$ from Theorem~\ref{thm:free} defined as the Jin--Kirk construction on the Whitehead link $L=\Wh$: The left-most picture shows $\Wh$ in the equatorial $S^3$ of $S^4$; then, moving to the right into $S^3\times I\subset D^4$, the track of a null-homotopy of $\Wh$ is described by a crossing change on $l_2$, an isotopy and finally spanning disks for the resulting unlink.
The other half of $\FR$ is constructed using the symmetry of the Whitehead link to first swap the link components by an isotopy and then apply the crossing-change to $l_1$ while moving into the other hemisphere of $S^4$.

The crossing change on $l_2$ in Figure~\ref{fig:whitehead-null-htpy} corresponds to a negative self-intersection in the second component $f_2$ of $\FR$ with a double point loop representing a meridian to $f_1$.
The Kirk invariant is computed by counting intersections between $f_2$ and a parallel push-off $f'_2$:
\[
\sigma_2(\FR) := \lambda(f_2,f'_2) = \mu(f_2) +\iota\mu(f_2) + e(\nu(f))\cdot 1 = -x -x^{-1} +2 = z,
\]
where $\iota$ is the involution on $\Z[\Z]$ given by $x\mapsto x^{-1} $.
This is Wall's formula \cite{Wa} relating self-intersection and intersection invariants quite generally, see Section~\ref{sec:standard-position}.

The orientation switch in the construction of the other half of $\FR$ has the effect of switching the signs of the self-intersections in $f_1$, and we get $\sigma_1(\FR)=x+x^{-1}-2=-z$.

The additivity of $\JK$ means that {\em any} ambient connected sum of classical $2$-component links is sent to the sum of the corresponding elements in the group $\LM_{2,2}^4$. 
It was shown by Jin in his thesis \cite{J} that the Kirk invariants of $\JK(L)$ are Cochran's $\beta^i$-invariants \cite{Co} of $L$. More precisely, in the above notation, we can write
\[
\sigma_2(\JK(L)) = \sum_{i=1}^N \beta_2^i(L) \cdot z^i \in z\cdot \Z[z]\quad\mbox{ and }\quad -\sigma_1(\JK(L)=\sum_{i=1}^N \beta_1^i(L) \cdot z^i \in z\cdot \Z[z]
\]
Cochran had shown that his invariants $\beta_1^i(L)$ are integer lifts of Milnor's invariants $\mu_{1\dots_{2i}\dots 122}(L)$ for links with trivial linking numbers, and similarly for $\beta_2^i(L)$. In particular, the Sato--Levine invariant $\beta_1^1(L) = \mu_{1122}(L)=\beta_2^1(L) $ is symmetric in the components 1 and 2, and changes sign under an orientation change of $S^3$, explaining again the cokernel of the Kirk invariants in Theorem~\ref{thm:image}. 

It follows from Theorem~\ref{thm:main} that $\JK(L) = \JK(L')$ if and only if $L$ and $L'$ share the same Cochran $\beta$-invariants. Moreover, the vanishing of Cochran's invariants is equivalent the vanishing of the corresponding Milnor invariants (which are a priori not well defined as integers), implying:
\begin{cor}\label{cor:JK}
The map $\JK:\mathcal L \sra \LM_{2,2}^4$ is onto. Moreover,  $\JK(L) = 0$ if and only if the following Milnor invariants vanish:
\[
\mu_{1\dots_{2i}\dots 122}(L) = 0 = \mu_{2\dots_{2i}\dots 211}(L) \quad \forall \, i \geq 2.
\]
\end{cor}
This gives a very satisfying 4-dimensional characterization of the vanishing of these Milnor invariants. The surjectivity of $\JK$ can be proven by writing down sufficiently many links $L \in \mathcal L$ to realize the image of the Kirk invariants, see Section~\ref{sec:proof-thm-image}. Alternatively, one can lift the $R$-action from Theorem~\ref{thm:free} to $\mathcal L$ and get away with only knowing the Whitehead link!

\subsection{Pulling away an embedded component}\label{subsec:intro-embedded-component}
Another consequence of  Theorem~\ref{thm:main} is the non-existence of a ``4D Hopf link'', even allowing one wild embedding and the other component to self-intersect arbitrarily in the complement:

\begin{cor} \label{cor:main}
Let $(f_1,f_2): S^2 \amalg S^2 \to S^4$ be a link map such that $f_1$ is an embedding (a homeomorphism onto its image). Then $(f_1,f_2)$ is link homotopic to the trivial link.
\end{cor}

The main difficulty that arises in any proof of this result is that $\pi_2(S^4 \smallsetminus \im(f_1))$ can be nontrivial even if $f_1$ is a smooth embedding, as in the Andrews--Curtis example described in the next section.  One is therefore forced to introduce self-intersections into $f_1$ during the link homotopy and it's difficult to retain the information that $f_1$ started off as an embedding.

Kirk had noticed that {\em both} his
invariants vanish even if only one of the components is smoothly embedded and he asked whether the link map is trivial in this case. We will show in Section~\ref{sec:embedded} that the Kirk invariants also both vanish if $f_1$ is a wild embedding, so Corollary~\ref{cor:main} follows directly from Theorem~\ref{thm:main}.

In the same section, we will show that it already follows from an intermediate result, the Standard Unlinking Theorem~\ref{thm:I2}, which doesn't use our Proposition~\ref{prop:I-squared} (whose proof takes up the last 40 pages of this paper). 

We are not aware of other theorems about general 2-knots in $S^4$, but in \cite{BHV} infinitely many wild $2$--knots were constructed as limit sets of (geometrically finite) Kleinian groups.

\subsection{Some history of embedded 2--spheres in 4--space}\label{sec:history}

In 1925, Emil Artin \cite{A} constructed the first knotted 2-spheres in 4-space by spinning a knotted arc $K$, with ends on a  2-plane in $\R^3$, by 360 degrees. During this rotation of $\R^4$ which keeps the plane fixed, the rotated arcs sweep out a 2-sphere. Artin showed that the fundamental group of the complement of this {\em spun 2-knot} $S_K\subset \R^4$ is that of the original arc and hence that 2-knot theory is at least as complicated as classical knot theory. 
\begin{figure}[ht!]
		\centerline{\includegraphics[scale=.275]{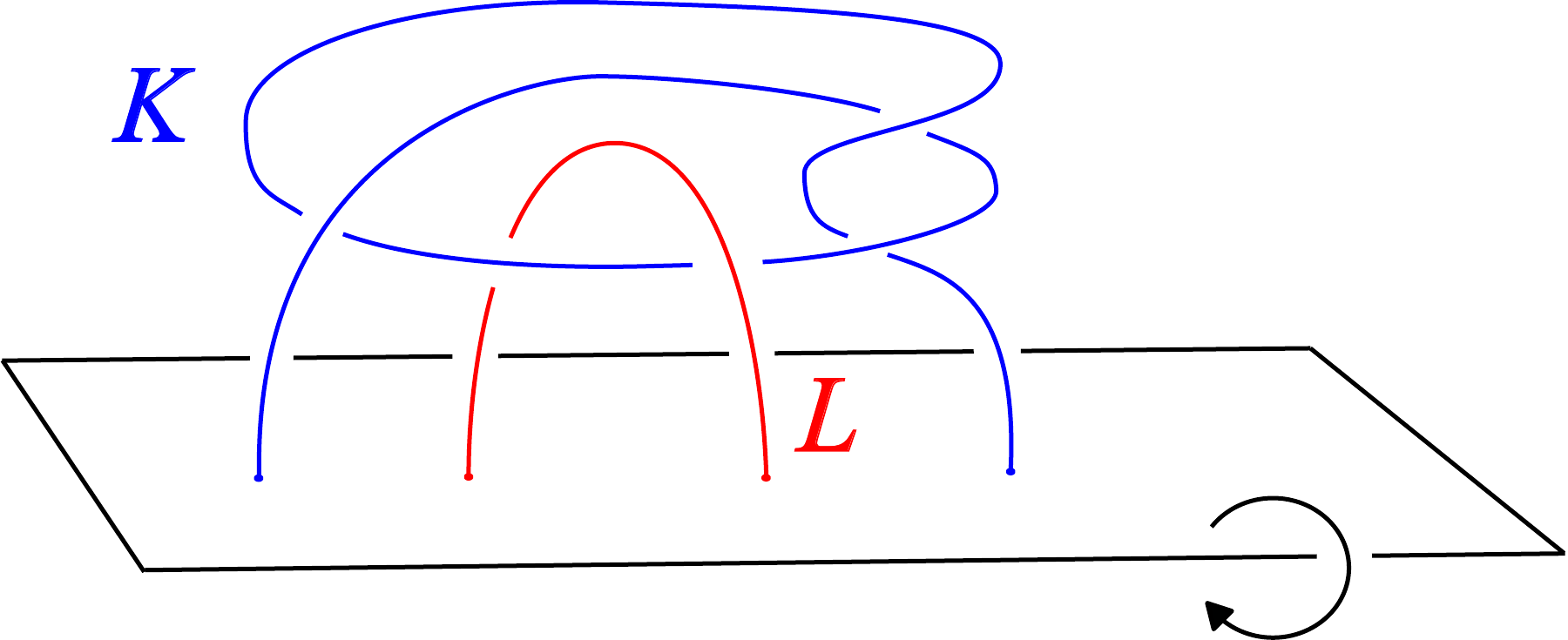}}
         \caption{Artin's spun 2-knot $S_K$ with an Andrews--Curtis 2-sphere $S_L$.}
         \label{fig:spun-2-knots}
\end{figure}

Andrews--Curtis showed in \cite{AC} that the second homotopy group $\pi_2(\R^4 \smallsetminus S_K)$ can be non-trivial. 
In fact, they proved that spinning the arc $L$ in Figure~\ref{fig:spun-2-knots} gives an embedded 2-sphere $S_L$ in the complement of $S_K$ which is not null homotopic. 

Note that since the arc $L$ is unknotted, the spun 2-knot $S_L$ is unknotted as well and hence $\pi_2(\R^4 \smallsetminus S_L)=0$. In particular, $S_K$ shrinks in the complement of $S_L$. By symmetry, this null homotopy can be excluded by introducing a trefoil knot $T$ into $L$ (changing $L$ by a homotopy in the complement of $K$), leading to a new two-component 2-link $(S_K, S_T): S^2 \amalg S^2 \hra \R^4$ such that neither component is null homotopic in the complement of the other.  

However, by Corollary~\ref{cor:main}, this 2-link is homotopically trivial. This can also be seen directly from the construction as follows:
There is a homotopy of $S_T$ in the complement of $S_K$ (leading from the trefoil arc $T$ back to the arc $L$ in Figure~\ref{fig:spun-2-knots}) and then a null homotopy of $S_K$ in the complement of $S_L$. 

The question whether an {\em embedding} of two 2-spheres in $S^4$ is link homotopically trivial
was first posed by Massey and Rolfsen in 1982, see the problem list \cite[p.258]{F}. 
Arthur Bartels and the second author answered this question by proving in  \cite{BT} that an embedded link of
2-spheres in $S^4$ {\em with arbitrarily many components} is link homotopically trivial. 

In 1986, Fenn and Rolfsen \cite{FR} gave the first example of a link map $S^2 \amalg S^2\to S^4$, necessarily with both components non-embedded, which is not link homotopically trivial, Figure~\ref{fig:whitehead-null-htpy}. It is particularly satisfying that 30 years later we find out that this is the free generator $\FR$ in the sense of Theorem~\ref{thm:free}.

\subsection{Related work} The current paper replaces a manuscript that was circulated by the second author in the mid 1990's but was never published. That manuscript contained a gap in the proof of an immersed version of the current Whitney Homotopy in Section~\ref{sec:the-whitney-htpy} which is now filled by an application of Freedman's disk embedding theorem. It remains an open problem whether Freedman's infinite constructions can be replaced by a specific immersed version of our Whitney homotopy. 
 
 All key ideas from the prior manuscript remain central to our current argument.
However, this paper greatly extends the results and most significantly, Theorem~\ref{thm:main} was unknown when the previous manuscript circulated, which contained only Corollary~\ref{cor:main}. 
The new ingredients for proving Theorem~\ref{thm:main} are the manipulations of Whitney arcs and Whitney disks in Section~\ref{sec:thm-main-proof}. 
 We only found these new constructions after developing an obstruction theory for {\em Whitney towers} over the last decade, see e.g.~\cite{CST0,ST1,ST2}.

Completing the history of the topic, we recall that in the paper \cite{L1}, a $\Z/2$-valued invariant to detect nontrivial link maps in the kernel of the Kirk invariants was defined. However, Pilz \cite{Pilz} pointed out a computational error in the examples given in  \cite{L1}. It was shown recently by Lightfoot \cite{Lightfoot} that this invariant vanishes identically on the kernel of the Kirk invariants. In fact, Lightfoot shows that the invariant is defined if $\sigma_1(f_1,f_2)=0$ but is completely determined by $\sigma_2(f_1,f_2)$.


\subsection{Main ideas}\label{subsec:intro-main-ideas}

Our method for proving Theorem~\ref{thm:main} is based on a $4$-dimensional version of an elementary link homotopy in $3$-dimensions
(Figure~\ref{fig:1dim-whitney-homotopy}) as will be explained in Section~\ref{sec:w-spheres-and-homotopy}. The central idea is to consider a {\em Whitney sphere}
corresponding to a (framed embedded) Whitney disk $W$ for a pair of self-intersections of an immersion $f:S^2 \imra X^4$. The Whitney sphere $S_W$ is embedded in the complement  of the image of $f$ and is made out of four copies of the Whitney disk (see \cite[\S 3.1, Ex.(2)]{FQ}, where these spheres were introduced as dual spheres to accessory disks). 

A \emph{Whitney homotopy} is given by shrinking $S_W$ to a point in the complement of $f$ after doing the Whitney move. Then reversing the Whitney move, we obtain a link homotopy
\[
(f,S_W) \simeq (f, *)
\]
in $X$, which we'll refer to as a {\em Whitney homotopy}.

In the case of a link map $(f_1,f_2): S^2 \amalg S^2 \to S^4$, we first arrange for $f_1$ to be in {\em standard position}, see Section~\ref{sec:standard-unlinking}. This means that $f_1$ is obtained from the trivial sphere in $S^4$ by finitely many local finger moves. In particular, the fundamental group of the 4-manifold $S^4 \smallsetminus f_1$ is the (good) group $\Z$, a fact which makes the algebraic calculations tractable, as well as Freedman's disk embedding theorem hold. 

In order to shrink $f_2$ in the complement of $f_1$ using our Whitney homotopy finitely many times, it suffices to write $[f_2]\in \pi_2(S^4 \smallsetminus f_1)$ as a $\Z[\Z]$-linear combination of disjoint Whitney spheres $S_{W_i}$. In Section~\ref{sec:unlinking-theorems}
we'll use Kirby calculus  to get a complete picture of the 4-manifold $S^4 \smallsetminus f_1$, showing in particular that its intersection form is metabolic. 

Our Metabolic Unlinking Theorem~\ref{thm:I2-metabolic} gives necessary and sufficient conditions for $(f_1,f_2)$ to be link homotopically trivial in terms of a {\em metabolic collection} of immersed disks in $S^4 \smallsetminus f_1$, including a ``Lagrangian'' consisting of immersed Whitney disks.

To prove this Unlinking Theorem with our link homotopies shrinking $S_{W_i}$, we need to find sufficiently many disjoint Whitney disks $W_i$ that are framed and embedded. Freedman's disk embedding theorem \cite[\S 5]{FQ} requires a collection of algebraic dual spheres in addition to immersed Whitney disks. Our terminology is set up in a way that this means that one needs a metabolic collection of (Whitney and accessory) disks in $S^4 \smallsetminus f_1$.

In fact, each dual sphere will be constructed out of four copies of an accessory disk as in \cite[\S 3.1]{FQ} - we shall refer to such spheres as {\em accessory spheres}. 
If the Kirk invariants of $(f_1,f_2)$ vanish, the new geometric constructions described in the proof of Proposition~\ref{prop:I-squared} locate such a metabolic collection of disks in $S^4 \smallsetminus f_1$, completing our proof.


\hspace{2mm}

\noindent
{\em Acknowledgements:}  It is a pleasure to thank Tim Cochran and Mike Freedman for valuable discussions on the subject. The second author thanks Colin Rourke for pointing out a gap in his original manuscript. The first author was supported by a Simons Foundation \emph{Collaboration Grant for Mathematicians}, and both authors thank the Max Planck Institute for Mathematics in Bonn, where most of this work was carried out. We also thank the referee for asking about the geometric meaning of the $\Z[z]$-module structure on $\LM_{2,2}^4$, implicit in Theorem~\ref{thm:image}, which lead us to our ultimate formulation in Theorem~\ref{thm:free}.

\tableofcontents


\section{Whitney spheres and Whitney homotopies} \label{sec:w-spheres-and-homotopy}
This section introduces the notion of a \emph{Whitney homotopy}, which will play a key role in our ability to derive geometric conclusions from algebraic data. 
A Whitney homotopy is supported near a framed embedded Whitney disk $W$ on a generic immersion $f:S^2\imra X^4$, and involves shrinking an embedded \emph{Whitney sphere} $S_W$ to a point after performing a Whitney move on $W$, with $S_W$ at all times remaining disjoint from $f$. Since the Whitney disks for our Whitney homotopies will be constructed via Freedman's disk embedding machinery, we work in the (locally flat) topological category throughout this section, invoking the notions of 4-dimensional topological tranversality from \cite[chap.9]{FQ}.

 \begin{figure}[ht!]
         \centerline{\includegraphics[scale=.275]{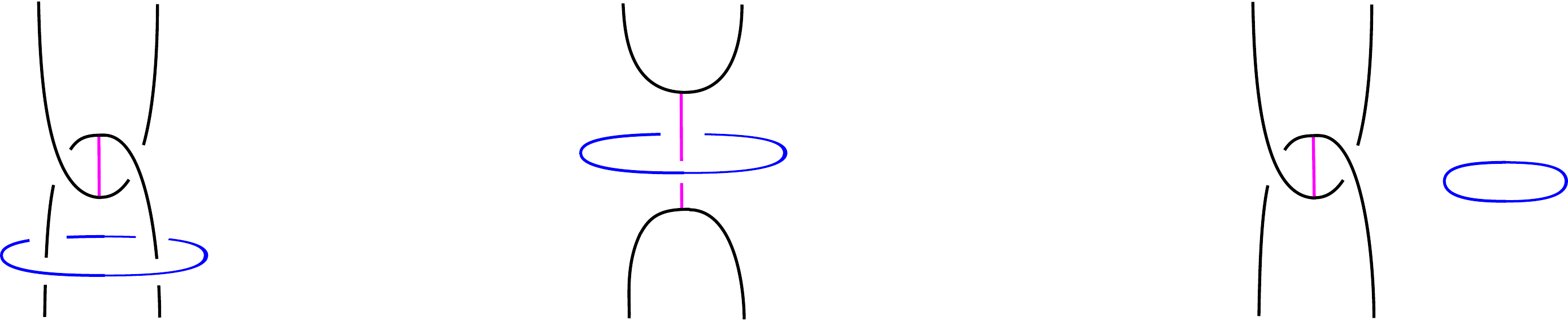}}
         \caption{A $3$-dimensional analogue of the Whitney homotopy.}
         \label{fig:1dim-whitney-homotopy}
\end{figure}

\subsection{A $3$-dimensional elementary homotopy}\label{sec:whitney-htpy-motivation}
As motivation towards a description of the Whitney homotopy of 2-spheres in dimension four, start by considering the analogous homotopy of $1$--manifolds in $3$--space shown in Figure~\ref{fig:1dim-whitney-homotopy}.
The sequence of pictures in Figure~\ref{fig:1dim-whitney-homotopy} describes a $3$-dimensional link homotopy in which the black $1$--manifold comes back to
its original position while the blue circle shrinks to a small unknot. This homotopy keeps black and blue disjoint throughout, and has exactly two singularities, namely the black-black intersections from the un-clasping and re-clasping (guided by the purple arcs).
 \begin{figure}[ht!]
         \centerline{\includegraphics[scale=.275]{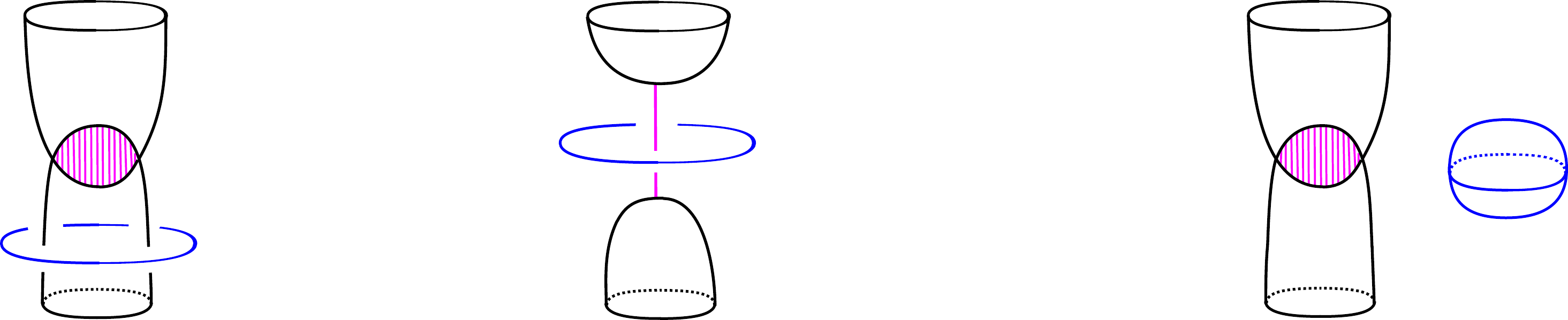}}
         \caption{A schematic illustration of the $4$-dimensional Whitney homotopy.}
         \label{fig:2dim-whitney-homotopy}
\end{figure}

The $4$-dimensional analogue of the link homotopy in Figure~\ref{fig:1dim-whitney-homotopy} is the composition of a Whitney move (``un-clasping'') and its reverse finger move (``re-clasping''). The Whitney move is guided by its Whitney disk $W$ whereas the finger move is again
guided by an arc. Thus the analogue of the blue circle in Figure~\ref{fig:1dim-whitney-homotopy} is an embedded 2--sphere $S_W$ which is the boundary of a 3--ball, normal to
the arc guiding the finger move. We call $S_W$ the {\em Whitney sphere}. Before starting a detailed description of this $4$-dimensional Whitney homotopy, it may be helpful to consider the schematic description shown in Figure~\ref{fig:2dim-whitney-homotopy}.

\subsection{Whitney moves and Whitney bubbles}\label{sec:embedded-whitney-move}
\begin{figure}[ht!]
         \centerline{\includegraphics[scale=.4]{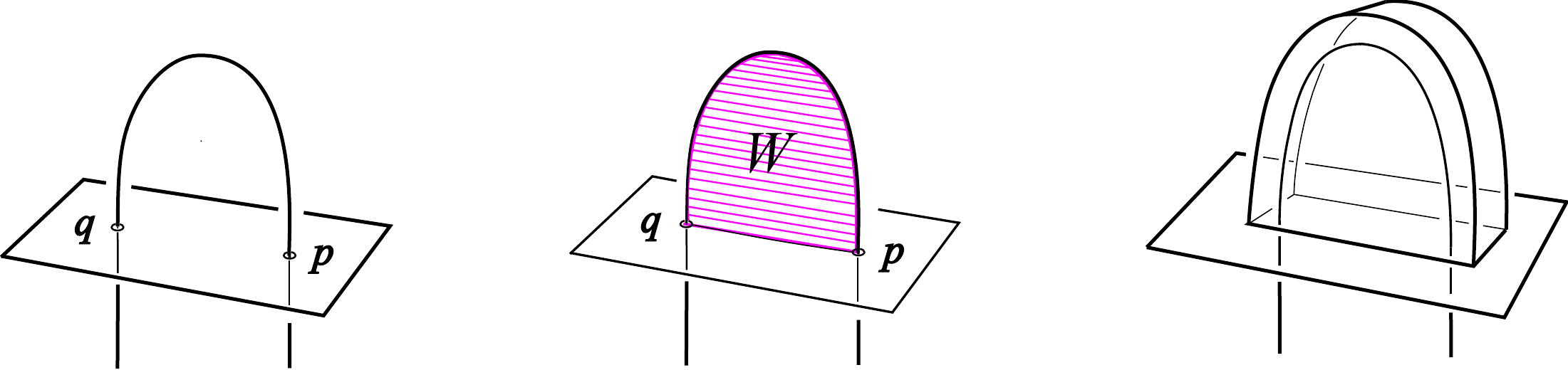}}
         \caption{The pair of intersections $p,q$ (left) admits a purple Whitney disk $W$ (center) which guides a Whitney move eliminating $p,q$ by adding a \emph{Whitney bubble} to the horizontal sheet (right).}
         \label{fig:W-disk-and-W-move-color}
\end{figure}

Figure~\ref{fig:W-disk-and-W-move-color} describes a model \emph{Whitney move} in a $4$--ball neighborhood of a \emph{Whitney disk} $W$: A pair of oppositely-signed transverse intersections $p$ and $q$ between two surface sheets is eliminated by a homotopy which isotopes one of the sheets across $W$. By convention, we write our 4-ball $D^4=D^3 \times D^1$ and think of the $D^1$-coordinate as {\it time}. Each of the three pictures in Figure~\ref{fig:W-disk-and-W-move-color} completely describes certain surfaces in this $D^4$: The horizontal sheet is completely contained in the present $D^3 \times 0$, while the curved arc extends into past $D^3 \times [-1,0)$ and future $D^3\times (0,1]$ (in the center picture, the Whitney disk $W$ is also in the present). In Figure~\ref{fig:W-disk-and-W-move-color} these extensions into past and future are understood but not shown, however 
in many cases our figures
will explicitly show several pictures of varying time-coordinates to describe surfaces via a ``movie'' (e.g. Figure~\ref{fig:whitehead-null-htpy} in the introduction and
Figures~\ref{fig:Whitney-sphere-for-elementary-homotopy}~and~\ref{fig:Whitney-sphere-and-ball-null-homotopy} just below).

The Whitney move has an additional homotopy parameter $s\in I$ which satisfies $s=0$ in the center picture of Figure~\ref{fig:W-disk-and-W-move-color} and $s=1$ in the right picture. For $0<s<1$, part of the horizontal sheet near the Whitney disk boundary moves upwards along $W$ until it reaches the ``bubble position'' on the right where it is disjoint from the other sheet. 
This \emph{Whitney bubble} has the same boundary as the original horizontal sheet and consists of two parallel copies of $W$ joined by a narrow curved rectangle shown in the right picture of Figure~\ref{fig:W-disk-and-W-move-color}.

In general, a Whitney move is described by any embedding of this model into a $4$--manifold which preserves the product structures and transversality of the sheets and $W$.
All Whitney disks in this section are embedded, and the reader may notice that they are also \emph{framed} which corresponds to the disjointness of parallel copies, as used in the Whitney bubble. 
Details on Whitney disks and Whitney moves are given in Section~\ref{sec:appendix}.

\subsection{Whitney spheres}\label{sec:embedded-w-spheres}
A \emph{Whitney sphere} $S_W$ is formed from a Whitney disk $W$ by joining together the boundaries of two parallel copies of Whitney bubbles by an annulus as illustrated in Figure~\ref{fig:Whitney-sphere-for-elementary-homotopy}. 
This annulus consists of an interval's worth of normal parallel copies of the boundary circle of a disk neighborhood around an arc of $\partial W$ that would be exchanged for a Whitney bubble under a Whitney move along $W$ (in Figure~\ref{fig:Whitney-sphere-for-elementary-homotopy} these circles are rectangles).
By construction $S_W$ is contained in an arbitrarily small $4$--ball neighborhood of $W$, and $S_W$ contains four parallel copies of $W$ (two from each Whitney bubble).  
Note that in our movie convention, $S_W$ extends into past (left) and future (right) in Figure~\ref{fig:Whitney-sphere-for-elementary-homotopy}, while $W$ and the horizontal sheet are contained in the present (center).        

\begin{figure}[ht!]
         \centerline{\includegraphics[scale=.35]{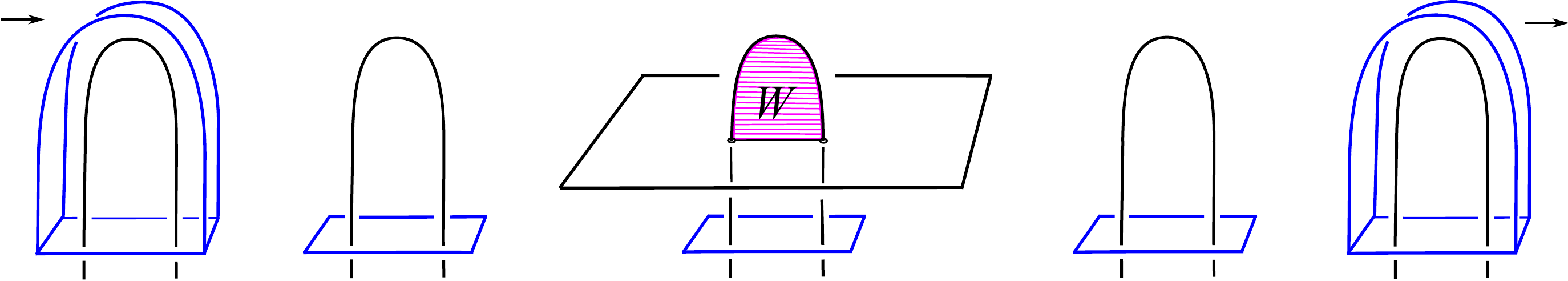}}
         \caption{A movie describes the blue \emph{Whitney sphere} $S_W$ associated to $W$.}
         \label{fig:Whitney-sphere-for-elementary-homotopy}
\end{figure}
         
Note that the $4$--ball described by the entire movie of Figure~\ref{fig:Whitney-sphere-for-elementary-homotopy} is the accurate description of just the left-most schematic picture in Figure~\ref{fig:2dim-whitney-homotopy}.

\subsection{The Whitney homotopy}\label{sec:the-whitney-htpy}

Let $W$ be a framed embedded Whitney disk on a generic map $f:F^2\imra X^4$, such that the interior of $W$ is disjoint from the image of $f$. Then the Whitney sphere $S_W$ from the previous section is contained in a regular neighborhood of $W$ and is embedded
in $X^4 \smallsetminus \im(f)$. Moreover, the link map $(f,S_W)$ is link homotopic to $(f,*)$ by a link homotopy which is supported in any  neighborhood of $W$ as described in Figure~\ref{fig:Whitney-sphere-and-ball-null-homotopy}. It shows that after a $W$-move on $f$ there is a $3$--ball in the complement of $f$ which is bounded by $S_W$. In our coordinates $D^3 \times D^1$, this 3--ball is simply the Whitney bubble cross time $D^1$. 

We remark that, in direct analogy with the $3$-dimensional link homotopy of Figure~\ref{fig:1dim-whitney-homotopy},
after ``unclasping'' $f$ by the Whitney move and shrinking $S_W$ to a point in $X^4\setminus f$ one could ``reclasp'' $f$ back to its original position via a finger move guided by the purple arc in the middle picture of Figure~\ref{fig:Whitney-sphere-and-ball-null-homotopy}. 

\begin{figure}[ht!]
         \centerline{\includegraphics[scale=.35]{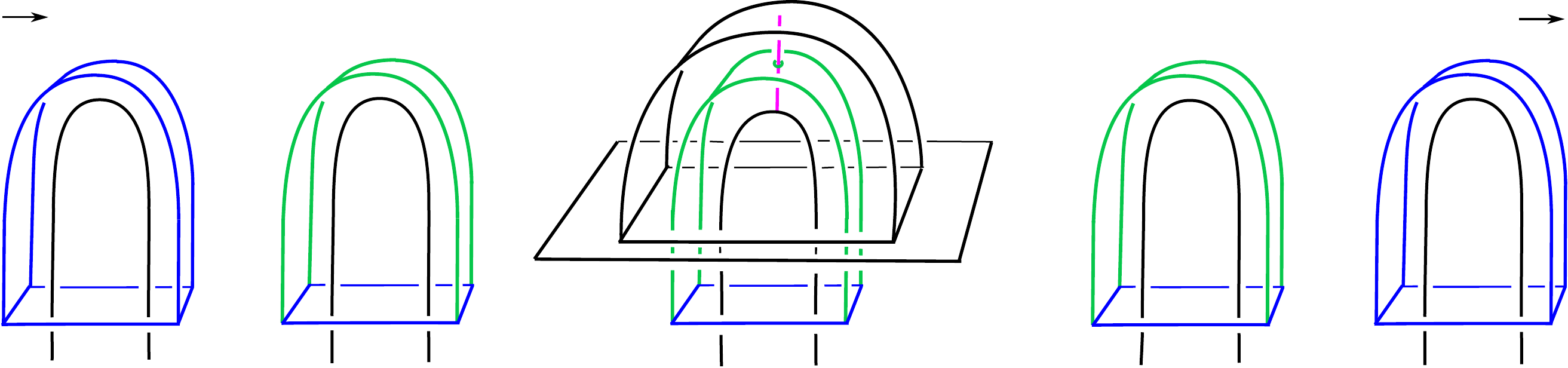}}
         \caption{After the Whitney move, the Whitney sphere $S_W$ from Figure~\ref{fig:Whitney-sphere-for-elementary-homotopy} bounds the green $3$--ball given by the Whitney bubble $ \times\ D^1$.}
         \label{fig:Whitney-sphere-and-ball-null-homotopy}
\end{figure}

If ${\#}S_W$ denotes a connected sum of parallel copies of $S_W$ along arbitrary tubes, then $(f,{\#}S_W)$ is also link homotopic to $(f,*)$: There are $k$ parallels of the green $3$--balls in the complement of $f$ that can be connected by arbitrary arcs, thickened by $D^2$ normal subbundles. The boundary of the resulting 3--ball is by construction the tube sum ${\#}S_W$.

More generally, if $W_i$ are disjointly embedded framed Whitney disks then any $\Z\pi_1X$-linear combination $C:=\sum_i \lambda_i S_{W_i}$, i.e. iterated tubing of copies of the Whitney spheres, leads to a link homotopy of $(f,C)$ with $(f,*)$. The disjointness of $W_i$ is needed since we're doing Whitney moves on $f$ as part of the Whitney homotopy and these have to be disjoint from the Whitney spheres in a link homotopy.

\section{Link maps in standard position}\label{sec:standard-position}
Let $(f_1,f_2): S^2 \amalg S^2 \to S^4$ be a link map. After a small perturbation (which is a link homotopy) we may assume that both $f_i$ are immersions with transverse double points. Furthermore, we may perform local cusp homotopies (again keeping
$f_i $ disjoint) to get the signed sum of the self-intersection points of each $f_i$ to be zero.

{\em From now on, we will always assume that our maps $f:S^2\imra X^4$ are immersions with only transverse self-intersections whose signed sum is zero. For such  maps, regular homotopy agrees with homotopy because cusps can be cancelled in pairs}.

It follows that Wall's self-intersection invariant $\mu(f) \in \Z[\pi_1X]/ \langle g - g^{-1} \rangle$ always satisfies $\e(\mu(f))=0$ where $\e:\Z[\pi_1X]\to\Z$ is the augmentation map which sends all group elements $g$ to $1$ (and hence $\e(\mu(f))$ counts all double points with sign, ignoring group elements). This normalization makes $\mu$ into a homotopy (rather than regular homotopy) invariant which comes about as follows: In general position, a regular homotopy between two maps as above, is the composition of finitely many finger moves and then finitely many Whitney moves (up to isotopy). 

\begin{rem}\label{rem:mu}
If $X=S^4$ then $\pi_2(S^4)=0$ implies that we may perform finger moves on $f_1$ (disjoint from $f_2$ by general position) to get $f$ for
which some Whitney moves lead to the standard unknotted embedding $f_0:S^2\subset S^4$. Conversely, $f$ is obtained from the unknot $f_0$ by finitely many finger moves. 

Moreover, Wall's formula mentioned in the introduction reduces to
\[
\lambda(f_1,f_1) = \mu(f_1) +\iota\mu(f_1)
\]
for $f_1:S^2\imra S^4\smallsetminus f_2$ because $\pi_2(S^4)=0$ also implies $\e(\lambda(f_1,f_1))=0$ and hence the term involving the normal Euler number vanishes. In fact, this Euler number equals the signed sum of double points for generic maps in $S^4$.
\end{rem}

\begin{defi}\label{def:standard position}
A map $f:S^2\imra S^4$ which is obtained from the {\em unknot} $f_0$ by finitely many finger moves is said to be in {\em standard
position}. Here $f_0$ is an unknotted embedding $f_0:S^2\subset \R^3\subset S^4$. Moreover, any collection of
disjointly embedded Whitney disks for $f$ which lead back to $f_0$ is called a {\em standard collection} of Whitney disks.
\end{defi}
The above discussion proves the following:
\begin{lem}\label{lem:standard}
Any link map $(f_1,f_2): S^2 \amalg S^2 \to S^4$ is link homotopic to one in which both components are in standard position. In particular, the complements $S^4 \smallsetminus f_i$ have all algebraic properties listed in Lemma~\ref{lem:algtop} below.
\end{lem}

\subsection{Some algebraic topology}
We'll next discuss the algebraic topology of the complement of a map $f:S^2\imra S^4$ in standard position.
Along with the Whitney spheres from Section~\ref{sec:w-spheres-and-homotopy}, the following lemma involves \emph{accessory spheres},
which are introduced in the proof and described in further detail in Section~\ref{sec:accessory-sphere-revisit}. Just as a Whitney sphere is constructed from four copies of a Whitney disk, an accessory sphere is constructed from four copies of an \emph{accessory disk} bounded by a sheet-changing loop in $f$ which passes through a self-intersection point.

\begin{lem} \label{lem:algtop}
Let $W_1,\dots, W_n$ be a standard collection of Whitney disks for $f:S^2\imra S^4$ in standard position. Then
the compact $4$--manifold
$$
M:= S^4 \smallsetminus \text {tubular neighborhood of $f$}
$$
has the following algebraic topological properties:
\begin{enumerate}
\item\label{item:pi1-Z} 
$\pi_1M$ is infinite cyclic, generated by a meridian $x$. Set $\Lambda:=\Z[x^{\pm 1}]$.
\item\label{item:pi2-free-rank-2n}  
$\pi_2M$ is a free $\Lambda$-module of rank $2n$ with a basis of the form
$$
\{S_{W_1},\dots,S_{W_n},S_{A_1},\dots, S_{A_n}\}
$$
where $S_{W_i}$ are the Whitney spheres corresponding to the Whitney disks $W_i$ and the $S_{A_i}$ are {\em accessory spheres} corresponding to positive {\em accessory disks} $A_i$. 
\item\label{item:pi2-metabolic}  
Wall's intersection form
$\lambda:\pi_2M \times\pi_2M \to \Lambda$
is {\em metabolic} in the sense that
\[
\lambda(S_{W_i},S_{W_j})=0  \quad \text{ and } \quad
\lambda(S_{A_i},S_{A_j})=z\cdot \delta_{ij}=\lambda(S_{W_i},S_{A_j}). 
\]
\item\label{item:W-SES}  
Relative intersection numbers with the Whitney disks $W_i$ induce a short exact sequence of free
$\Lambda$-modules
$$
\begin{diagram}
\node{0} \arrow{e} \node{\langle S_{W_i} \rangle}
 \arrow{e} \node{\pi_2M} \arrow{e,t}{ \lambda(W_i,-)} \node{\Lambda^n}
\arrow{e} \node{0}
\end{diagram}
$$
\end{enumerate}
\end{lem}
Here $z:=2-x-x^{-1}$, as in the introduction.
\begin{proof}
By assumption, $f$ is obtained from $f_0$ by $n$ finger moves which are the inverses of the Whitney moves along $W_i$.
The isotopy class of any finger move only depends on the isotopy class of its
guiding arc which is unchanged by adding a meridian to the sphere.  As a consequence, all fingers moves on the unknot $f_0$ are isotopic. Using standard Kirby calculus techniques (e.g.~Chapter 6 of \cite{GS}, in particular \cite[Fig. 6.27]{GS}) a handle diagram for $M$ can be derived as in the left-hand side of 
Figure~\ref{standard-sphere-Kirby-diagram-3}. The dotted circle represents a 1-handle, and together with a 0-handle it forms the 4-manifold $S^1 \times D^3$ which is the complement of $f_0$. Each finger move then adds a pair of (green) 2-handles, so $M$ contains no 3- or 4-handles. 
\begin{figure}[ht!]
         \centerline{\includegraphics[scale=.25]{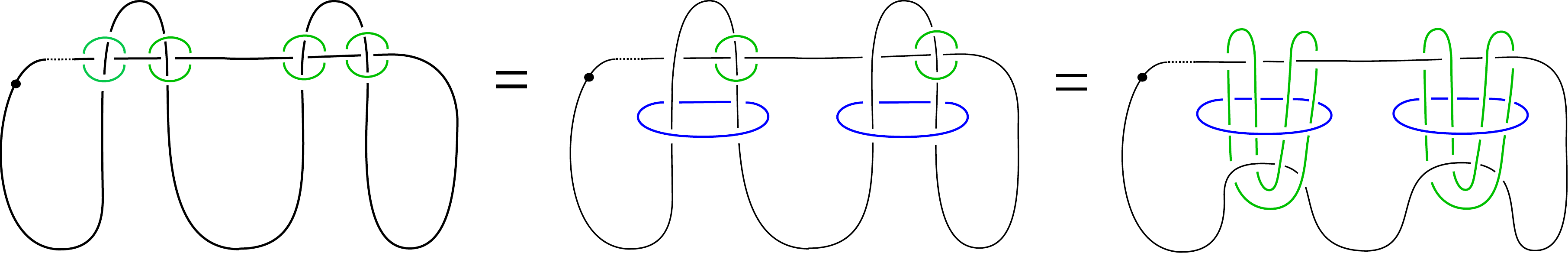}}
         \caption{Kirby diagrams for  $M$: The dotted black unknot represents a $1$-handle, and $0$-framed $2$-handles are attached to green and blue circles.}
         \label{standard-sphere-Kirby-diagram-3}
\end{figure}

Since the $2$-handles' attaching circles are null homotopic, it follows that
\[
M \simeq S^1 \vee \bigvee^{2n} S^2.
\]
So $\pi_1M = \langle x \rangle \cong \Z$, where $x$ is a meridian to the 1-handle, and $\pi_2M$ is a free $\Lambda$-module of rank $2n$.
Sliding one of each pair of green 2-handles in the left-hand diagram of Figure~\ref{standard-sphere-Kirby-diagram-3} over the other yields the blue 2-handle attaching circles in the center diagram. After an isotopy to get the right-hand diagram, the unions of cores of the $2$-handles attached to the blue circles with the evident disks bounded by these attaching circles form $n$ generators $S_{W_i}$, which will be shown in Section~\ref{sec:compare-w-spheres} to be Whitney spheres as constructed in Section~\ref{sec:embedded-w-spheres}.

Figure~\ref{standard-sphere-Kirby-diagram-4} shows an isotopy of Kirby diagrams, starting from the left side of Figure~\ref{standard-sphere-Kirby-diagram-3}. The diagram on the right gives the construction of positive and negative \emph{accessory spheres} $S_{A^\pm_i}$: They are formed from cores of the $2$-handles together with null-homotopies of their green attaching circles. The evident $\mp$-crossing change in each green circle corresponds to a $\mp$-self-intersection of $S_{A_i^\pm}$ whose associated group element is the generator $x$.
For Lemma~\ref{lem:algtop} we only use the $n$ positive accessory spheres $S_{A_i}=S_{A^+_i}$.
\begin{figure}[ht!]
         \centerline{\includegraphics[scale=.25]{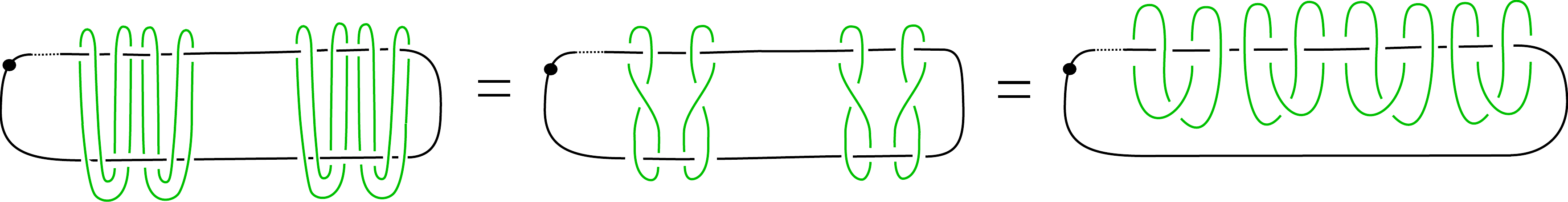}}
         \caption{Kirby diagrams isotopic to the left picture of Figure~\ref{standard-sphere-Kirby-diagram-3}, with the accessory spheres $S_{A_i^\pm}$ visible on the right.}
         \label{standard-sphere-Kirby-diagram-4}
\end{figure}

By construction it is clear that the $S_{W_i}$ are disjointly embedded in $M$, and that their intersection numbers with $S_{A_j}$ are
as claimed: These spheres are disjoint for $i\neq j$, and for $i=j$ they intersect in precisely four points (two oppositely-signed pairs) as seen in the right-most picture of Figure~\ref{standard-sphere-Kirby-diagram-3}.
Choosing orientations and basings appropriately, one can check from Figure~\ref{standard-sphere-Kirby-diagram-3} that $\lambda(S_{W_i},S_{A_i})=2-x-x^{-1}=z$. 	

Again by construction, the spheres $S_{A_i}$ are disjoint from each other and each have one negatively-signed self-intersection with a meridional double point loop. By our convention of making the signed sum of double points zero, we add another positively-signed self-intersection with trivial group element. The self-intersection invariant $\lambda(S_{A_i},S_{A_i})$ is computed from intersections between $S_{A_i}$ and a pushoff, and it
follows that it equals to $2-x-x^{-1}=z$.

To prove statement~\ref{item:pi2-metabolic} note that the Whitney disks $W_i$ are disjoint from
$S_{W_j}$ and have geometric intersections $\delta_{ij}$ with $S_{A_j}$. This can be read off from the center picture of Figure~\ref{standard-sphere-Kirby-diagram-3} (see also Section~\ref{sec:accessory-sphere-revisit}), and implies statement~\ref{item:W-SES}.
\end{proof}

For future reference, we note that the right-most picture of Figure~\ref{standard-sphere-Kirby-diagram-4}
implies the following intersection information about accessory
spheres:
\begin{lem}\label{lem:accessory-sphere-basis}
The set of positive and negative accessory spheres $S_{A^\pm_i}$ forms a basis of disjointly immersed 2-spheres for $\pi_2M$ with
$\lambda(S_{A^+_i},S_{A^+_j})= z\cdot \delta_{ij}=-\lambda(S_{A^-_i},S_{A^-_j})$. \hfill$\square$
\end{lem}

\subsection{Kirk's invariant in standard position}\label{subsec:Kirk-invariant-in-standard-position}
Our descriptions of the complement of a standard position sphere provide the following convenient computation of Kirk's invariants. 

Suppose $(f_1,f_2): S^2 \amalg S^2 \to S^4$ is a link map with $f_1$ in standard position. Then  $f_2$ is homotopic in $S^4\smallsetminus f_1$ to $\sum_i \alpha_i^+\cdot S_{A_i^+}+\alpha_i^-\cdot S_{A_i^-}$ for accessory spheres $S_{A_i^\pm}$ as in Lemma~\ref{lem:accessory-sphere-basis} and
$\alpha_i^+,\alpha_i^-\in\Z[x^{\pm 1}]$.
It follows that 
\[
\sigma_2(f_1,f_2)=\lambda(f_2,f_2)=\sum_i(\alpha^+_i\cdot \iota \alpha^+_i-\alpha^-_i\cdot \iota \alpha^-_i)\cdot z,
\]
 for $\iota$ the  involution on the group ring determined by $\iota(x)=x^{-1}$. 
 
 To compute $\sigma_1(f_1,f_2)$ we first show that Wall's self-intersection invariant satisfies
 \[
 \mu(f_1) = \sum_i x^{\epsilon(\alpha^+_i)} - x^{\epsilon(\alpha^-_i)}.
 \]
This sum contains one contribution for each self-intersection of $f_1$ with the correct sign and it remains to show that the linking number of the corresponding double point loop with $f_2$ in $S^4$  is given by $\epsilon(\alpha^\pm_i)$ for $\epsilon:\Z[x^{\pm 1}]\to\Z$ the augmentation map:
Examination of the Kirby diagram in the left picture of Figure~\ref{standard-sphere-Kirby-diagram-3} shows that the $i$th $\pm$-crossing of the black dotted circle corresponds to the $i$th $\pm$-self-intersection of $f_1$, and this self-intersection has a double point loop that links (in $S^3$) the green attaching circle of its $\pm$-accessory sphere once, and has zero linking with all other green circles. 
Moreover, this linking in $S^3$ equals the linking of the double point loop with the accessory spheres in $S^4$ by construction. To compute linking with $f_2$ we need to sum over the algebraic count of parallel copies of $S_{A_i^\pm}$ in $f_2$ which is given by $\epsilon(\alpha^\pm_i)$. 

As a consequence of Wall's formula we finally get
\[
\sigma_1(f_1,f_2)=\lambda(f_1,f_1)=\sum_i (x^{\epsilon(\alpha^+_i)}+ x^{-\epsilon(\alpha^+_i)})-(x^{\epsilon(\alpha^-_i)}+ x^{-\epsilon(\alpha^-_i)}).
\]

\begin{rem}

The Fenn--Rolfsen link map $\FR$ is the case where $f_1$ is the trivial sphere with one self-fingermove and $f_2=S_{A^+}$.
To see this, just compare Figure~\ref{fig:whitehead-null-htpy} with the center diagram of Figure~\ref{standard-sphere-Kirby-diagram-4}.
As already computed in the introduction, we get $\sigma_2(\FR)=z=-\sigma_1(\FR)$. 
\end{rem}

\subsection{The boundary of the standard position complement}\label{subsec:preview-uniqueness-of-SW}

It is also instructive to look at the homotopy exact sequence of pairs
\[
\pi_2(\partial M) \ra \pi_2(M) \ra \pi_2(M,\partial M) \ra \pi_1(\partial M) \sra \Z.
\]
If $f$ is an embedding then $\partial M\cong S^1 \times S^2$ has non-trivial $\pi_2$. Otherwise, $\partial M$ is prime, because $\pi_1M$ is infinite and not a free product, and hence $\pi_2(\partial M)$ vanishes. (A framed link description of $\partial M$ can be gotten by replacing the dots on circles by $0$-framings in any of the Kirby diagrams of Figures~\ref{standard-sphere-Kirby-diagram-3} and~\ref{standard-sphere-Kirby-diagram-4}.)

In the presence of a Whitney disk $W$ for $f$ we are in this latter case and hence the homotopy class of a Whitney sphere
$S_W$ is determined by its homotopy class in the relative group $\pi_2(M,\partial M)$. But this group contains as
an element $[W]$ represented by the Whitney disk (with a small collar removed so that its boundary lies in $\partial M$).
In this case we get a monomorphism $\pi_2(M) \ira \pi_2(M,\partial M)$ which sends $[S_{W_i}]$ to $z\cdot [W_i]$, as can be seen in the right picture of Figure~\ref{standard-sphere-Kirby-diagram-3} and 
will be proven in detail in Lemma~\ref{lem:W-sphere-homotopy-def} (even for Whitney spheres formed from immersed Whitney disks as defined in Section~\ref{sec:cleanly-immersed-w-moves}). It explains the repeated occurrence of the factor $z$ throughout our paper.

\begin{rem}\label{rem:pi2-standard-boundary-trivial}
If $X^4$ is an orientable $4$-manifold (without boundary) and $g:S^2\imra X^4$ is any generic immersion with $n>0$ cancelling pairs of self-intersections, then
one can similarly conclude that the boundary of the complement of a tubular neighborhood of $g$ has trivial $\pi_2$. In fact, a framed link description of this boundary 
only differs from the description 
of $\partial M$ as above in that the dotted circle will get a framing equal to the normal Euler number of $g$ in $X$.
\end{rem}

\section{Computations in $\LM^4_{2,2}$}

In this section, we'll derive Theorems~\ref{thm:free} and~\ref{thm:image} 
from our main result, Theorem~\ref{thm:main}, using elementary algebraic arguments along with the algebraic topological properties of standard position link maps described in Section~\ref{sec:standard-position}. 

\subsection{The Kirk invariants give a monomorphism}\label{sec:mono}
Assuming that the Kirk invariants detect the unlink, the following result follows directly from the additivity of the Kirk invariants under connected sum and the fact that it changes sign under reflection of $S^4$. 

Alternatively, one could use Koschorke \cite[Prop.2.3]{Ko}, who showed that the ambient connected sum operation is well defined and makes $\LM_{p,q}^n$ into an abelian monoid for $p,q\leq n-2$.
Moreover, the usual rotation argument and the fact that ``link concordance implies link homotopy'' shows that inverses exist in all these cases. The latter was  proved by Bartels and the second author for $n=4$ in \cite{BT}.

\begin{cor}\label{cor:group}
For a link map $(f_1,f_2): S^2 \amalg S^2 \to S^4$, let $-(f_1,f_2)$ denote its composition with a reflection of $S^4$. Then the link map $(f_1,f_2) \# -(f_1,f_2)$ is homotopically trivial and $\LM_{2,2}^4$ becomes an abelian group under connected sum. Moreover, the Kirk invariant
$(\sigma_1, \sigma_2): \LM^4_{2,2} \hra \Z[\Z]\oplus \Z[\Z]$ is a group monomorphism.$\hfill\square$
\end{cor}

\subsection{The image of the Kirk invariants: Proof of Theorem~\ref{thm:image}}\label{sec:proof-thm-image}

To describe the image of the Kirk invariants, fix a (multiplicative) generator $x$ of $\Z$ so that $\Z[\Z] = \Z[x^{\pm 1}]$ is the ring of Laurent polynomials. 
Recall that $z = (1-x)\cdot (1-x^{-1}) = (1-x) + (1-x^{-1}).$
\begin{lem} \label{lem:image} 
For $I$ the augmentation ideal in $\Z[\Z]$ and $\iota$ the usual involution on the group ring determined by $\iota(x)=x^{-1}$, we have
\[
I^k\cap \Z[\Z]^\iota = z^k \cdot \Z[z] \quad \forall \, k\in \N.
\]
\end{lem}
The left hand side of the equation above consists of certain Laurent polynomials that are invariant under $\iota$. Since Wall's intersection invariant is hermitian, it follows that $\lambda(f,f)\in \Z[\pi]^\iota$ for any fundamental group $\pi$. Moreover, the augmentation ideal $I$ is the kernel of the augmentation map $\Z[\pi]\to\Z$ and $\lambda(f,f)$ is carried to the self-intersection of homology classes. Since $H_2(S^4)=0$, it follows that $\lambda(f_i,f_i)\in I\cap \Z[\Z]^\iota = z \cdot \Z[z]$.

\begin{proof}[Proof of Lemma~\ref{lem:image}]
Since $\iota(z) = z \in I$, one inclusion is clear. For the other, start with the case $k=0$.  If $P \in \Z[x^{\pm 1}]$ is invariant under $\iota$ then it is symmetric around $x^0$ and hence we can subtract an integer multiple of $z^n = (-1)^n(x^n + \dots + x^{-n})$ to reduce the degree of $P$. By induction on $n$ we see that $P\in \Z[z]$.

Now take $P \in I^k = (1-x)^k\cdot \Z[x^{\pm 1}]$ and assume $ P=\iota P$. Then there is a $Q\in \Z[x^{\pm 1}]$ such that
\[
(1-x)^k \cdot Q = P = \iota P = (1-x^{-1} )^k \cdot \iota Q.
\]
Since the group ring is a unique factorization domain, it follows that $P$ is divisible by $z^k = (1-x)^k\cdot (1-x^{-1} )^k$ and that the quotient is again invariant under $\iota$. The case $k=0$ now shows that $P \in z^k \cdot \Z[z]$.
\end{proof}
\begin{lem}\label{lem:symmetry}
For any link map $(f_1,f_2): S^2 \amalg S^2 \to S^4$, the Kirk invariants satisfy the antisymmetry relation $\sigma_1(f_1,f_2) + \sigma_2(f_1,f_2) \in z^2\cdot \Z[z]$. 
\end{lem}
\begin{proof}
As in Lemma~\ref{lem:standard}, we may assume that $(f_1,f_2)$ is in general position and that $f_1$ is standard. By additivity of the Kirk invariants, it suffices to consider the case where $f_1$ is obtained from the unknot by one finger move and that $f_2 = \alpha\cdot S_A$ is a multiple of a single accessory sphere $S_A$, for some element $\alpha\in \Z[x^{\pm 1}]$ as in Lemma~\ref{lem:accessory-sphere-basis}.
 Assuming that $S_A$ is a positive accessory sphere, from Section~\ref{subsec:Kirk-invariant-in-standard-position} we have: 
\[
\sigma_1(f_1,f_2) = \lambda(f_1, f_1) = x_2^{\epsilon(\alpha)} + x_2^{-\epsilon(\alpha)}-2   \quad \text{ and } \quad \sigma_2(f_1,f_2) = \lambda(f_2, f_2) =\alpha \cdot \iota \alpha \cdot z,
\] 
with $\epsilon:\Z[x^{\pm 1}]\to\Z$ the augmentation map.
We want to show that after setting $x_2=x$ the sum of these two expressions lies in $z^2\cdot \Z[z]$. Let $k:=\epsilon(\alpha)$ and
$P_k := 1 + x + x^2 +\dots + x^{k-1}$.
Then 
\[
x^k + x^{-k} - 2 = -(1-x^k)\cdot (1-x^{-k}) = -(1-x) \cdot P_k \cdot  (1-x^{-1}) \cdot\iota P_k = -z\cdot P_k \cdot\iota P_k.
\]
Thus it suffices to show that for all $\alpha\in \Z[x^{\pm 1}]$, we have $\alpha \cdot \iota \alpha-P_k \cdot\iota P_k \in z\cdot \Z[z]$. Since this element is invariant under the involution $\iota$, it suffices to show by Lemma~\ref{lem:image} that it lies in the augmentation ideal $I$. But this is clearly true since $\epsilon(\alpha) = k = \epsilon(P_k)$.
\end{proof}
Recall the statement of Theorem~\ref{thm:image}, that the Kirk invariants $(\sigma_1, \sigma_2)$ fit into a short exact sequence of groups
\[
0 \ra \LM^4_{2,2} \ra  z\cdot \Z[z]\oplus z\cdot \Z[z] \ra \Z \ra 0,
\]
where we use the addition map to $\Z = z\cdot \Z[z]/z^2\cdot \Z[z]$.  
So
to compute the cokernel of the Kirk invariants, we are left with proving that every element in $z\cdot \Z[z]\oplus z\cdot \Z[z]$ satisfying the antisymmetry relation of Lemma~\ref{lem:symmetry} is realized as the Kirk invariant of a link map. 
\begin{proof} 
[Proof of Theorem~\ref{thm:image}]
Using linear combinations of the above example from the proof of Lemma~\ref{lem:symmetry} with 
$\sigma_1(f_1,f_2) = x_2^{\epsilon(\alpha)} + x_2^{-\epsilon(\alpha)}-2$, we see that the first component can be any element in $z\cdot \Z[z]$. Now we add link maps $(g_1,g_2)$ with $\sigma_1(g_1,g_2)=0$ to correct the second component. For $g_2 = \alpha\cdot S_A$, as in the above example, this just means that $\epsilon(\alpha)=0$, i.e.\ $\alpha = (1-x) \cdot \beta$. Taking sums again, we need to realize any element in $z^2\cdot \Z[z]$ by linear combinations of elements $\alpha\cdot \iota \alpha\cdot z = \beta\cdot \iota \beta \cdot z^2$.

This is possible by an induction on degree, since we can use summands with $\beta=1$ and $\beta=(1-x^k)$ for arbitrary $k$. 
\end{proof}

\begin{rem}
We leave it as a fun exercise for the reader to draw examples of classical links $L\in\mathcal L$ that realize these Kirk invariants in their $\JK$-construction, proving surjectivity in Corollary~\ref{cor:JK} again. As a hint, one can start with the Whitehead link and apply the algebraic ideas from the next section diretly in the 3-dimensional setting. 
\end{rem}

\subsection{The Fenn--Rolfsen link map generates freely: Proof of Theorem~\ref{thm:free}}\label{sec:thm-free-proof}
In the exact sequence of Theorem~\ref{thm:image}, the kernel $K$ of $z\cdot \Z[z]\oplus z\cdot \Z[z] \ra \Z$ is easily seen to be a free $\Z[z]$-module on the two generators $(-z,z)$ and $(-z^2,0)$. 

We embed $\Z[z]$ as a subring of $R=\Z[z_1,z_2]/(z_1z_2)$ from Theorem~\ref{thm:free} by sending $z$ to $z_1+z_2$. Then $R$ also becomes a $\Z[z]$-module and it's easy to check that $1, z_1$ are free generators. 
It follows that there is a $\Z[z]$-linear isomorphism $\rho: K \overset{\cong}{\ra} R$ sending $(-z,z)$ to $1$ and $(z^2,0)$ to  $-z_1$.
Composing the Kirk invariant $\sigma=(\sigma_1, \sigma_2)$ with this isomorphism leads to an isomorphism
\[
\rho\circ\sigma:\LM^4_{2,2} \overset{\cong}{\ra} K \overset{\cong}{\ra} R.
\]
By construction, the Fenn--Rolfsen link $\FR$ is sent to $\rho(\sigma(\FR))= \rho(-z,z)=1\in R$. We are left with constructing a geometric $R$-module structure on $\LM^4_{2,2}$ which makes this map $R$-linear. Note that evaluating this on the action of $z_1+z_2$ also gives a $\Z[z]$-module structure on $\LM^4_{2,2}$. 

The action of $z_1$ on a link map $(f_1,f_2): S^2 \amalg S^2 \imra S^4$ in standard position is defined as follows: Push off an oppositely oriented parallel copy $f'_1$ of $f_1$ and then tube the two copies together along an arc that follows the meridian to $f_2$. 
We claim that this gives a well defined element $z_1\cdot [f_1,f_2]$ in $\LM^4_{2,2}$: 
\begin{itemize}
\item The choice of push-off $f_1'$ is unique up to homotopy (in the complement of $f_2$).
\item Two arcs (with the same boundary) that follow a meridian to $f_2$ are homotopic in the complement of $f_2$ since $\pi_1(S^4 \smallsetminus f_2)\cong\Z$, generated by that meridian.
\item Adding a finger move to $f_1$ in the complement of $f_2$ keeps $[f_1,f_2]$ constant in $\LM^4_{2,2}$ but the $z_1$-action commutes with that operation (except that we need two parallel finger moves if we act by $z_1$ first). 
\item By general position, a finger move of $f_2$ in the complement of $f_1$ can still be realized after applying the $z_1$-action.
\end{itemize}
Recall that homotopic link maps can be stabilized by finger moves into a position where they are isotopic (both components in standard position). As a consequence, the above cases imply that $z_1\cdot [f_1,f_2]\in\LM^4_{2,2}$ is well defined. By sliding tubes over tubes, it is straight-forward to check that 
\[
z_1\cdot ([f_1,f_2] +[g_1,g_2])= z_1\cdot [f_1,f_2] + z_1 \cdot [g_1,g_2]
\]
which can also be derived from the computation of the Kirk invariants below.

We define the $z_2$-action symmetrically (by pushing off an oppositely oriented copy of $f_2$ and tubing the two copies together along a tube following the meridian to $f_1$). As the last step in getting an $R$-action, we need to verify that $z_1z_2$ acts trivially on $\LM^4_{2,2}$. Without a geometric argument in place, we'll use the injectivity of the Kirk invariants and check that they vanish after applying $z_1z_2$ to any link map: $\sigma(z_1z_2\cdot[f_1,f_2]) = 0$. This follows by symmetry from the relation $\sigma_2(z_1\cdot[f_1,f_2]) = 0$: Every double point loop of $f_2$ links the new first component of $z_1\cdot[f_1,f_2]$ algebraically zero times. 

Moreover, abbreviating by $x$ the meridian to $f_2$, we compute
\[
\sigma_1(z_1\cdot[f_1,f_2]) = \lambda((1-x)\cdot f_1 , (1-x)\cdot f_1) = (1-x)\cdot\lambda(f_1,f_1) \cdot(1-x^{-1}) = z \cdot \sigma_1(f_1,f_2).
\]
As a consequence, the Kirk invariant becomes $R$-linear if we define the $R$-action on $K$ by
\[
z_1 \cdot (\sigma_1, \sigma_2) := (z\cdot \sigma_1, 0) \text{ and } z_2 \cdot (\sigma_1, \sigma_2) := (0, z\cdot \sigma_2).
\]
Restricted to the subring $\Z[z]\subset R$ generated by $z\mapsto z_1+z_2$ we get the original action: 
\[
(z_1+z_2)\cdot (\sigma_1,\sigma_2) = z_1\cdot (\sigma_1,\sigma_2) + z_2\cdot (\sigma_1,\sigma_2)= (z\cdot \sigma_1,0)+(0,z\cdot\sigma_2)=(z\cdot \sigma_1,z\cdot\sigma_2).
\]
We are left with checking that our isomorphism $\rho:K\ra R$ is $R$-linear. By construction, $\rho$ is $\Z[z]$-linear, so we only need to check linearity under multiplication with $z_1\in R$:
\[
\rho(z_1\cdot(-z,z)) = \rho(-z^2, 0)=z_1 \cdot 1=z_1 \cdot \rho(-z,z) \text{ and}
\]
\[
\rho(z_1 \cdot (-z^2, 0)) = \rho(-z^3, 0)=z_1\cdot z_1 =z_1\cdot \rho(-z^2,0).
\]
Since the $\Z[z]$-module $K$ is free on the generators $(-z,z)$ and $(-z^2,0)$, these equations complete our proof. 
 $\hfill\square$

We record the identity
$(\sigma_1, \sigma_2)((z_1+z_2)\cdot[f_1,f_2]) = z \cdot (\sigma_1, \sigma_2)[f_1,f_2]$ which can be translated to a geometric $\Z[z]$-action on $\LM^4_{2,2}$. The generator $z$ acts by tubing 3 copies (algebraically one copy) of $f_1$ and 3 copies (algebraically one copy) of $f_2$ together in the way specified by our formulas.



 \section{Unlinking Theorems}\label{sec:unlinking-theorems}
 
Returning to the path towards our goal of proving Theorem~\ref{thm:main}, this section gives criteria for a link map to be trivial.
As a first step, the Standard Unlinking Theorem~\ref{thm:I2} is formulated and proven except for one crucial implication. Then we generalize our notions of Whitney spheres and accessory spheres to provide the input into the Metabolic Unlinking Theorem~\ref{thm:I2-metabolic}. It will be proven using Freedman's disk embedding theorem and it implies the Standard Unlinking Theorem.  In addition, it is strong enough to yield Theorem~\ref{thm:main} in combination with Proposition~\ref{prop:I-squared} which will be proved in Section~\ref{sec:thm-main-proof}.
 
\subsection{The Standard Unlinking Theorem}\label{sec:standard-unlinking} 

There is a simple sufficient condition for a link map $(f_1,f_2): S^2 \amalg S^2 \to S^4$ to be link homotopically trivial:
Assume that $f_1$ is in a standard position such that $f_2$ is disjoint from a standard collection of Whitney disks
$W_i$ for $f_1$. Then one can do the Whitney moves on $W_i$ in the complement of $f_2$ to get a link homotopy from
$(f_1,f_2)$ to $(f_0,f_2)$. But this link map is trivial since
$$
\pi_2(S^4 \smallsetminus f_0 ) \cong \pi_2(S^1 \times D^3)=0.
$$
The goal of this section is to give algebraic conditions that model the disjointness of $f_2$ and the $W_i$, and as a first step we have the following result:
\begin{thm}[Standard Unlinking Theorem]\label{thm:I2}
The following statements are equivalent:
\begin{enumerate}
\item[(i)] The link map $(f_1,f_2)$ is link homotopically trivial.
\item[(ii)] There are finger moves on $f_1$ (disjoint from $f_2$) which bring $f_1$
into a standard position $f$ with a standard collection of Whitney disks $W_i$ for $f$ such that 
\[
\lambda(W_i,f_2)=0.
\]
 \item[(iii)] There are finger moves on $f_1$ (disjoint from $f_2$) which bring $f_1$
into a standard position $f$ with a standard collection of Whitney disks $W_i$ for $f$ such that
 $$
 \lambda(f_2,f_2)=0  \quad \text{ and } \quad\lambda(W_i,f_2)\in z\cdot \Z[x^{\pm 1}] .
 $$
 \end{enumerate}
As usual, $z=(1-x)(1-x^{-1})$ where $x$ is a generator of $\pi_1(S^4 \smallsetminus f )\cong \Z$. The intersection
 numbers $\lambda(-,f_2)$ are measured in the group ring $\Z[\pi_1(S^4 \smallsetminus
f )]=\Z[x^{\pm 1}]$.
 \end{thm}

\begin{proof}[Proof that $(i)\Rightarrow (ii)$ in Theorem~\ref{thm:I2}]
By general position we may
arrange that the link homotopy from $(f_1,f_2)$ to the unlink is of the following form: First finger moves on $f_1$
bring it into standard position $f$, then finger moves and Whitney moves  on $f_2$ happen (in the complement of
$f$). Finally, Whitney moves on a standard collection of Whitney disks $W_i$ for $f$ (in the complement of
$f_2'=f_0$) lead to the unlink. Working in the complement of $f$, we clearly have $\lambda(f_2',W_i)=0$ since $f_2'$ is disjoint from the $W_i$. However, $f_2'$ differs from $f_2$ by a regular homotopy in the complement of $f$ which leaves these
intersection numbers unchanged, even though $W_i$ and $f_2$ may intersect. Hence statement $(ii)$ follows.
\end{proof}

\begin{proof}[Proof that $(ii)\Rightarrow (iii)\& (i)$ in Theorem~\ref{thm:I2}]
Statement 4 of Lemma~\ref{lem:algtop} implies that any sphere $[f_2]\in\pi_2M$ which satisfies $\lambda(W_i,f_2)=0$ is homotopic (in the complement of $f_1$)
to a $\Lambda$-linear combination of Whitney spheres $S_{W_i}$. Statement 3 then implies $\lambda(f_2,f_2)=0$.
Moreover, by applying our Whitney homotopy from Section~\ref{sec:the-whitney-htpy} several times to the Whitney spheres $S_{W_i}$, we see that $(f_1,f_2)$ is link homotopically trivial.
\end{proof}
The harder work we have to do is
to relax the strong conditions $\lambda(W_i,f_2)=0$ to the much weaker condition
$\lambda(W_i,f_2)\in z\cdot \Z[x^{\pm 1}]$ of statement $(iii)$ in Theorem~\ref{thm:I2}. This missing step that $(iii)\Rightarrow (i)$ will eventually be accomplished by the Metabolic Unlinking Theorem~\ref{thm:I2-metabolic}, whose proof uses Freedman's disk embedding theorem.


\subsection{Pulling away an embedded sphere: Proof of Corollary~\ref{cor:main}} \label{sec:embedded}
Before continuing the journey towards Theorem~\ref{thm:I2-metabolic} we pause here to derive Corollary~\ref{cor:main} directly from Theorem~\ref{thm:I2}. For logical purposes (via Theorem~\ref{thm:main}) it would be sufficient to only show that the Kirk invariants vanish for link maps with one embedded component. We will do this below as well, but it's good to also know that the long arguments for Proposition~\ref{prop:I-squared} are not needed for the proof of Corollary~\ref{cor:main}. 

We will show that for any embedding $f_1:S^2 \hookrightarrow S^4$, the link map $(f_1,f_2)$ is link homotopic to $(f,f_2)$ where $f$ is in standard position with a standard collection $\{W_i\}$ of Whitney disks such that with $z=(1-x)(1-x^{-1}) \in \Lambda:= \Z[x^{\pm 1}]$ we have
\[
\lambda(W_i,f_2) \in z\cdot  \Lambda=(1-x)^2\cdot \Lambda  \quad \text{ and } \quad \lambda(f_2,f_2) =0.
\]
Thus part (iii) of Theorem~\ref{thm:I2} will apply to show that $(f_1,f_2)$ is link homotopically trivial.
Alexander duality holds for any embedding, so we know that 
\[
H_1(S^4 \smallsetminus f_1) \cong H^2(S^2) \cong \Z  \quad \text{ and } \quad H_2(S^4 \smallsetminus f_1) \cong H^1(S^2)=0.
\]
In particular, we have an infinite cyclic cover Y of this $2$--knot complement and denote by $t \in \pi_1(S^4 \smallsetminus  f_1)$ any embedded loop that maps to a generator in $H_1$. Using the coefficient exact sequence
$$
\begin{diagram}
\node{0} \arrow{e} \node{\Lambda}
 \arrow{e,t}{\cdot(1-t)} \node{\Lambda} \arrow{e,t}{ \epsilon} \node{\Z}
\arrow{e} \node{0}
\end{diagram}
$$
we get the corresponding Bockstein exact sequence
$$
\begin{diagram}
 \node{\cdots}\arrow{e} \node{H_2(Y)} \arrow{e,t}{ \cdot(1-t)} \node{H_2(Y)}
\arrow{e} \node{0}
\end{diagram}
$$
The map $f_2 : S^2 \to S^4 \smallsetminus f_1$ lifts to $Y$ and hence we get a representative $[f_2]\in H_2(Y)$. For any $N\in\N$, the above exact sequence says that there is an $a_N \in H_2(Y)$ such that $(1-t)^N\cdot a_N = [f_2]$. The equality is given by a (compact) 3-chain $c$ and hence we can find an open neighborhood $U_N$ of $f_1(S^2)$ in $S^4$ such that $S^4 \smallsetminus U_N$ contains the compact sets $f_2(S^2)$ and $t$, and the induced infinite cyclic covering $Y_N$ contains the compact sets $a_N$ and  $c$. 

Now pick a generic smoothly immersed approximation $f_N$ of $f_1$ inside $U_N$. We have a pair of inclusions
\[
S^4 \smallsetminus f_1 \supset S^4 \smallsetminus U_N \subset S^4 \smallsetminus f_N
\]
and by construction of $U_N$, the objects $t,f_2,a_N,c$ do exist in $S^4 \smallsetminus f_N$. By Alexander duality, $H_1(S^4 \smallsetminus f_N) \cong \Z$ and so $t$ is homologous to some power $x^k$ of the meridian $x$ to $f_N$. If $Y'$ is the infinite cyclic covering of $S^4 \smallsetminus f_N$, then we arrive at the equation
\[
[f_2] = (1-x^k)^N a_N = (1-x)^N a_N' \in H_2(Y'), \text{ for some }  a_N' \in H_2(Y').
\]
There are embedded arcs $b_1,\dots, b_n$ in $S^4 \smallsetminus f_2$ which have exactly their endpoints on $f_N$ and such that doing finger moves on $f_N$ guided by $b_i$ leads to an immersed $2$-sphere $f$ in standard position, see Definition~\ref{def:standard position}. Let $B_i$ denote the Whitney disks for $f$ which are opposite to the arcs $b_i$, i.e.\  the corresponding Whitney moves lead from $f$ back to $f_N$. 

One can get from $f_N \cup b_i$ to $f_N$ by cutting the arcs $b_i$. By slightly thickening these arcs this implies that one can get from $S^4 \smallsetminus (f_N \cup \nu(b_i))$ to $S^4 \smallsetminus f_N$ by attaching $3$-cells with boundaries $S_{B_i}$, the Whitney spheres for the Whitney disks $B_i$ (which are dual to the arcs $b_i$). Thus there is an exact sequence
$$
\dots \ra \langle S_{B_i} \rangle \ra H_2(S^4 \smallsetminus (f_N \cup b_i)) \ra H_2(S^4 \smallsetminus f_N) \ra 0
$$
In fact, this sequence is also exact if we use $\Z[x^{\pm 1}]=\Lambda$-coefficients which is the same as the homology of the infinite cyclic coverings. Considering $[f_2]\in H_2(Y')$ we had concluded that
$$
[f_2]=(x-1)^N\cdot a_N', \text{ for some } a_N'\in H_2(Y') = H_2(S^4 \smallsetminus f_N;\Lambda).
$$
But we may also consider $[f_2]$ as an element in $H_2(S^4 \smallsetminus (f_N \cup b_i);\Lambda)$. From the exactness of the above sequence we conclude that
$$
[f_2]+\sum_i \lambda_i\cdot S_{B_i}=(x-1)^N\cdot A_N, \text{ for some }  A_N\in H_2(S^4 \smallsetminus (f_N\cup b_i);\Lambda), \ \lambda_i\in \Lambda.
$$
After a finite number of our elementary link homotopies which add copies of the Whitney spheres $S_{B_i}$ to $f_2$, we may assume that indeed
$$
[f_2]=(x-1)^N\cdot A_N, \text{ for some } A_N\in H_2(S^4 \smallsetminus (f_N\cup b_i);\Lambda).
$$
The last step in the argument is to observe that after thickening $b_i$ and $B_i$ there is an ambient diffeomorphism 
$$
S^4 \smallsetminus \nu(f_1\cup b_i) \cong S^4 \smallsetminus \nu(f \cup B_i)
$$
Applying the inclusion $S^4 \smallsetminus \nu(f \cup B_i) \hra S^4 \smallsetminus f$ we see
that
$$
[f_2]=(x-1)^N\cdot \alpha_N, \text{ for some } \alpha_N \in H_2(S^4 \smallsetminus f;\Lambda)\cong \pi_2(S^4 \smallsetminus f).
$$
It follows with $N=2$ that $\lambda(f_2,W) = (1-x)^2 \lambda(\alpha_N,W)\in (x-1)^2\cdot \Lambda$
for any Whitney disk on $f$. In particular, this applies to a standard collection of Whitney disks $W_1,\dots,W_n$ for $f$, where Whitney moves on $W_i$ lead from $f$ to the unknotted $2$-sphere $f_0$.

Moreover, the polynomial $(1-x)^N$ divides $\lambda(f_2,f_2) = (1-x)^N \lambda(\alpha_N,(1-x)^N \alpha_N)$.
This implies that $\lambda(f_2,f_2) =0$: This is a fixed Laurent polynomial since it can be computed from the original link map $(f_1,f_2)$. Moreover, it is divisible by $(1-x)^N$ so picking $N$ larger than the degree of this  Laurent polynomial, we see that it must vanish.

We have thus checked the conditions (iii) of our Standard Unlinking Theorem~\ref{thm:I2}.  \hfill $\qed$



\subsection{Metabolic Forms, Whitney and accessory disks} \label{sec:metabolic}

To prove our main result Theorem~\ref{thm:main}
we will need a strengthening of the (as yet unproved) key implication $(iii)\Rightarrow (i)$ in the Standard Unlinking Theorem~\ref{thm:I2}.
This will involve using Whitney moves on \emph{immersed} Whitney disks to construct more general versions of Whitney spheres and accessory spheres
in a way that preserves the essential metabolic properties from the standard setting of Lemma~\ref{lem:algtop}.
So with the goal of presenting this Metabolic Unlinking Theorem~\ref{thm:I2-metabolic} in Section~\ref{sec:metabolic-unlinking-thm}, we start by
re-examining the
complement of a standard position $f:S^2\imra S^4$, and giving alternate constructions of Whitney spheres and accessory spheres using Whitney moves in local coordinates.   
\begin{figure}[ht!]
         \centerline{\includegraphics[scale=.4]{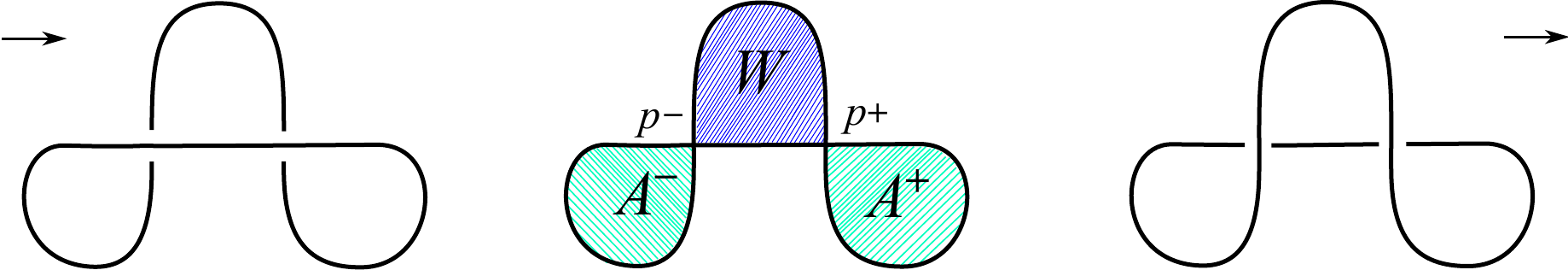}}
         \caption{A standard Whitney disk $W$ on $f$ with standard accessory disks $A^\pm$. }
         \label{fig:Whitney-and-accessory-disks}
\end{figure}

\subsubsection{Standard Whitney disk-accessory disk triples}\label{sec:w-disk-acc-disk-triples}

Observe that for any $f:S^2\imra S^4$ in standard position, it can be arranged that each standard Whitney disk $W$ pairing self-intersections $p^\pm$ is contained in a $4$--ball with $D^3 \times D^1$ coordinates as in Figure~\ref{fig:Whitney-and-accessory-disks}.
E.g.~choose the deleted slice disk representing the $1$-handle in the left Kirby diagram of Figure~\ref{standard-sphere-Kirby-diagram-3}
to have a minimum for each $\pm$-pair of green $2$-handles.

The local model of Figure~\ref{fig:Whitney-and-accessory-disks} defines what we call \emph{standard accessory disks} $A^\pm$ bounded by sheet-changing circles in $f$ and having interiors disjoint from $f$. 
A standard $W$ together with accessory disks $A^\pm$ as in Figure~\ref{fig:Whitney-and-accessory-disks} is called a \emph{standard Whitney disk-accessory disk triple}. Changing the positive crossing corresponds to the positive self-intersection $p^+$ of $f$, and changing the negative crossing corresponds to the negative self-intersection $p^-$ of $f$. The circle of $f$ on the right shrinks to a point at future times and disappears with no further singularities.

Using standard $W$ and $A$, we will provide alternative constructions of the Whitney sphere $S_W$ and accessory sphere 
$S_A$ described in the proof of Lemma~\ref{lem:algtop}. Here we start with local 2--spheres that intersect $f$ near $\partial W$, and then use Whitney moves on parallels of $W$ and $A$ to remove the intersections with $f$. This approach has the advantage of generalizing the construction to the cases where $W$ and $A$ are not necessarily standard (Section~\ref{sec:cleanly-immersed-w-moves}).

\subsubsection{Whitney spheres revisited}\label{sec:w-sphere-revisit}

Figure~\ref{fig:Whitney-sphere-blue-movie} shows an embedded \emph{pre-Whitney sphere} for a standard Whitney disk $W$. It is the union of the left-most upper and lower blue disks together with the annulus traced out by the blue loop moving clockwise from left to right along the top row, then right to left along the bottom row. The two pairs of intersections between this pre-Whitney sphere and $f$ (black) are paired by Whitney disks $W'$ and $W''$ (purple) which are Whitney parallel to $W$ (light blue). 

\begin{figure}[ht!]
         \centerline{\includegraphics[scale=.36]{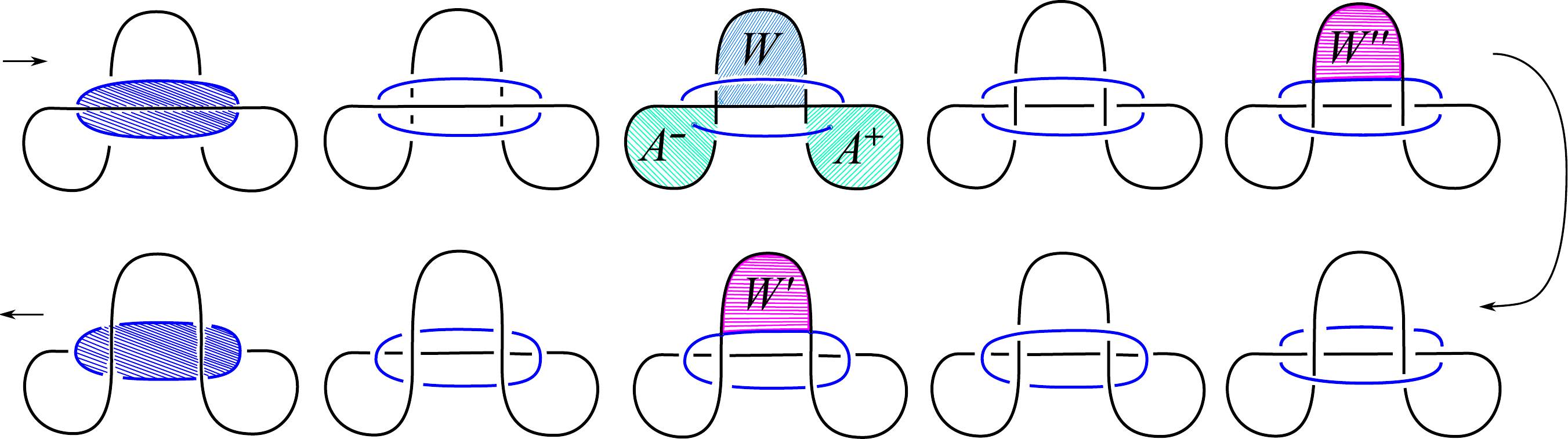}}
         \caption{A blue \emph{pre-Whitney sphere} for a standard Whitney disk $W$.}
         \label{fig:Whitney-sphere-blue-movie}
\end{figure}

Applying the $W'$- and $W''$-Whitney moves to the pre-Whitney sphere yields the \emph{Whitney sphere} $S_W$ in the complement of $f$. Note that $S_W$ is geometrically dual to each accessory disk, via an intersection inherited from the pre-Whitney sphere in the top middle picture (cf.~Figure~\ref{fig:Whitney-and-accessory-disks})

The two Whitney bubbles in the resulting $S_W$ are the same as those in the first Whitney sphere description of Figure~\ref{fig:Whitney-sphere-for-elementary-homotopy}. It is an artifact of the different choices of local coordinates that in Figure~\ref{fig:Whitney-sphere-for-elementary-homotopy} the Whitney bubbles appear to be ``on different sides'' of $W$ in the 3-dimensional present, while in Figure~\ref{fig:Whitney-sphere-blue-movie} they appear to be ``on the same side'' of $W$ (the vertical arcs descending from $W'$ and $W''$ are both over-crossings): In both cases the bubbles ``differ by meridians to $f$'', meaning that any loop formed by a path between the bubbles in $S_W$ followed by a path back through parallels of the interior of $W$ will link $f$. It will follow from the homotopy-theoretic characterization given in Lemma~\ref{lem:W-sphere-homotopy-def} that, up to link homotopy, the description here yields the same Whitney spheres as in Section~\ref{sec:w-spheres-and-homotopy} as well as in the proof of Lemma~\ref{lem:algtop}.

\begin{figure}[ht!]
         \centerline{\includegraphics[scale=.37]{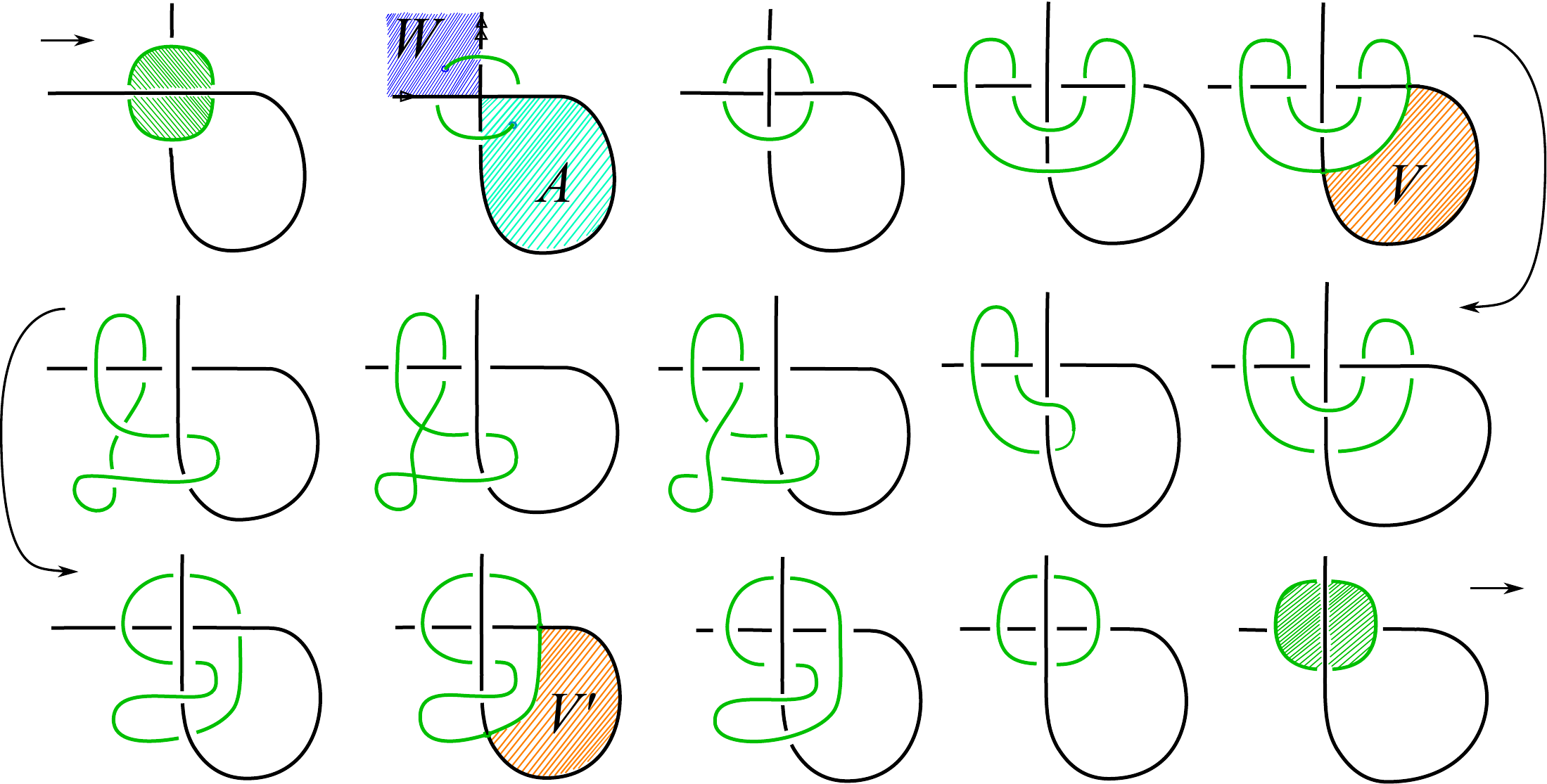}}
         \caption{A green immersed \emph{pre-accessory sphere} associated to a standard accessory disk $A$.}
         \label{fig:Accessory-sphere-movie}
\end{figure}

\subsubsection{Accessory spheres revisited}\label{sec:accessory-sphere-revisit}
Figure~\ref{fig:Accessory-sphere-movie} shows an immersed positive \emph{pre-accessory sphere} (green) supported near the corresponding self-intersection of $f$ (black). It  is the union of the two green disks (top left and bottom right) together with the homotopy of the green loop snaking from left to right along the top row, then right to left along the middle row, then left to right along the bottom row. The pre-accessory sphere is geometrically dual to each of $W$ and $A$  (top second-from-left picture).

The pair of green-black crossing changes in the top-right picture in the top row of the figure are paired by a Whitney disk $V$ formed from a parallel copy of $A$. The pair of green-black crossing changes in the bottom second-from-left picture are paired by a Whitney disk $V'$ formed from another parallel copy of $A$.
Applying the $V'$- and $V''$-Whitney moves to the pre-accessory sphere yields the positive \emph{accessory sphere} $S_A$ in the complement of $f$.  See Section~\ref{sec:w-move-on-acc-disks} for details on using parallel copies of accessory disks to guide Whitney moves.

 A pair of green-green self-crossing changes are visible in the second-from-left picture in the middle row of Figure~\ref{fig:Accessory-sphere-movie}. These self-crossing changes correspond to oppositely-signed self-intersections of $S_A$, and from the figure it can be computed that $\lambda(S_A,S_A)=2-x-x^{-1}$ as in the proof of Lemma~\ref{lem:algtop}.

From this description it is clear that $S_A$ contains four parallel copies of $A$ (two oppositely-oriented copies from each Whitney bubble), and that $S_A$ is geometrically dual to each of $W$ and $A$, via intersections inherited from the pre-Whitney sphere.

\begin{figure}[ht!]
         \centerline{\includegraphics[scale=.37]{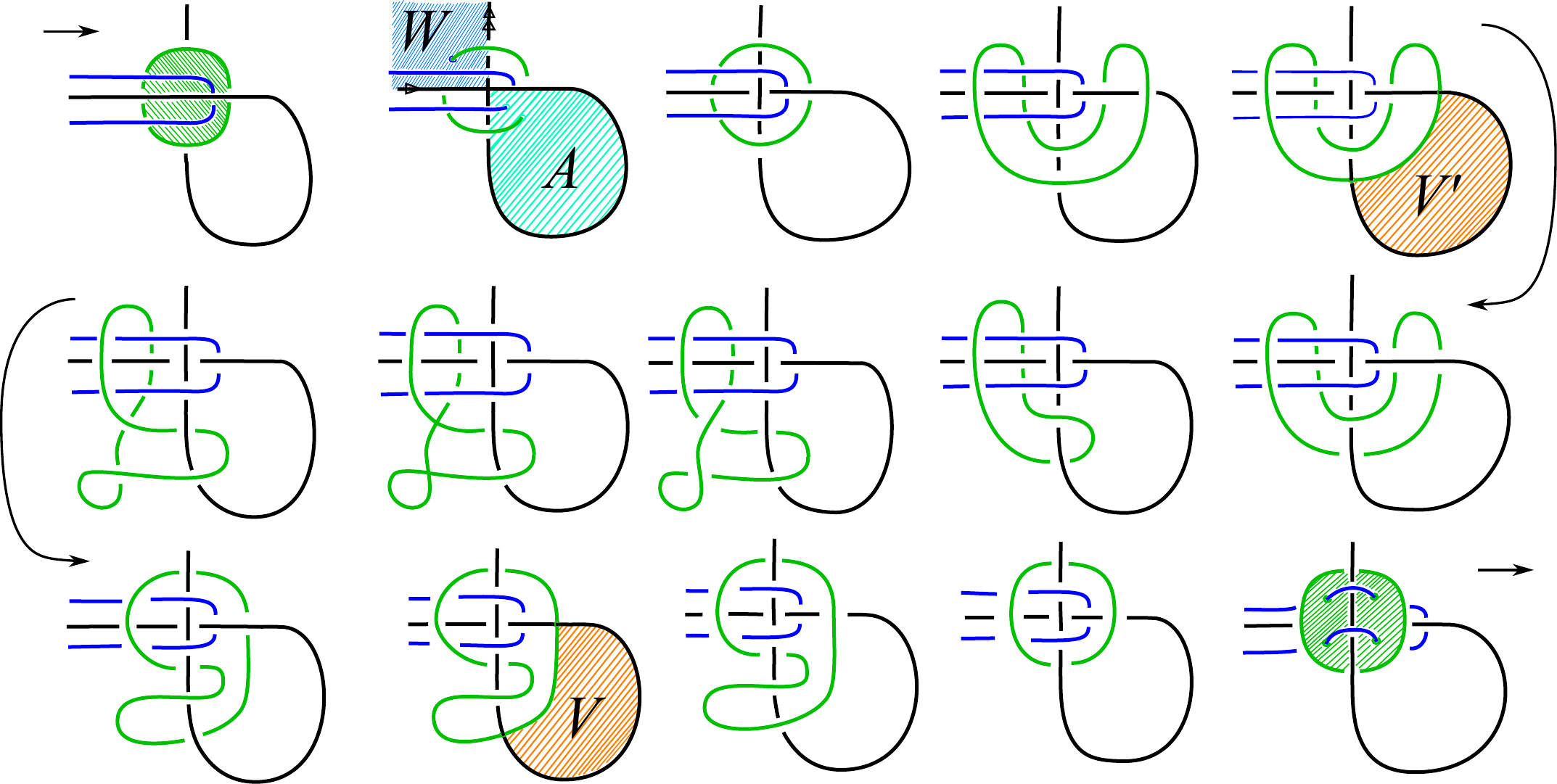}}
         \caption{The four intersection points between a positive pre-accessory sphere $S_{A}$ (green) and its pre-Whitney sphere $S_W$ (blue) are contained entirely in the bottom right picture.}
         \label{fig:Accessory-AND-Whitney-sphere-movie}
\end{figure}

\subsubsection{Whitney sphere-accessory sphere intersections}\label{sec:w-accessory-sphere-ints}
Since the standard $W$ and $A$ (and parallel copies) have disjoint interiors, the intersections between $S_W$ and $S_A$ are the same as those between the pre-Whitney and pre-accessory spheres. 
Figure~\ref{fig:Accessory-AND-Whitney-sphere-movie} shows that the only intersections between $S_A$ and $S_W$ occur near the self-intersection point of $f$ in $\partial A$; see the bottom right picture of Figure~\ref{fig:Accessory-AND-Whitney-sphere-movie}. 
It can be computed from the figure that $\lambda(S_{A},S_W)=2-x-x^{-1}$, which is consistent with Lemma~\ref{lem:algtop}.


 \subsubsection{Cleanly immersed Whitney disks and accessory disks}\label{sec:cleanly-immersed-w-moves}
An \emph{immersed Whitney disk} $W$ differs from a Whitney disk as described in Section~\ref{sec:embedded-whitney-move} in that the interior of $W$ is allowed to have finitely many transverse self-intersections, and $W$ is said to be \emph{twisted} if the Whitney section over $\partial W$ does not extend to a nowhere-vanishing normal section over $W$.  In general, if the Whitney section over $\partial W$ has normal Euler number $n\in\Z$ relative to the disk framing of $W$, then $W$ is said to be \emph{$n$-twisted}; with the case $n=0$ corresponding to $W$ being \emph{framed}. 
See Section~\ref{sec:framed-w-disks} for details.

As in the framed embedded setting, an immersed Whitney bubble can be formed from two parallel oppositely-oriented copies of $W$, 
and a \emph{Whitney move} guided by $W$ is defined by the analogous cut-and-paste construction which replaces a neighborhood of an arc of $\partial W$ in one sheet by this immersed Whitney bubble. 

Each self-intersection of $W$ gives rise to a self-intersection in any parallel copy of $W$ and a pair of intersections between the two parallel copies of $W$ which are part of the Whitney bubble. Moreover, if $W$ is $n$-twisted, then there will be $n$ additional intersections between the two parallel copies of $W$.

\begin{defi}\label{def:cleanly-immersed}
An immersed (possibly twisted) Whitney disk or accessory disk on $f$ whose interior is disjoint from $f$ is said to be \emph{cleanly immersed}.
\end{defi}
This leads to natural generalizations of Whitney spheres and accessory spheres:
In sections~\ref{sec:w-sphere-revisit} and~\ref{sec:accessory-sphere-revisit} the pre-Whitney and pre-accessory spheres are supported near $f$, so the constructions of Whitney spheres and accessory spheres described there can be carried out using cleanly immersed Whitney disks and accessory disks which are not necessarily standard. 
\textbf{From now on we take this more general definition of \emph{Whitney spheres} and \emph{accessory spheres}.}  

\begin{remnn}
See \cite{COP} for alternative constructions of Whitney spheres and accessory spheres in the complements of immersed disks in the $4$--ball, with applications to classical knot invariants.  
\end{remnn}

\subsubsection{Metabolic Whitney disks and accessory disks}\label{sec:metabolic-disks}

The following definition gives criteria for cleanly immersed disks to be ``sufficiently standard'' so as to be useful in the upcoming Metabolic Unlinking Theorem~\ref{thm:I2-metabolic}: 
\begin{defi}\label{def:metabolic-collection}
A {\em metabolic collection} of disks for a standard position $f:S^2\imra S^4$ with $2n$ self-intersections consists of cleanly immersed Whitney disks $W_1, \dots, W_n$ and cleanly immersed positive accessory disks $A_1, \dots, A_n$ satisfying the following conditions:
\begin{itemize}

\item The boundaries of all $W_i$ and $A_j$ are disjointly embedded, except that $\partial W_i$ intersects $\partial A_i$ in the $i$th positive self-intersection $p_i^+$ of $f$.
We also require that {\em in a collar neighborhood} of their boundaries, $W_i$ and $A_i$ are as in Figure~\ref{fig:Whitney-sphere-blue-movie}, in particular, that there are no intersection points in these collars other than $p_i^+=\partial W_i\cap\partial A_i$.

\item The $\Lambda$-valued intersection invariants, computed in the complement of $f$, satisfy
\[
\lambda(W_i, W_j) = 0  = \lambda(W_i, A_j) ,
\]
where $\lambda(W_i, W_i)$ is computed using a Whitney section. We note that, as a consequence, each $W_i$ is framed. 

\end{itemize}
A Whitney or accessory disk for $f$ is called {\em metabolic} if it is part of a metabolic collection.
\end{defi}

Applying the constructions described in sections~\ref{sec:w-sphere-revisit} and~\ref{sec:accessory-sphere-revisit} to a metabolic collection $\{W_i,A_i\}$ yields Whitney spheres $S_{W_i}$ and accessory spheres $S_{A_i}$ which have the same $\Lambda$-valued intersections as the Whitney and accessory spheres in Lemma~\ref{lem:algtop} \emph{except} that the intersections and self-intersections $\lambda(S_{A_i}, S_{A_j})$ of the accessory spheres are not controlled.

Lemma~\ref{lem:algtop} shows that  a metabolic collection always exists and the following result strengthens this statement:

\begin{lem} \label{lem:free}
Let $f:S^2\imra S^4$ be in standard position, with $M$ the complement of a regular neighborhood of $f$ as in Lemma~\ref{lem:algtop}. 
\begin{enumerate}
\item Any collection of the self-intersections of $f$ in oppositely-signed pairs is induced from a
standard collection of Whitney disks for $f$. Together with standard positive accessory disks, they form a metabolic collection.
\item In any metabolic collection $\{W_i, A_i\}$ for $f$, the corresponding Whitney and accessory spheres form a basis for the 
$\Lambda$-module $\pi_2(M)$. Moreover, the same sequence as in part (4) of Lemma~\ref{lem:algtop} is exact.
\end{enumerate}
\end{lem}
\begin{proof}
Claim (1) follows from the right-most Kirby diagram in Figure~\ref{standard-sphere-Kirby-diagram-4}: Each green accessory handle attaching circle corresponds to a self-intersections of $f$, and these circles can be slid past each other to get the left-most diagram of Figure~\ref{standard-sphere-Kirby-diagram-3} for any choice of pairing. 
The metabolicity follows from Lemma~\ref{lem:algtop}.

For (2), consider the $\Lambda$-homomorphism $\varphi:\Lambda^{2n} \to \pi_2(M)$ that sends the free generators to the Whitney spheres and accessory spheres constructed from the metabolic collection. Then there is a commutative diagram
\[
\begin{diagram}
\node{\Lambda^{2n}} \arrow{e,t}{\varphi} \arrow{s,l}{\cong} \node{\pi_2(M)}
\arrow{s,lr}{\cong}{\lambda/z}  \\
\node{(\Lambda^{2n})^*} \node{\pi_2(M)^*} \arrow{w,t}{\varphi^*}
\end{diagram}
\]
where $\lambda:\pi_2(M)\to \pi_2(M)^*$ is the intersection form and $z:= 2-x-x^{-1}$ is the factor arising in part (3) of Lemma~\ref{lem:algtop}. This lemma also implies that $\lambda$ is divisible by $z$ and after dividing it, becomes unimodular. The left vertical arrow is obtained by going around the diagram and hence is given by the intersection numbers (divided by $z$) of the Whitney and accessory spheres. 

By the definition of a metabolic collection, this pairing is also unimodular. The diagram hence implies that $\varphi^*$ is onto which stays true when tensoring it with $\Q$. Since $\pi_2(M)^*$ is a free $\Lambda$-module and $\Lambda\otimes \Q=\Q[x^{\pm 1}]$ is a principal ideal domain, it follows that $\varphi^*$ is also injective, hence an isomorphism.

Again the commutativity of the diagram implies that $\varphi$ is an isomorphism which proves both parts of our statement (2).
\end{proof}

\subsubsection{Metabolic isometries}\label{sec:metabolic-isometry}
The following realization result for metabolic isometries will provide the starting point for the proof of the subsequent Metabolic Unlinking Theorem:

\begin{lem}\label{lem:isometry}
Let $(f,f_2)$ be a link map with $f$ in standard position, and with $\lambda(f_2,f_2)=0$, such that
$\lambda(W_i,f_2)\in z\cdot\Lambda$ for some metabolic collection $\{W_i, A_i\}$ for $f$.
Denote by $M_n$ the complement of a regular neighborhood of $f$, where $n$ is the number of pairs of self-intersections.
 After two more finger moves on $f$ (in the complement of $f_2$) there is an
isometry $\Phi$ of the intersection form $(\pi_2M_{n+2},\lambda)$ such that:
\begin{enumerate}
\item $\Phi \equiv \id \mod z$, i.e.\  for all $a\in\pi_2M_{n+2}$ there is an element $b\in \pi_2M_{n+2}$ such that
$\Phi(a)-a=z\cdot b$, and
\item $\Phi^{-1}(f_2)\in \langle S_{W_i} \rangle$.
\end{enumerate}
\end{lem}
\begin{proof}
First observe that since $\Lambda$ has no zero-divisors an isometry of $(\pi_2M_n,\lambda)$ is the same
as an isometry of the form $(\pi_2M_n,\lambda')$ where
$$
\lambda'(a,b):=\frac{1}{z}\cdot \lambda(a,b) \quad\forall a,b\in\pi_2M_n.
$$
This hermitian form $\lambda'$ is well-defined and unimodular by part~(3) of Lemma~\ref{lem:algtop} and Definition~\ref{def:metabolic-collection}. We prefer to work with $\lambda'$ in place of $\lambda$ throughout this proof.

By our assumption $\lambda(W_i,f_2)\equiv 0\mod z$ and part~(4) of Lemma~\ref{lem:algtop} (via Lemma~\ref{lem:free}) there is a sphere $g$ with
$$
[g]\in \langle S_{W_i} \rangle \leq\pi_2M_n
$$
 such that $[g]\equiv [f_2]\mod z$. It also follows that $\lambda(g,g)=0$.
 We do one finger move on $f$ along an arc which links $f_2$ precisely once.
This link homotopy changes $f_2$ to $f_2 + S_{W_{n+1}}$ in $\pi_2M_{n+1}$, where $W_{n+1}$ is a local framed embedded Whitney disk inverse to the finger move. By changing $g$ to $g+ S_{W_{n+1}}$ we may
assume that in $\pi_2M_{n+1}$ we still have $[g]\equiv[f_2]\mod z$ and in addition
$$
\lambda'(g,g)=0=\lambda'(f_2,f_2) \text{ and } \lambda'(g,S_{A})=1=\lambda'(f_2,S_{A}).
$$
Here $S_A$ is the accessory sphere constructed from the standard (positive) accessory disk $A=A_{n+1}$ corresponding to the finger move. Let $C_1$ be the orthogonal complement of the $\Lambda$-span $\langle g, S_{A} \rangle$ in $(\pi_2M_{n+1},\lambda')$. Similarly, let $C_2$ be the orthogonal
complement of  $\langle f_2, S_{A} \rangle$ in $(\pi_2M_{n+1},\lambda')$. Since the forms on $\pi_2M_{n+1}$, $\,
\langle g, S_{A} \rangle$ and $ \langle f_2, S_{A} \rangle$ are unimodular the $\Lambda$-modules $C_i$ are stably
free and hence free \cite{Bass}. We define two bases $\mathfrak{B}_i$ of $\pi_2M_{n+1}$ by
$$
\mathfrak{B}_1:=\{ g, S_{A}, \text{ some basis for } C_1\} \quad\text {and }\quad \mathfrak{B}_2:=\{ f_2, S_{A},
\text{ some basis for } C_2 \}.
$$
Let $J$ be the $(2n+2,2n+2)$-matrix with entries in $\Lambda$ which represents the identity map on $\pi_2M_{n+1}$
in terms of the basis $\mathfrak{B}_1$ and $\mathfrak{B}_2$. Then $J$ is the matrix for an isometry between the {\em
abstract} hermitian forms
$$
\left(\begin{matrix}
0 & 1 \\
1 & 1
\end{matrix} \right)
\perp (C_1,\lambda') \text{ and }
\left(\begin{matrix}
0 & 1 \\
1 & 1
\end{matrix} \right)
\perp (C_2,\lambda').
$$
Moreover, since $[g]\equiv[f_2]\mod z$ we know that $C_1$ {\em equals} $C_2\,\mod z$, and by choosing the
above bases for $C_i$ to be equal in $\pi_2M_{n+1}/z$ we may assume that
$$
J_{ij}\equiv\delta_{ij}\mod z.
$$
 Now we introduce one further trivial finger move to $f$, leaving $f_2$ and $g$ unchanged in
$\pi_2M_{n+1}\subset\pi_2M_{n+2}$. The new orthogonal complements $D_i$ to $ \langle g, S_A \rangle$ and $ \langle
f_2, S_A \rangle$ in $(\pi_2M_{n+2},\lambda')$ are then given by
$$
D_i= \langle S_{W_{n+2}}, S_{A_{n+2}} \rangle \perp C_i.
$$
Since $\lambda'$ restricted to the first summand is just the form
$\left(\begin{smallmatrix}
0 & 1 \\
1 & 1
\end{smallmatrix} \right)$
we conclude that the matrix $J$ induces an isometry from $D_1$ to $D_2$ which is the identity modulo $z$. But
this isometry can be extended to the desired isometry of $(\pi_2M_{n+2},\lambda)$ by sending $g$ to $f_2$ and leaving
$S_A$ fixed.
\end{proof}

\subsection{The Metabolic Unlinking Theorem}\label{sec:metabolic-unlinking-thm}
The following result gives the promised strengthening of the implication $(iii)\Rightarrow (i)$ in Theorem~\ref{thm:I2}:

\begin{thm}[Metabolic Unlinking] \label{thm:I2-metabolic}

A link map $(f_1,f_2)$ is link homotopically trivial if
there are finger moves on $f_1$ (disjoint from $f_2$) which bring $f_1$
into a standard position $f$ such that $\lambda(f_2,f_2)=0$ and
 there is a metabolic collection $\{W_i,A_i\}$ for $f$ with
 $$
 \lambda(W_i,f_2)\in z\cdot\Lambda.
 $$
\end{thm}	

\begin{proof}
By Lemma~\ref{lem:isometry} we may assume that we have a link map $(f,f_2)$, with $f$ in standard position, such that the
$4$--manifold
$
M=S^4 \smallsetminus \text{ tubular neighborhood of $f $}
$
 allows an isometry
$\Phi$ of the intersection form $(\pi_2M,\lambda)$ such that
\begin{enumerate}
\item $\Phi \equiv \id \mod z$ and
\item $\Phi^{-1}(f_2)\in \langle S_{W_i} \rangle$
\end{enumerate}
for some metabolic collection $\{W_i,A_i\}$ for $f$. 
We want to use our Whitney homotopy Section~\ref{sec:the-whitney-htpy} several times to conclude that $(f,f_2)$ is link homotopically trivial.
Therefore, we need to express $f_2$ as a linear combination of Whitney spheres formed from disjoint framed embedded Whitney disks. Define spheres
$$
T_i:=\frac{1}{z}\cdot (\Phi(S_{W_i})-S_{W_i})\;\in\pi_2M.
$$
This definition makes sense since $\Phi\equiv \id\mod z$ and $\Lambda$ has no zero divisors. Then define Whitney disks
$$
V_i := W_i \# T_i
$$
which are the interior connected sum of the metabolic Whitney disks $W_i$ with the spheres $T_i$. We assert that the $V_i$ are cleanly immersed Whitney disks for $f$ with $S_{V_i}=\Phi(S_{W_i})$, satisfying $\lambda(V_i,V_j)=0$, and with algebraically transverse spheres $V_i^\dagger := \Phi(S_{A_i})$ satisfying $\lambda(V_i,V_j^\dagger)=\delta_{ij}$.

To prove this assertion, first note that by our convention we represent the spheres $T_i$ by immersions with trivial
normal bundle. In particular, adding them to $W_i$ does not change the relative Euler number of the Whitney disks. By
Lemma~\ref{lem:W-sphere-homotopy-def} in the subsequent Section~\ref{sec:homotopy Whitney sphere}, the class in $\pi_2M$ of any Whitney sphere $S_W$ is equal to $z\cdot [W]$ in the relative homotopy group $\pi_2(M,\partial M)$, so 
we have
$$
S_{V_i}=z\cdot [V_i]= z\cdot [W_i] + z\cdot T_i= S_{W_i} +
\Phi(S_{W_i})-S_{W_i}=\Phi(S_{W_i}).
$$
It follows that $\lambda(V_i,V_j)=0$ because $\Phi$ is an isometry, and thus
$$
0=\lambda(S_{W_i},S_{W_j})=\lambda(S_{V_i},S_{V_j})=z\cdot z\cdot \lambda(V_i,V_j).
$$
A similar comparison shows that $V_j^\dagger$ are algebraically transverse spheres for $V_i$:
$$
\lambda(V_i,V_j^\dagger)=\lambda(W_i,S_{A_j})=\delta_{ij}.
$$
Since $\pi_1M$ is cyclic and thus a good group, we can  apply Freedman's embedding theorem
\cite[Cor.5.1B]{FQ}
to obtain disjoint topologically framed embedded Whitney disks, regularly homotopic to $V_i$ and with algebraically transverse spheres. We continue to call the resulting collection of Whitney disks $V_i$.

The corresponding Whitney spheres $S_{V_i}$ are disjointly embedded 
and the condition $\Phi^{-1}(f_2)\in \langle S_{W_i} \rangle$ from Lemma~\ref{lem:isometry} translates into
$$
[f_2] \in \langle S_{V_i} \rangle \leq\pi_2M\;\text{ i.e.\  } \; [f_2]=\sum_{i=1}^n\alpha_i\cdot S_{V_i} \text{ with }
\alpha_i\in\Z[x^{\pm 1}].
$$
We will next apply the Whitney Homotopy from Section~\ref{sec:the-whitney-htpy} to $(f,f_2)\ n$ times. The first application is the link homotopy which does the Whitney move on $V_1$, then shrinks  
$\alpha_1\cdot S_{V_1}$, and finally does the finger move that reverses the Whitney move and returns $f$ to its original position. We get the link map
$$
(f,f_2-\alpha_1\cdot S_{V_1})=(f,\sum_{i=2}^n\alpha_i\cdot S_{V_i}).
$$
Then we apply the same moves using the disjointly embedded Whitney disks $V_2, V_3, \dots V_n$ to get a link homotopy from $(f,f_2)$ to $(f,*)$. But now $f$ shrinks in $S^4 \smallsetminus *=\R^4$ and Theorem~\ref{thm:I2-metabolic} is proven, modulo the proof of Lemma~\ref{lem:W-sphere-homotopy-def} below.
\end{proof}

We note that in the above argument, we brought $f_2$ into a position where it's the geometric sum of embedded Whitney spheres $S_{V_i}$, and is in particular disjoint from the embedded Whitney disks $V_i$.  So we could do Whitney moves on $f$ along all the $V_i$, in the complement of $f_2$, to arrive at an embedded link map $(f',f_2)$ with possibly knotted components. If this embedding was smooth one could conclude that $(f',f_2)$ is trivial from the main result of \cite{BT}, which uses singular handles and stratified Morse theory. 

Our proof here works in the topological category and is more explicit: We see step by step how $f_2$ becomes geometrically the sum of fewer and fewer Whitney spheres since each $S_{V_i}$ shrinks after doing the Whitney move on $V_i$ by a Whitney homotopy, until $f_2$ eventually becomes trivial.

\subsection{Homotopy uniqueness of the Whitney sphere} \label{sec:homotopy Whitney sphere}

The next lemma completes the proof of the Metabolic Unlinking Theorem~\ref{thm:I2-metabolic}
and gives a purely homotopy-theoretic definition of a Whitney sphere $S_W$ in terms of $W$.

Let $X^4$ be an oriented $4$-manifold (without boundary),
with $f:S^2\imra X^4$ generic and 
$M$ the complement of a tubular neighborhood of $f$.
If $f$ is not an embedding then $\pi_2(\partial M)=0$ by Remark~\ref{rem:pi2-standard-boundary-trivial}, and hence we get a monomorphism $\pi_2(M) \ira \pi_2(M,\partial M)$.
\begin{lem} \label{lem:W-sphere-homotopy-def}
For $W$ a cleanly immersed framed Whitney disk on $f:S^2\imra X^4$, the monomorphism $\pi_2(M) \ira \pi_2(M,\partial M)$ sends
$[S_W]$ to $(1+gh^{-1}-g-h^{-1})\cdot [W]$, where $g,h\in  \pi_1(\partial M)$ are both positively oriented meridians to $f$.
In particular, for $f:S^2\imra S^4$ in standard position
$[S_W]$ maps to $z\cdot [W]$. 
\end{lem}

\begin{figure}[ht!]
         \centerline{\includegraphics[scale=.35]{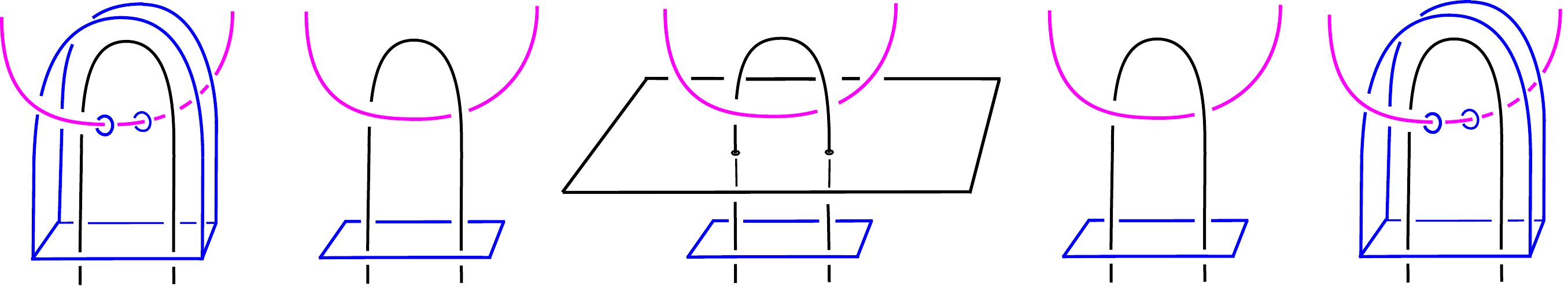}}
         \caption{In the domain of a Whitney disk $W$ the preimage of the sphere-minus-four-disks $S^0_W$ is shown in blue. The purple arc traces out a disk whose image is dual to $W$.  Compare Figure~\ref{fig:Whitney-sphere-for-elementary-homotopy}.}
         \label{fig:4-punctured-Whitney-sphere}
\end{figure}
\begin{figure}[ht!]
         \centerline{\includegraphics[scale=.25]{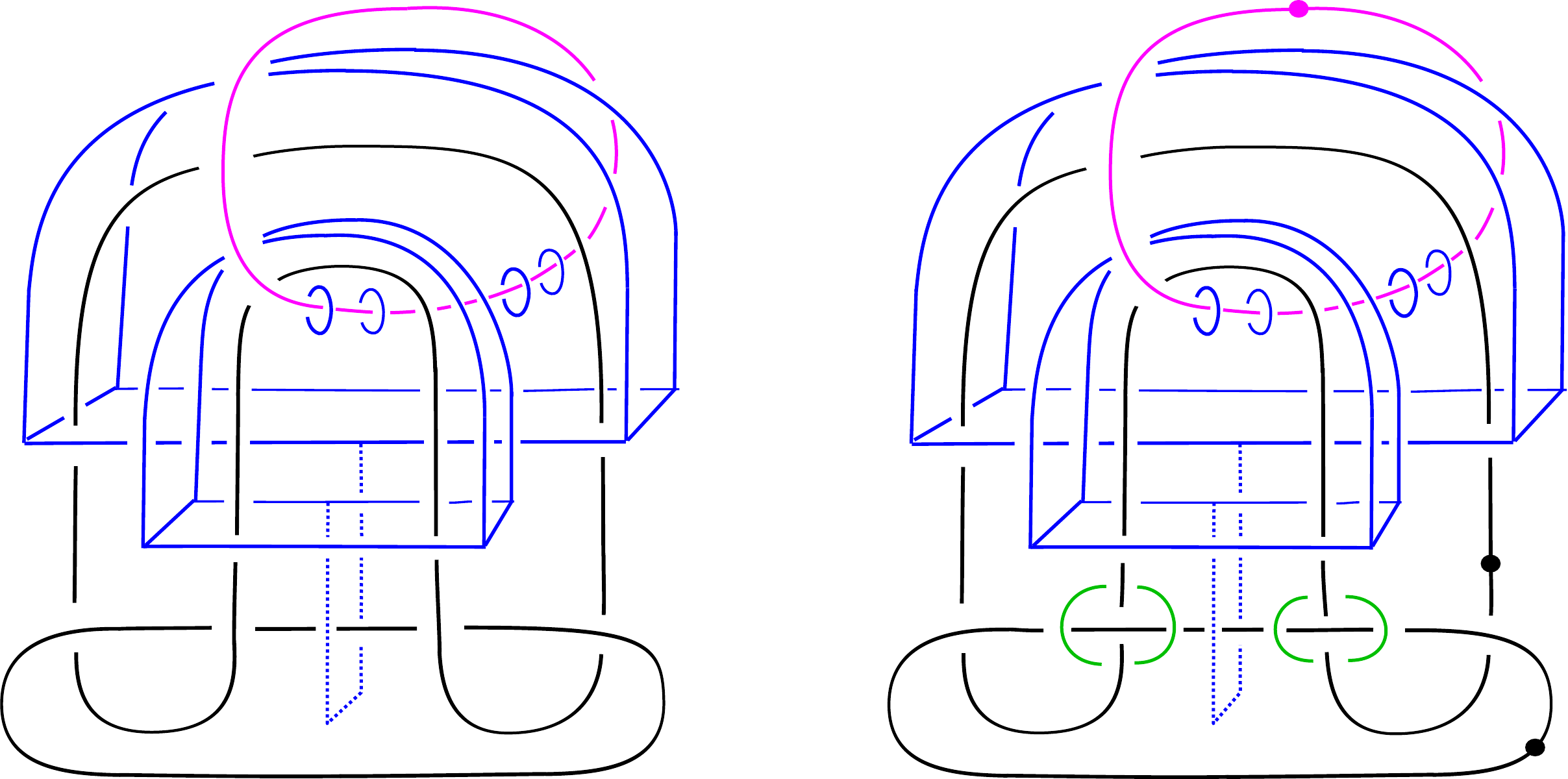}}
         \caption{Left: The black and purple sheets form a Borromean rings in the boundary $3$-sphere of Figure~\ref{fig:4-punctured-Whitney-sphere}. The two twice-punctured Whitney bubbles in $S^0_W$ are also shown, but suppressed from view (apart from the dotted PL arc) is the embedded annulus in $S^0_W$ connecting the Whitney bubbles' boundaries (cf.~Figure~\ref{fig:4-punctured-Whitney-sphere}).
         Right: A Kirby diagram of $M'$, the complement of a neighborhood of $f$ in a $S^1\times B^3$ neighborhood of $\partial W\subset M$. 
         The green circles are $0$-framed $2$-handles which realize the black-black crossing changes in the boundary.  Part of the annulus (again suppressed from view) in $S^0_W$ passes over the green $2$-handles, as shown in the left picture of Figure~\ref{fig:Whitney-sphere-annulus-and-group-elements-in-complement-of-sheets-A}.}
         \label{fig:complement-of-sheets-and-with-w-sphere}
\end{figure}
\begin{figure}[ht!]
         \centerline{\includegraphics[scale=.275]{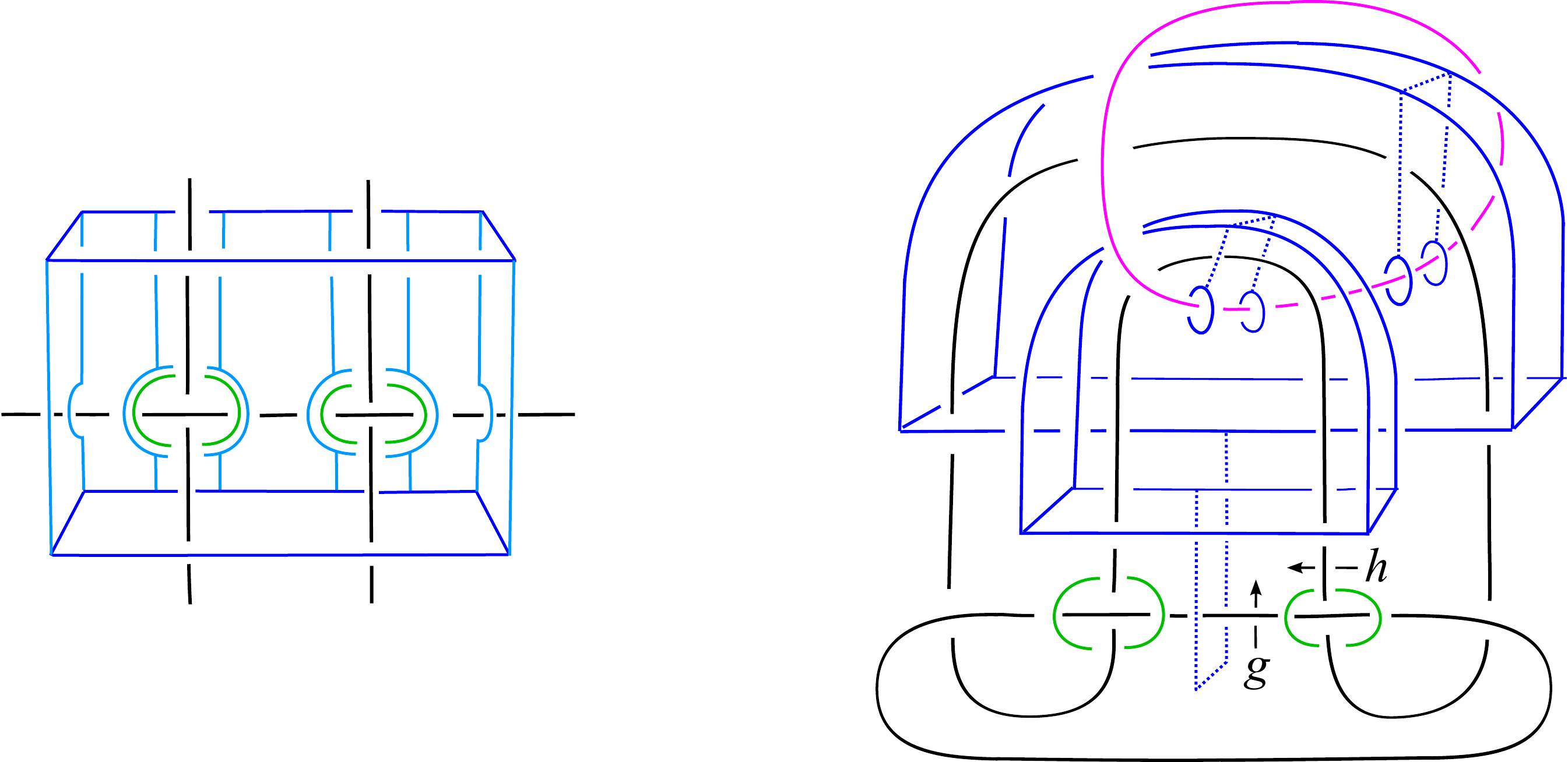}}
         \caption{Left: The part of the embedded annulus in $S^0_W$ connecting the Whitney bubbles' boundaries near the green $2$-handles.
         Right: The $3$--manifold $\partial M'\subset\partial M$, described by replacing the dotted circles in Figure~\ref{fig:complement-of-sheets-and-with-w-sphere} by $0$-framed $2$-handles.}
         \label{fig:Whitney-sphere-annulus-and-group-elements-in-complement-of-sheets-A}
\end{figure}

\begin{proof}
Deleting from $S_W$ the interiors of the four parallel copies of $W$ yields a sphere-minus-four-disks $S^0_W$ embedded in $\partial M$, so in the relative homotopy group $S_W$ maps to four copies of $W$. 
By Seifert--Van Kampen, it suffices to carry out the computation of the group elements in the intersection $M'$ of $M$ with a $S^1\times B^3$ neighborhood of $\partial W\subset X^4$. Starting with the description of $S_W$ in Figure~\ref{fig:Whitney-sphere-for-elementary-homotopy}, Figure~\ref{fig:4-punctured-Whitney-sphere} explains the Kirby diagram for $M'$ shown in Figure~\ref{fig:complement-of-sheets-and-with-w-sphere}, with $S^0_W\subset M'$. 

The right-hand side of Figure~\ref{fig:Whitney-sphere-annulus-and-group-elements-in-complement-of-sheets-A} shows two elements $g,h\in\pi_1\partial M'$ which map to a meridian of $f$. The group elements for the image of $S_W$ in $\pi_2(M,\partial M)$ are represented by loops through the punctures of $S^0_W$ that run along $S^0_W$ and return along the purple circle dual to $W$. Taking the basepoint at the second puncture from the left, yields the elements $h^{-1}$, $1$, $g$ and $gh^{-1}$, respectively, for the four punctures (from left to right). The signs of these elements alternate from left to right, so orienting $S_W$ appropriately yields $(1+gh^{-1}-g-h^{-1})\cdot [W]$ as its image in $\pi_2(M,\partial M)$. That $g$ and $h$ both map to the same meridian in $\pi_1M$ can be seen in Figure~\ref{standard-sphere-Kirby-diagram-3} or
Figure~\ref{fig:Whitney-and-accessory-disks}.
\end{proof}

\subsection{Comparing Whitney sphere descriptions}\label{sec:compare-w-spheres}
Here we check that $S_W\mapsto z\cdot W$ for the three geometric constructions of Whitney spheres given earlier in the paper.

For the descriptions of $S_W$ given in local coordinates (Figure~\ref{fig:Whitney-sphere-for-elementary-homotopy} and Figure~\ref{fig:Whitney-sphere-blue-movie}) it suffices to check that each bubble contains a pair of oppositely-oriented copies of $W$ which differ by a meridian to one sheet of $f$, and that the bubbles differ from each other by a meridian to the other sheet, as in Figure~\ref{fig:Whitney-sphere-annulus-and-group-elements-in-complement-of-sheets-A}. The same goes for the generalization of the Figure~\ref{fig:Whitney-sphere-blue-movie} construction using cleanly immersed $W$.

For the Kirby diagram description of $S_W$ in the right-most picture of Figure~\ref{standard-sphere-Kirby-diagram-3}, the embedded disk in $S_W$ bounded by the unknotted blue attaching circle contains four normal disks to a green circle, each of which is a parallel of $W$ as can be seen from the left-most picture of Figure~\ref{standard-sphere-Kirby-diagram-3}. The group elements and orientations can be computed from the figure.



\section{Injectivity of the Kirk invariant: Proof of Theorem~\ref{thm:main}}\label{sec:thm-main-proof}

This section completes the proof of Theorem~\ref{thm:main} via the following proposition which shows that a link map with vanishing Kirk invariants satisfies the algebraic intersection condition of the Metabolic Unlinking Theorem~\ref{thm:I2-metabolic}:

\begin{prop}\label{prop:I-squared}
If $(f,f_2)$ is a link map with vanishing Kirk invariants, then after a link homotopy it can be arranged that $f$ is in standard position with metabolic Whitney disks $W_i$ satisfying $\lambda(W_i,f_2)\in z\cdot\Lambda$ for all $i$. 
\end{prop}
Recall that $\Lambda=\Z[x^{\pm 1}]$ for $\pi_1(S^4\setminus f)=\langle x \rangle\cong\Z$, and $z=2-x-x^{-1}$, so $z\cdot\Lambda=I^2$ for $I$ the augmentation ideal of $\Lambda$. 
An outline of the proof of Proposition~\ref{prop:I-squared} will be given in Section~\ref{sec:proof-outline-preliminary} after introducing some terminology to facilitate the discussion:

\noindent {\bf Primary and secondary multiplicities:}
Let $(f,f_2)$ be an oriented link map with $\pi_1(S^4\setminus f)=\langle x \rangle\cong\Z$, for $x$ a positive meridian to $f$, and let $D\looparrowright S^4$ be an oriented immersed disk with $\partial D\subset f$ such that the interior of $D$ is disjoint from $f$, e.g.~$D$ is a cleanly immersed Whitney disk or accessory disk on $f$.

Then we can write 
$$
\lambda(D,f_2)=m+n(1-x)+P\in\Z[x^{\pm 1}],
$$
for $m\in\Z$, $n\in\Z$, and $P\in I^2$ (by choosing a basing of $D$ appropriately). We call the integers $m$ and $n$ the \emph{primary multiplicity} and \emph{secondary multiplicity}, respectively, of $D$. Note that this notion of ``multiplicity'' always refers to intersections with $f_2$.

The primary multiplicity is just the linking number of $\partial D$ and $f_2$, and hence can also be used in the case that $\partial D\cap f_2=\emptyset$, even if $\partial D$ is not contained in $f$.

The secondary multiplicity only depends on a push-off $\overline{\partial D}$ of $\partial D$ into 
the interior of $D$:
After (perhaps) some disjoint finger moves, we may assume that $\pi_1(S^4\setminus (f\cup f_2))$ is isomorphic to the free Milnor group on (positive) meridians $x$ and $x_2$.
Then $\overline{\partial D}$ represents $x_2^m[x,x_2]^n\in\pi_1(S^4\setminus (f\cup f_2))$, with the commutator exponent $n$ equal to the secondary multiplicity of $D$.

\subsection{Outline of the proof of the Proposition~\ref{prop:I-squared}.}\label{sec:proof-outline-preliminary}

In this terminology the conclusion of Proposition~\ref{prop:I-squared} says that \emph{$f$ admits a metabolic collection whose Whitney disks all have vanishing primary and secondary multiplicities}.

It will not be difficult to arrange the vanishing of the Whitney disks' primary multiplicities, but killing the secondary multiplicities will involve some work!

The proof will proceed by an eight-step construction which starts with a standard collection and creates many more new Whitney disks and accessory disks. This will involve the creation and manipulation of Whitney disks whose interiors are temporarily not disjoint from $f$, but disjointness will eventually be recovered using carefully constructed ``secondary'' Whitney disks. Keeping track of the different types of new disks and the progress towards controlling their intersections and multiplicities will involve some tricky book-keeping, so the following outline is provided for guidance:

\renewcommand{\theenumi}{\alph{enumi}}
\begin{enumerate}

\item\label{item:outline-initial-form}
Before starting the construction, $f_2$ will be expressed as a linear combination of standard accessory spheres which intersect an initial collection of Whitney disks $W_i$ on $f$ in a controlled way (Lemma~\ref{lem:delta-w-disks}, Section~\ref{sec:initial-standard-form}). In particular, these $W_i$ will have primary multiplicity zero.

\item\label{item:outline-reduce-to-0-or-1}
The goal of Step~1 through Step~6 of the construction is to reduce to $0$ or $1$ the primary multiplicities of all accessory disks whose corresponding Whitney disks have possibly non-zero secondary multiplicity. This reduction will allow the secondary multiplicities of the corresponding Whitney disks to be killed in the Step~7 and Step~8.

\begin{itemize}

\item Step~1 eliminates pairs of intersections between initial Whitney disks $W_i$ and $f_2$ which represent non-zero secondary multiplicity, at the cost of creating new pairs of self-intersections of $f$.
This step also constructs Whitney disks (the $U$ in Lemma~\ref{lem:p0-q1-w-and-accessory-disks}) in preparation for cleaning up new Whitney disks that will be created during later steps.

\item
Step~2 constructs Whitney disks $V_0,V_1$ for each new pair of self-intersections created in Step~1, such that $V_0$ and $V_1$ have interior intersections with $f$. These $V_0$ and $V_1$ have accessory disks with primary multiplicities $0$ and $1$, respectively, that were created in Step~1 (Lemma~\ref{lem:p0-q1-w-and-accessory-disks}).

\item
Step~3 creates even more intersections between $f$ and $V_0$ by tubing $V_0$ and $f$ into standard accessory and Whitney spheres on $f_2$. This tubing is a key part of the construction as it serves to reduce the primary multiplicities on the new accessory disks that will be created when $f$ is later pushed off of the interior of $V_0$ (in Step~5). This step also constructs temporary bigon disks $B^j$ that will be used to form these new accessory disks.

\item
Step~4 applies the analogous tubing from Step~3 but to $V_1$, and constructs quadrilateral disks $Q^j$ that will be used to form secondary Whitney disks for later making $f$ disjoint from the interior of $V_1$ (in Step~6).

\item
Step~5 pushes $f$ off of the interior $V_0$ at the cost of creating new self-intersections of $f$ paired by controlled Whitney disks $W^j$ whose associated accessory disks $A^j$ (formed from the $B^j$ in Step~3) have primary multiplicity $0$ or $1$. 

{\bf Here the super-script notation serves to differentiate these new $W^j,A^j$ from the initial collection $W_i,A_i$.}

The interior of each $W^j$ intersects $f$, but these intersections are paired by parallels $U^j$ of the previously 
constructed $U$ from Step~1. At this point only those $W^j$ whose $A^j$ has primary multiplicity $1$ are made cleanly embedded by $U^j$-moves. 

\item
Step~6 removes all interior intersections between $f$ and $V_1$ using secondary Whitney disks $R^j$ that were created during Step~5 from the quadrilaterals $Q^j$ constructed in Step~4.

\end{itemize}

A summary of the result of Step~1--Step~6 is given before Section~\ref{sec:step-7-double-boundary-twist}.

\item\label{item:outline-double-boundary-twist}
In Step~7 each Whitney disk whose associated accessory disk has primary multiplicity $1$ will be made to have
secondary multiplicity $0$ by a ``double boundary-twisting'' operation. It temporarily creates intersections between
the Whitney disk and $f$ which are then eliminated via a secondary Whitney move guided by a parallel copy of the accessory disk. More generally, this operation can change the secondary multiplicity of a Whitney disk by any integral multiple of the primary multiplicity of an associated accessory disk (Lemma~\ref{lem:double-boundary-twist}).

\item
Step~8 deals with the remaining $W^j$ from Step~5 whose $A^j$ have primary multiplicity $0$. First, the vanishing of $\lambda(f_2,f_2)$ is used to show that there are an even number of such $W^j$. Then a ``double transfer move'' is used to collect the $U^j$ in pairs that yield canceling contributions to the secondary multiplicities of the $W^j$ after making them cleanly embedded by $U^j$-Whitney moves. 

This step also creates more self-intersections of $f$ which are paired by local Whitney disks which are arranged to have vanishing secondary multiplicity. Constructing accessory disks for these local Whitney disks requires the full generality of the definition of a metabolic collection (Definition~\ref{def:metabolic-collection}).

\end{enumerate}
These eight steps are carried out first in a low-multiplicity base case, which is then checked in Section~\ref{sec:general-case-reduce-primary} to be extendable to the general case. The definition of a metabolic collection stipulates \emph{positive} accessory disks, but this will be easily arranged after Step~8 in Section~\ref{sec:check-metabolic-disks}.

Throughout the constructions $f$ and $f_2$ will only be changed by disjoint finger moves (and isotopy), hence $f$ will always be in standard position.

Details on techniques used in the constructions are given in Section~\ref{sec:appendix}.%

As an illustration of a technique that will be used repeatedly during the proof, we have the following observation on how using Whitney moves to make disks' interiors disjoint from $f$ affects the resulting primary and secondary multiplicities:

\subsubsection{Observation:}\label{sec:observation}
Let $D\imra S^4$ be an oriented immersed disk, with $D \cap f_2=\emptyset$, such that 
$D\pitchfork f$ consists of two oppositely-signed intersections paired by a Whitney disk $U$ which has interior disjoint from $f$. Then the result $D'$ of doing the $U$-Whitney move on $D$ has interior disjoint from $f$, and satisfies
$$
\lambda(D',f_2)=(1-x)\cdot\lambda(U,f_2)\in\Z[x^{\pm 1}]
$$
for an appropriate orientation and basing choice on $U$.
\begin{figure}[ht!]
         \centerline{\includegraphics[scale=.55]{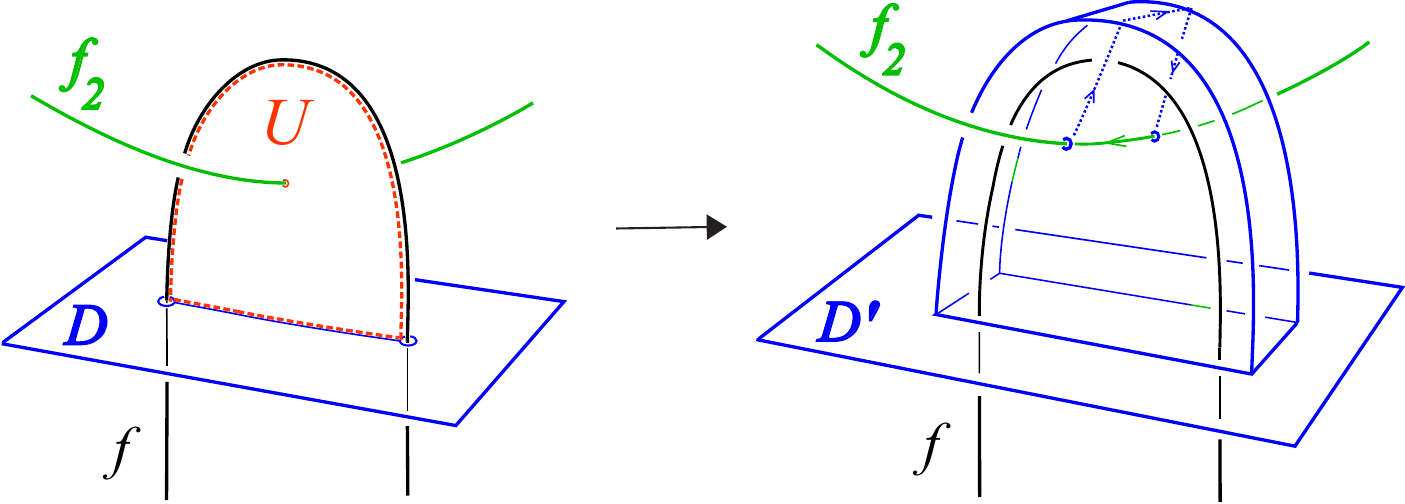}}
         \caption{Left: A Whitney disk $U$ pairing intersections between $D$ (blue) and $f$ (black), such that $U$ has an interior intersection with $f_2$ (green). Right: The result $D'$ of doing the $U$-Whitey move on 
         $D$ has a pair of oppositely-signed intersections with $f_2$ whose group elements differ by a meridian to $f$ (indicated by the dotted blue arc).}
         \label{W-disk-int-and-W-move-D-color}
\end{figure}

To see why this equation holds, observe that the $U$-Whitney move forms $D'$ from $D$ by adding two oppositely oriented copies of $U$, so each each intersection in $U\cap f_2$ gives rise to two oppositely-signed intersections between $D'$ and $f_2$ with group elements that differ by the generator $x$ represented by a meridian to $f$ (Figure~\ref{W-disk-int-and-W-move-D-color}). 
Any self-intersections or non-trivial twisting of $U$ will yield self-intersections in $D'$, but will not affect $\lambda(D',f_2)$. 

The same equation holds if $U$ is a \emph{union} of Whitney disks (connected by basings in $D$) pairing $D\pitchfork f$ with interiors disjoint from $f$, and $D'$ denotes the result of doing all the Whitney moves on these Whitney disks. 
In particular, \emph{the primary multiplicity of such a $D'$ is always zero}, and \emph{the secondary multiplicity of $D'$ is equal to the primary multiplicity of $U$}.


\subsubsection{Orientation conventions}\label{sec:orientation-conventions}

Here we fix conventions relating a choice of orientation on a Whitney disk with orientations on its accompanying accessory disks.
The reader not familiar with the construction of a Whitney disk from two accessory disks may want to at least glance at Figure~\ref{joined-accessory-disks-1} in Section~\ref{sec:W-disk-from-A-disks} before reading the second paragraph below.

A Whitney disk $W$ pairing $\{p^+,p^-\} \in f\pitchfork f$ is oriented by choosing a \emph{positive boundary arc} $\partial_+W\subset \partial W$ which is oriented from the negative self-intersection $p^-$ towards the positive self-intersection $p^+$, and a \emph{negative boundary arc} $\partial_-W\subset\partial W$ running back to $p^-$ in the other sheet of $f$. This choice of orientation on $\partial W$ together with the usual ``outward first'' convention orients $W$. 
Any positive and negative accessory disks $A^+$ and $A^-$ associated to such an oriented Whitney disk $W$ are oriented according to the convention indicated in Figure~\ref{fig:Whitney-and-accessory-disks-ORIENTED}: The boundaries $\partial A^\pm$ are oriented to run from the negative sheet of $f$ at $p^\pm$ to the positive sheet of $f$ at $p^\pm$ (again using the outward first convention).

On the other hand, given a pair of oriented accessory disks $A^{\pm}$ for $p^{\pm}$, Section~\ref{sec:W-disk-from-A-disks} shows how a Whitney disk $W$ for $p^{\pm}$ can be constructed by half-tubing together $A^+$ and $A^-$ along $f$ so that the resulting triple $W,A^+,A^-$ satisfies the orientation convention of the previous paragraph. If the $A^{\pm}$ are standard, or more generally framed and disjointly cleanly embedded, then a $4$--ball neighborhood of the triple $W,A^+,A^-$ will be diffeomorphic to Figure~\ref{fig:Whitney-and-accessory-disks-ORIENTED}. (\emph{Framed} accessory disks are discussed in Section~\ref{sec:framed-acc-disks}.)

\begin{figure}[ht!]
         \centerline{\includegraphics[scale=.48]{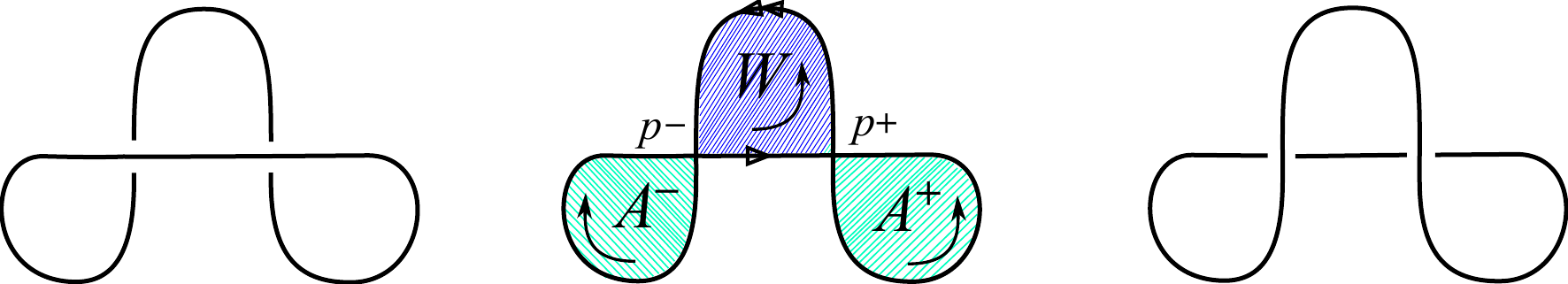}}
         \caption{Orientation conventions for Whitney disks and accessory disks.}
         \label{fig:Whitney-and-accessory-disks-ORIENTED}
\end{figure}


\subsubsection{Initial form for the proof of Proposition~\ref{prop:I-squared}:}\label{sec:initial-standard-form}

We may assume that $f$ is in standard position, admitting standard Whitney disk--accessory disk triples $W_i,A_i^+,A_i^-$ 
as in Section~\ref{sec:w-disk-acc-disk-triples} (framed, cleanly embedded, only intersecting at the self-intersections of $f$).
Since $\lambda(f,f)=0\in\Z[x_2^\pm]$, we may also assume by Lemma~\ref{lem:free}(1) that $A_i^+$ and $A_i^-$ have the same primary multiplicity $m_i$, for each $i$,  
with
orientations on the $A_i^\pm$ such that $m_i\geq 0$, and $W_i$ oriented as in 
Section~\ref{sec:orientation-conventions}, with a neighborhood of $W_i\cup A_i^+\cup A_i^-$ diffeomorphic to Figure~\ref{fig:Whitney-and-accessory-disks-ORIENTED}.

We fix orientations on the basis of standard accessory spheres $S_{A_i^\pm}$ from Lemma~\ref{lem:accessory-sphere-basis} so that $\lambda(W_i, S_{A_i^\pm})=\pm 1$, and write $f_2$ as a $\Z[x^{\pm 1}]$-linear combination of the $S_{A_i^\pm}$:
\begin{equation}\label{eq:f2-linear-comb-1}
f_2=\sum \alpha_i^+\cdot S_{A_i^+}+\alpha_i^-\cdot S_{A_i^-}
\end{equation}
where each $\alpha_i^\pm$ is of the form
\begin{equation}\label{eq:alpha-linear-comb-1}
 \alpha_i^\pm=m_i+n_i^\pm(1-x)+P_i^\pm
\end{equation}
with $n_i^\pm\in\Z$ the secondary multiplicities of $A_i^\pm$, and $P_i^\pm$ Laurent polynomials in $I^2$.

Since $W_i$ is dual to each of $S_{A_i^+}$ and $-S_{A_i^-}$, we have
\begin{equation}\label{eq:Wi-f2-ints}
\lambda(W_i,f_2)=(m_i-m_i)+(n_i^+-n_i^-)(1-x)+P_i^+-P_i^-=0+n_i(1-x)+P_i
\end{equation}
with $P_i=P_i^+-P_i^-\in I^2$, and  $n_i=n_i^+-n_i^-\in\Z$. It follows that each $W_i$ has primary multiplicity $0=m_i-m_i$, and the integer $n_i$ is the secondary multiplicity of $W_i$. The starting point for the proof of Proposition~\ref{prop:I-squared} is the observation that the tubes and parallels of $S_{A_i^\pm}$ realizing $f_2$ as the connected sum in Equation~(\ref{eq:f2-linear-comb-1}) can be chosen so that the intersections in $W_i\cap f_2$ contributing to the secondary multiplicities $n_i$ admit the ``controlled'' Whitney disks illustrated in Figure~\ref {fig:delta-first-example} and described in  
the following lemma:

\begin{figure}[ht!]
         \centerline{\includegraphics[scale=.275]{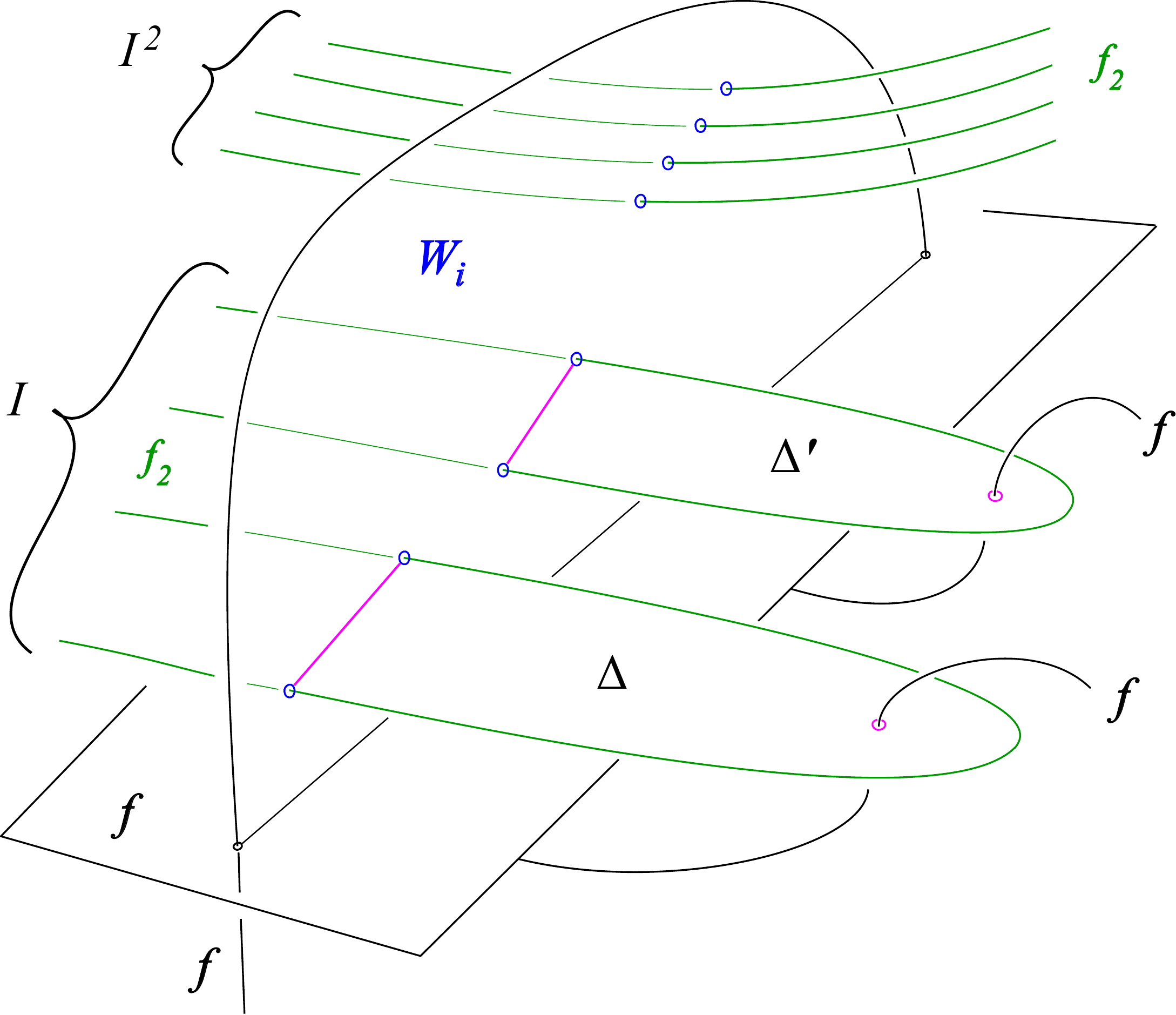}}
         \caption{Illustration of Lemma~\ref{lem:delta-w-disks}: The pairs in $W_i\cap f_2$ which contribute the terms $n_i(1-x)\in I$ to $\lambda(W_i,f_2)$ admit Whitney disks (here $\Delta$ and $\Delta'$) which each contain just a single intersection with $f$, as in Item~(\ref{lem-item:delta}).}
         \label{fig:delta-first-example}
\end{figure}
\begin{lem}\label{lem:delta-w-disks}
The connected sum decomposition of $f_2$ in Equation~\ref{eq:f2-linear-comb-1} can be realized
so that for each $i$ the following three conditions are satisfied:
\begin{enumerate}[(a)]

\item\label{lem-item:whitney-sphere-terms-disjoint-from-Wi}
The $m_i$-many copies of each $S_{A_i^\pm}\subset f_2$ 
which contribute the canceling primary multiplicities in Equation~(\ref{eq:Wi-f2-ints}) are disjoint from $W_i$.

\item\label{lem-item:delta}
The intersections in $W_i\cap f_2$ which contribute the terms $n_i(1-x)\in I$ to $\lambda(W_i,f_2)$
in Equation~(\ref{eq:Wi-f2-ints})
come in $|n_i|$-many pairs, each of which admits a framed embedded 
Whitney disk containing only a single interior intersection with $f$.

\item\label{lem-item:P-in-I-squared}
The rest of the intersections in $W_i\cap f_2$ contribute the terms $P_i\in I^2$ to $\lambda(W_i,f_2)$ in Equation~(\ref{eq:Wi-f2-ints}).

\end{enumerate}
\end{lem}

\begin{proof}

Condition~(\ref{lem-item:whitney-sphere-terms-disjoint-from-Wi}): The $m_i$-many pairs of parallel copies of each of $S_{A_i^+}$ and $S_{A_i^-}$ corresponding to the coefficients $m_i$ in Equation~\ref{eq:alpha-linear-comb-1} can be made disjoint from $W_i$ by tubing each pair of parallels of $S_{A_i^+}$ and $S_{A_i^-}$ together at their dual $\pm$-intersections with $W_i$ along a tube of normal circles to $W_i$. (Each connected sum of parallels of $S_{A_i^\pm}$ is an embedded Whitney sphere $S_{W_i}$ which is disjoint from the framed embedded $W_i$.)

\begin{figure}[ht!]
         \centerline{\includegraphics[scale=.45]{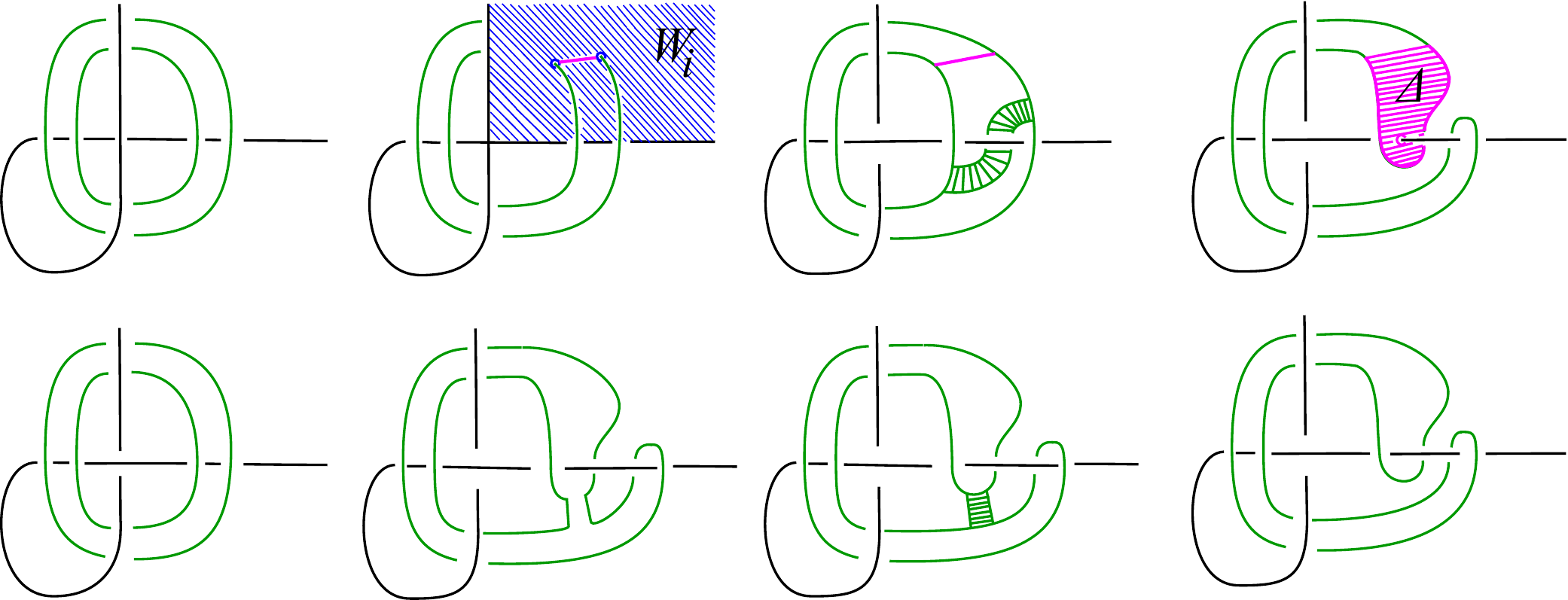}}
         \caption{From Item~(\ref{lem-item:delta}) of Lemma~\ref{lem:delta-w-disks}: Near the $i$th negative self-intersection of $f$ (black) (with time coordinate moving clockwise from left to right along the top row, then right to left along the bottom row), the Whitney disk $W_i$ (blue) intersects $\pm(1-x)\cdot S_{A_i^-}\subset f_2$ (green) in two points paired by an order $2$
Whitney disk $\Delta$ (purple). The intersection point between $\Delta$ and $f$ is visible in the bottom right picture. The negative accessory disk $A_i^-$ is suppressed from view in this figure, but could be shown in the same picture as (the corner of) $W_i$ (bottom row, second-from-left). Any copies of the $S_{A_i^\pm}\subset f_2$ corresponding to the terms of $m_i$ and $P_i^\pm$ are also not shown. In the positive accessory sphere case $\pm(1-x)\cdot S_{A_i^+}\subset f_2$ the analogous figure would look the same as this figure near $\Delta$ (compare the right-hand part of Figure~\ref{fig:x-1-term-A-delta-triple}).}
         \label{fig:x-1-term-A-Delta}
\end{figure}

Condition~(\ref{lem-item:delta}): Each pair of terms $(1-x)\cdot S_{A_i^+}\subset f_2$ contributing $+1$ to $n^+_i$, or pair of terms $-(1-x)\cdot S_{A_i^+}\subset f_2$ contributing $-1$ to $n^+_i$, (resp. $\pm(1-x)\cdot S_{A_i^-}\subset f_2$ contributing $\pm 1$ to $n^-_i$,) can be represented by taking two oppositely-oriented parallel copies of $S_{A_i^+}$ (resp. $S_{A_i^-}$) close to each other, and connecting them by a tube near $W_i$ that runs around a meridian to $f$ (see the $S_{A_i^-}$ case in Figure~\ref{fig:x-1-term-A-Delta}). 
As can be read off from Figure~\ref{fig:x-1-term-A-Delta}, each of these pairs of intersections in $W_i\cap f_2$ contributes $\pm(1-x)$ to 
$\lambda(W_i,f_2)$, and we choose the orientations of the copies of $S_{A_i^\pm}$ to yield the desired signs.
As shown in Figure~\ref{fig:x-1-term-A-Delta}, each of these pairs of intersections in $W_i\cap f_2$ admits a framed embedded Whitney disk whose interior is disjoint from all surfaces except for a single interior intersection with $f$. We only need $|n_i|$-many of these pairs since any cancellation between $n_i^+$ and $n_i^-$ in $n_i=n_i^+-n_i^-$ gets absorbed into $P_i$, as described in next paragraph.

Condition~(\ref{lem-item:P-in-I-squared}): The rest of the copies of $S_{A_i^\pm}$ in $P^\pm_i\cdot S_{A_i^\pm}\subset f_2$ (and their connecting tubes) corresponding to the terms of $P^\pm_i$ in Equation~\ref{eq:alpha-linear-comb-1} can be chosen to have no intersections with the spheres of the previous two items in the neighborhood of the self-intersection of $f$ described by Figure~\ref{fig:x-1-term-A-Delta} (they would appear as constant parallel green circles in Figure~\ref{fig:x-1-term-A-Delta}, tubed together outside the neighborhood).
The intersections of these factors of $f_2$ with $W_i$ correspond to the $P_i\in I^2$ term of $\lambda(W_i,f_2)$ in Equation~\ref{eq:Wi-f2-ints}.
We don't need to do any further clean-up on these intersections since our goal is to arrange for all metabolic Whitney disks to intersect $f_2$ in $I^2$.

Since tubes are supported near arcs, the sums of accessory spheres described so far in this proof can be connected into the single sphere $f_2$ without creating any other intersections. (Since $S_{A_j^\pm}\cap W_i=\emptyset$ for $j\neq i$, all of $W_i\cap f_2$ has been accounted for.)
\end{proof}

\subsection{Creating the assumption for Metabolic Unlinking: Proof of Proposition~\ref{prop:I-squared}}

\begin{proof}
Starting with $(f,f_2)$ as in Section~\ref{sec:initial-standard-form} satisfying Lemma~\ref{lem:delta-w-disks},
we will \textbf{first prove the special case that each accessory disk $A^\pm_i$ has primary multiplicity $0\leq m_i\leq 2$, and each Whitney disk $W_i$ either has secondary multiplicity $-1\leq n_i\leq 1$, or has arbitrary $n_i$ if $m_i=1$}. The proof will involve a construction that eliminates   
the pairs of intersections from Item~(\ref{lem-item:delta}) of Lemma~\ref{lem:delta-w-disks} which contribute $n_i(1-x)\cdot S_{A_i^\pm}\cap W_i$ at the cost of creating new metabolic Whitney disk--accessory disk pairs satisfying the criteria of 
Proposition~\ref{prop:I-squared}. This construction will be presented in eight steps in Sections~\ref{sec:step-1-reduce-primary} through~\ref{sec:step-8-pair-secondary-mults}. Then we complete the proof of 
Proposition~\ref{prop:I-squared} in Section~\ref{sec:general-case-reduce-primary} by describing how the construction can be extended to the general case of arbitrary $n_i$ and arbitrary (non-negative) $m_i$.


Before starting the construction for this special case, note that any initial pair $W_i,A_i^+$ such that $n_i=0$ already satisfies Proposition~\ref{prop:I-squared} (since $W_i$ has vanishing primary multiplicity by Lemma~\ref{lem:delta-w-disks}). Also, any initial $W_i,A_i^+$ with $m_i=1$
will be left ``as is'' for the first six steps of the construction (since Step~7 below will show how such unit primary multiplicity can be used to arrange that $n_i=0$).

So the first six steps of the construction for this special case will be applied simultaneously to all $W_i$ such that $n_i=\pm 1$, with $m_i=2$ or $m_i=0$. 
Since these steps of the construction are supported near each $W_i\cup A_i^-$ inside the $4$--ball neighborhood described by Figure~\ref{fig:Whitney-and-accessory-disks-ORIENTED}, and since these $4$--balls can be taken to be disjoint, it suffices to describe these steps locally. 

We will assume that $m_i=2$ during the first six steps, with modifications for the easier $m_i=0$ case pointed out  just after Step~6.
\subsubsection{Step 1}\label{sec:step-1-reduce-primary}

To start the construction, consider a pair $\pm(1-x)\cdot S_{A_i^\pm}\cap W_i$ of intersections between $f_2$ and a $W_i$ corresponding to the secondary multiplicity $n_i=\pm 1$ of $W_i$.  
By Lemma~\ref{lem:delta-w-disks} this intersection pair admits a framed embedded Whitney disk $\Delta$ having only a single interior intersection with $f$ as in Figure~\ref{fig:x-1-term-A-Delta}. Note that Figure~\ref{fig:x-1-term-A-Delta} shows the negative accessory sphere case $\pm(1-x)\cdot S_{A_i^-}\cap W_i$, with all other parallel copies of $S_{A_i^-}$ suppressed from view (including
the two copies of $S_{A_i^-}$ corresponding to primary multiplicity $m_i=2$ of $A_i^-$,
which would appear as constant arcs parallel to the green circle strands passing through the accessory circle $\partial A_i^-$, c.f.~Figure~\ref{fig:double-push-down-1-solid-picture}). Subsequent figures will illustrate the construction details specifically for this negative accessory sphere case, and it will be pointed out how the same constructions also work for the positive accessory sphere case $\pm(1-x)\cdot S_{A_i^+}\cap W_i$. In both cases the construction will leave the positive accessory disk $A_i^+$ unchanged. 

\begin{figure}[ht!]
         \centerline{\includegraphics[scale=.36]{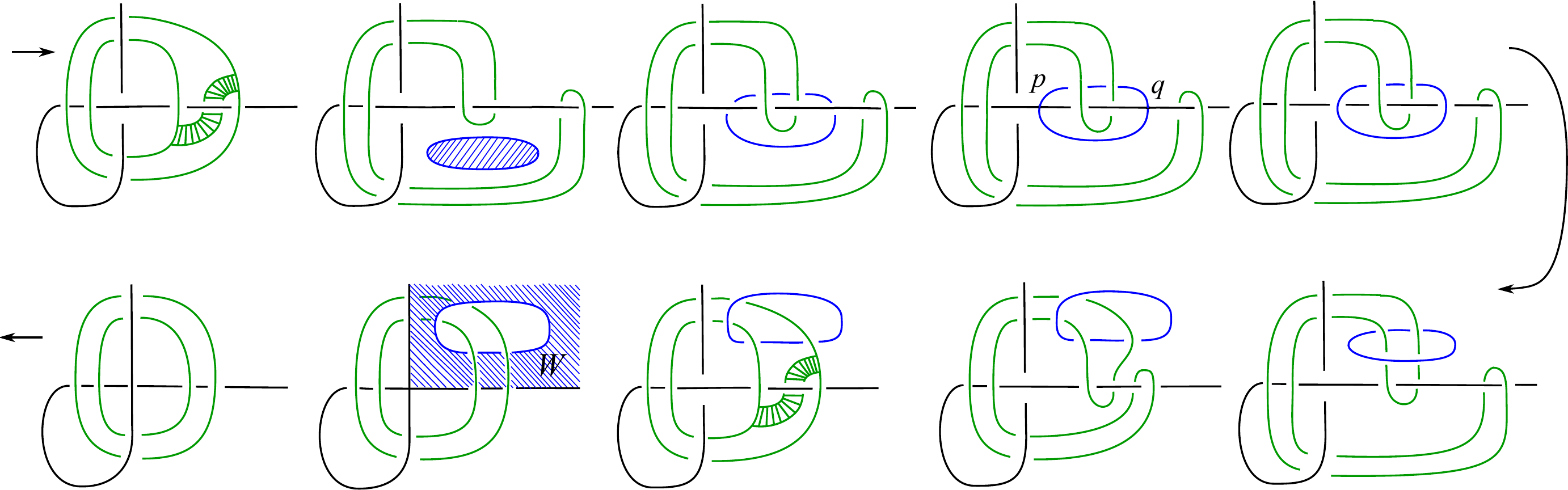}}
         \caption{The Whitney disk $W_i$ (blue) after eliminating the pair of intersections between $W_i$ and $f_2$ (green) by a Whitney move guided by $\Delta$ from Figure~\ref{fig:x-1-term-A-Delta}: Zooming in here all on but the top-left picture in Figure~\ref{fig:x-1-term-A-Delta}, the pair of interior intersections $p,q$ that $W_i$ now has with $f$ (black) corresponds to the blue-black crossing changes shown in the second-from-right picture of the top row. 
         Both the intersections $p$ and $q$ will be eliminated by the finger moves shown in Figure~\ref{fig:x-1-term-B}.}
         \label{fig:x-1-term-A1}
\end{figure}
\begin{figure}[ht!]
         \centerline{\includegraphics[scale=.34]{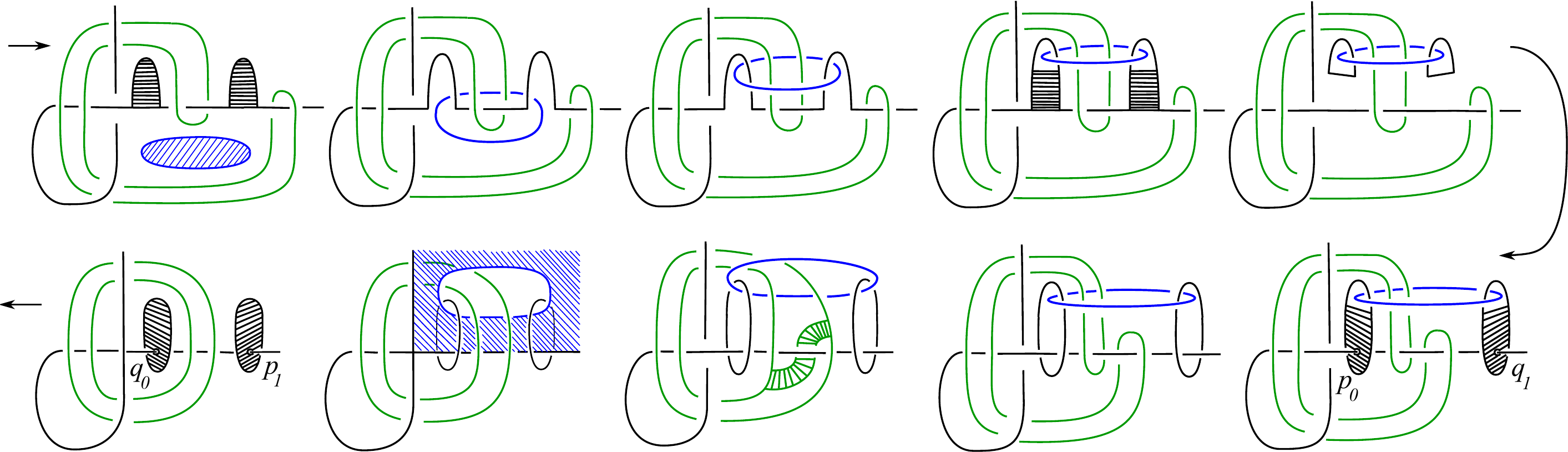}}
         \caption{This figure shows the result of eliminating the interior intersections $p,q\in W_i\pitchfork f$ (from Figure~\ref{fig:x-1-term-A1}) by two finger moves on $f$ along $W_i$ (across $\partial_+ W_i$): Zooming in on all but the top-left picture in Figure~\ref{fig:x-1-term-A1}, two of the new self-intersections $q_0,p_1\in f\pitchfork f$ are visible in the lower left-most picture, while the other two $p_0,q_1\in f\pitchfork f$ are visible in the lower right-most picture. The finger move that eliminated $p$ created the pair $p_0,q_0$, and the finger move that eliminated $q$ created $p_1,q_1$.}
         \label{fig:x-1-term-B}
\end{figure}

Change $W_i$ by doing the Whitney move guided by $\Delta$ (and keep the same notation for $W_i$).
This Whitney move eliminates the pair of intersections between $f_2$ and $W_i$, and creates a new pair of intersections $p,q$ between $f$ and the interior of $W_i$, as shown in Figure~\ref{fig:x-1-term-A1}.

Next, use two finger moves on $f$ to push each of $p,q\in W_i\pitchfork f$ off of $W_i$ (across $\partial_+ W_i$), creating two new pairs $p_0,q_0$ and $p_1,q_1$ of self-intersections of $f$, as illustrated in detail in Figure~\ref{fig:x-1-term-B}.

\begin{figure}[ht!]
         \centerline{\includegraphics[scale=.39]{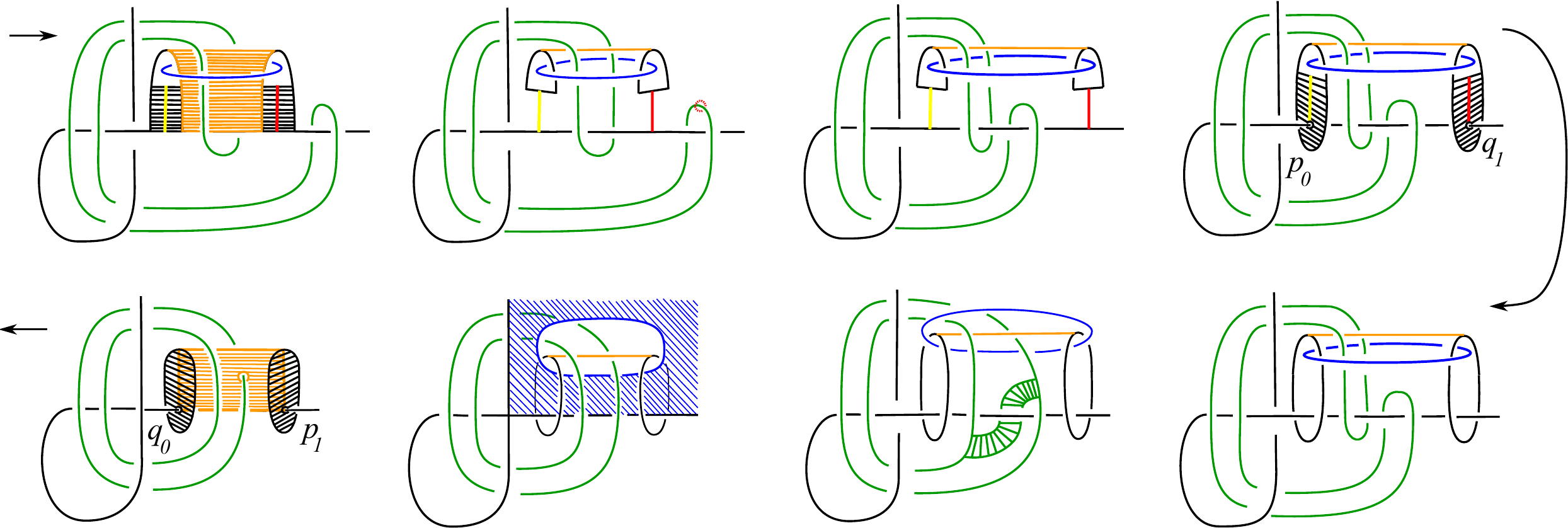}}
         \caption{Proof of Lemma~\ref{lem:p0-q1-w-and-accessory-disks}: These coordinates correspond to the last seven $3$-ball slices pictured in Figure~\ref{fig:x-1-term-B}, here with an additional slice inserted between the 5th and 6th slice. The Whitney disk $U$ for $p_1$ and $q_0$ is shown in orange. The intersection point between $U$ and $f_2$ is visible in the left-most slice in the bottom row. Across the top row, the accessory disk $A^{p_0}$ for $p_0$ is shown in yellow, and the accessory disk $A^{q_1}$ for $q_1$ is shown in red. The intersection point (indicated schematically by the small dotted red circle in the second from left slice in the top row) between $A^{q_1}$ and $f_2$ corresponds to the red-green crossing change between the center two slices in the top row. (Compare Figure~\ref{fig:double-push-down-1-solid-picture}.)}
         \label{fig:x-1-term-A0-A1-U}
\end{figure}

The disks described by the following lemma will be used repeatedly during constructions in subsequent steps:
\begin{lem}\label{lem:p0-q1-w-and-accessory-disks}
The self-intersections $p_0,q_0$ and $p_1,q_1$ of $f$ satisfy the following three conditions:

\begin{enumerate}[(i)]
\item
$p_0$ admits a framed embedded accessory disk $A^{p_0}$ such that the interior of $A^{p_0}$ is disjoint from both $f$ and $f_2$, and from all other Whitney and accessory disks.
\item
$q_1$ admits a framed embedded accessory disk $A^{q_1}$ such that the interior of $A^{q_1}$ has a single intersection with $f_2$, but is disjoint from $f$, and from all other Whitney and accessory disks.
\item
$p_1$ and $q_0$ admit a framed embedded Whitney disk $U$ such that
 the interior of $U$ has a single intersection with $f_2$, but is disjoint from $f$, and from all other Whitney and accessory disks.
\end{enumerate}
\end{lem}

\begin{proof}
The disks $A^{p_0}$, $A^{q_1}$ and $U$ are explicitly described in local coordinates in Figure~\ref{fig:x-1-term-A0-A1-U}
(see also Figure~\ref{fig:double-push-down-1-solid-picture}).
The positive accessory sphere case $\pm(1-x)\cdot S_{A_i^+}\cap W_i$ is covered by essentially the same figures.
\end{proof}
\begin{figure}[ht!]
         \centerline{\includegraphics[scale=.35]{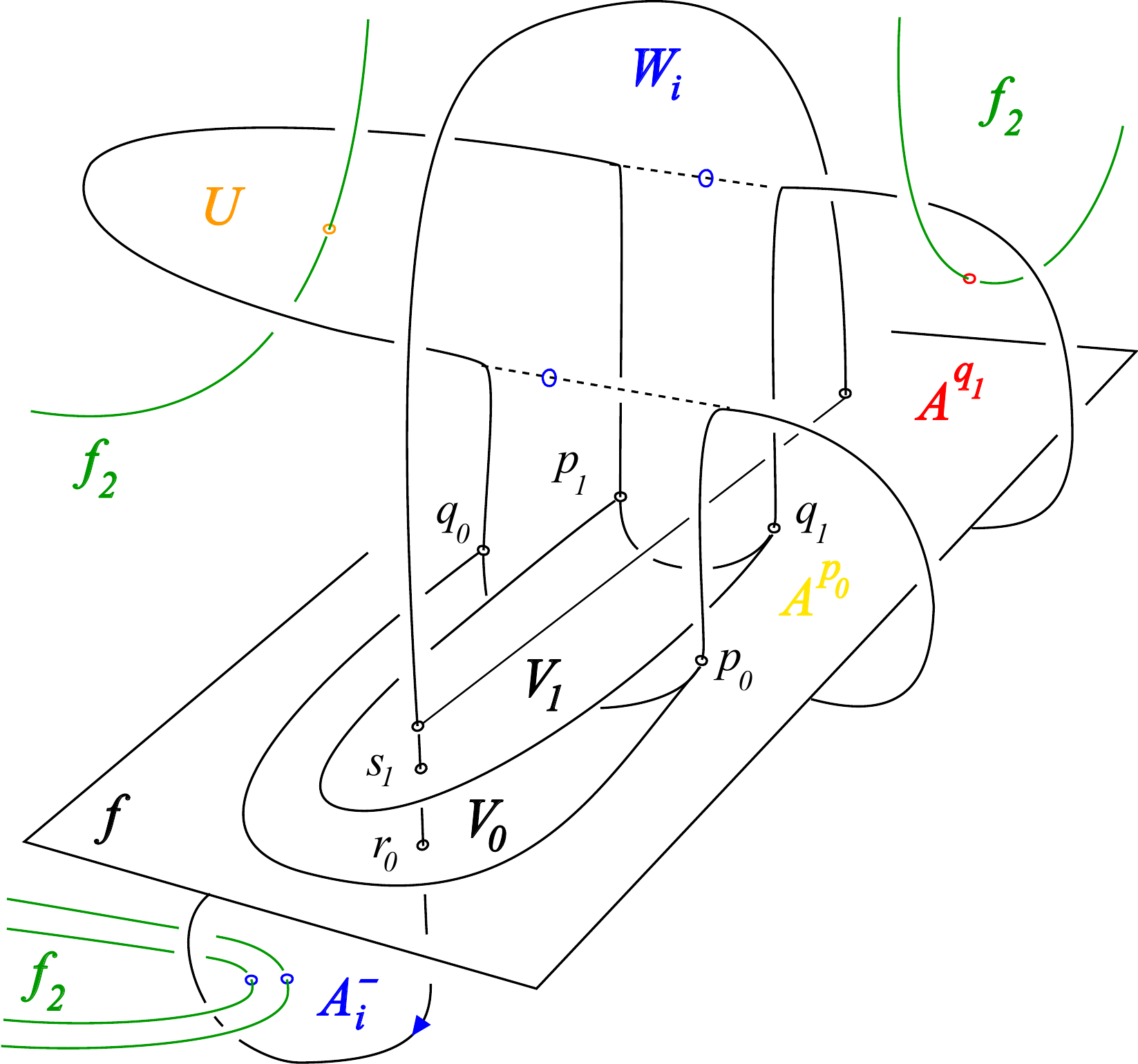}}
         \caption{From \textbf{Step 2}: The Whitney disks $V_0$ and $V_1$ (nested along $\partial_+W_i$) pairing $p_0,q_0$ and $p_1,q_1$.}
         \label{fig:double-push-down-1-solid-picture}
\end{figure}

\subsubsection{Step 2}\label{sec:step-2-reduce-primary}
As shown in Figure~\ref{fig:double-push-down-1-solid-picture}, the two new pairs $p_0,q_0$ and $p_1,q_1$ of self-intersections of $f$ created by the finger moves in Figure~\ref{fig:x-1-term-B} admit framed embedded Whitney disks $V_0$ and $V_1$ whose interiors each have only a single transverse intersection with $f$ (near the negative self-intersection of $f$ paired by $W_i$). By construction these points $r_0=V_0\pitchfork f$ and $s_1=V_1\pitchfork f$ lie on the accessory circle $\partial A_i^-$. Even if we had started with a positive accessory sphere pair $\pm(1-x)\cdot S_{A_i^+}\cap W_i$, we would still use these same nested $V_0$ and $V_1$ near the negative self-intersection of $f$.

Since the Whitney and accessory disks are framed and embedded, the $4$--ball pictured in Figure~\ref{fig:double-push-down-1-solid-picture} is accurate except that any pairs of oppositely-signed intersections that $W_i$ and $A_i^-$ may have with $f_2$ are suppressed from view. The two sheets of $f_2$ shown intersecting $A_i^-$ are like-oriented copies of $S_{A_i^-}$ corresponding to $m_i=2$ (these were suppressed in Figures~\ref{fig:x-1-term-A-Delta}, \ref{fig:x-1-term-A1}, \ref{fig:x-1-term-B} and \ref{fig:x-1-term-A0-A1-U}).

\begin{figure}[ht!]
         \centerline{\includegraphics[scale=.48]{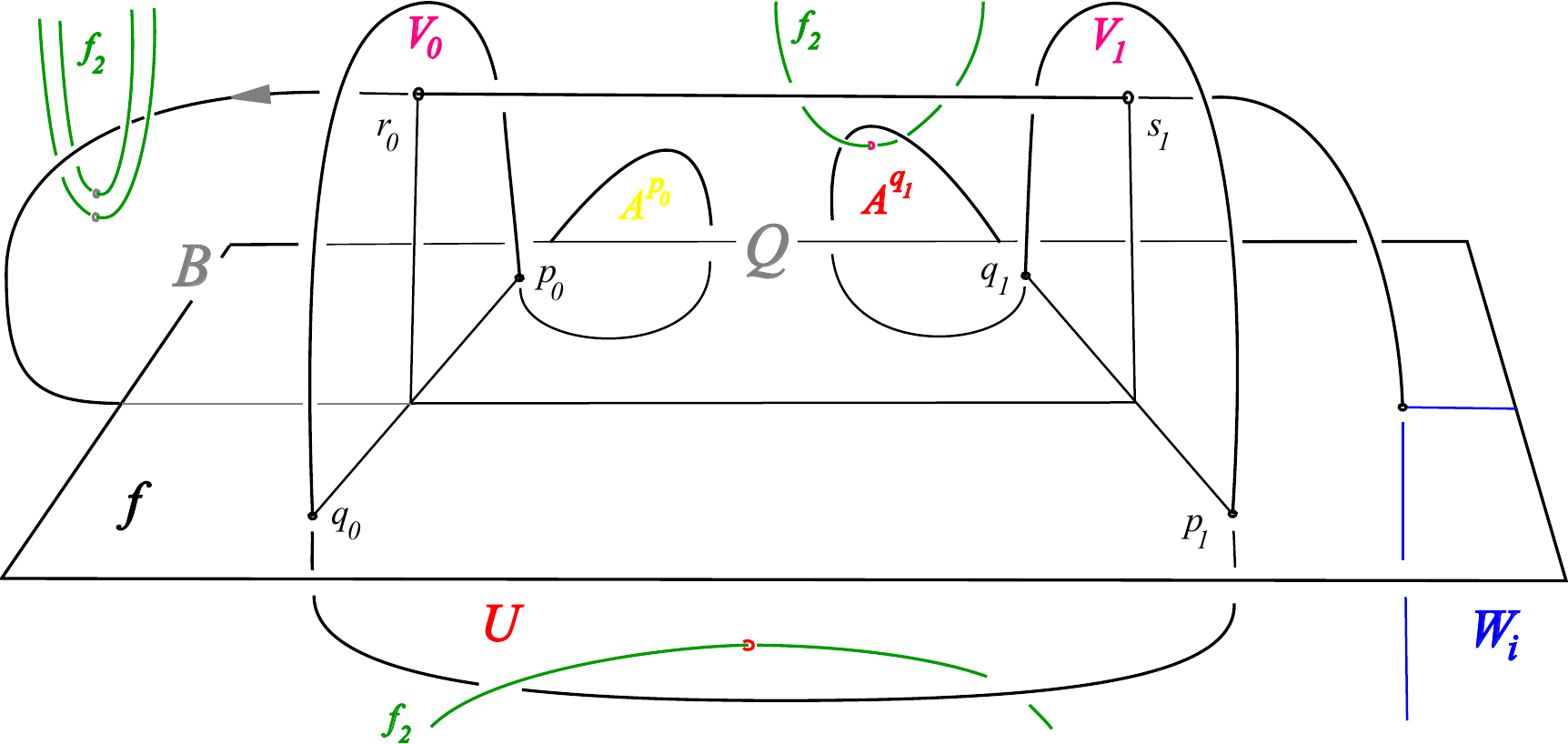}}
         \caption{\textbf{From Step 2:} A slightly different view of $V_0$ and $V_1$, from `underneath' the horizontal sheet of $f$ shown in Figure~\ref{fig:double-push-down-1-solid-picture} and after `straightening'  $\partial V_0$ and $\partial V_1$ by an ambient isotopy (of the $2$-complex).}
         \label{fig:double-push-down-2-picture}
\end{figure}

Figure~\ref{fig:double-push-down-2-picture} shows a different view of $V_0$ and $V_1$. Note that $r_0$ and $s_1$ admit a \emph{generalized Whitney disk} which is the quadrilateral $Q$ bounded by the indicated arcs running from $s_1$ along $f$ to $r_0$ then down $V_0$ and back along $f$ to $V_1$ and back up through $V_1$ to $s_1$ (Section~\ref{sec:transfer-move-appendix}). 
Also, $r_0$ admits the \emph{generalized accessory disk} $B$ which is the bigon bounded by an arc from $r_0$ along $f$ then back up through $V_0$ to $r_0$ (Section~\ref{sec:transfer-move-appendix}). This $B$ consists of most of the standard accessory disk $A_i^-$ (which we do not need, since $W_i$ also has the standard $A_i^+$), and contains all the intersections with $f_2$ that $A_i^-$ had.  

In Figure~\ref{fig:double-push-down-2-picture} the accessory disks $A^{p_0}$ (yellow) and $A^{q_1}$ (red) are visible behind the quadrilateral $Q$. A corner of $W_i$ is visible in the lower right. The (non-canceling) pair of intersections between $f_2$ and the bigon $B$ visible in the upper left are inherited from the intersections $f_2\cap A_i^-$ corresponding to $m_i=2$ shown in the lower left of Figure~\ref{fig:double-push-down-1-solid-picture}. Any canceling pairs of intersections between $B$ and $f_2$ (corresponding to canceling pairs in $A_i^-\cap f_2$) are suppressed, but otherwise the $4$--ball picture in Figure~\ref{fig:double-push-down-2-picture} is accurate since all disks are framed and embedded.

Note that the union $B\cup Q$ equals the standard $A_i^-$ with a small corner removed near the self-intersection of $f$.
The subsequent steps 3--6 will be supported in a $4$-ball neighborhood of $A_i^-$.

At this point the primary multiplicity of $B$ (the linking number of $\partial B$ with $f_2$ in $S^4$) is the same as the primary multiplicity of $A_i^-$, namely $2$,
but the construction will next exchange $r_0$ for new intersections between $f$ and $V_0$, each equipped with a generalized accessory disk whose boundary has linking number $0$ or $1$ with $f_2$. These new generalized accessory disks will be modified parallels of $B$, and will become metabolic accessory disks later in the construction.

\begin{figure}[ht!]
         \centerline{\includegraphics[scale=.5]{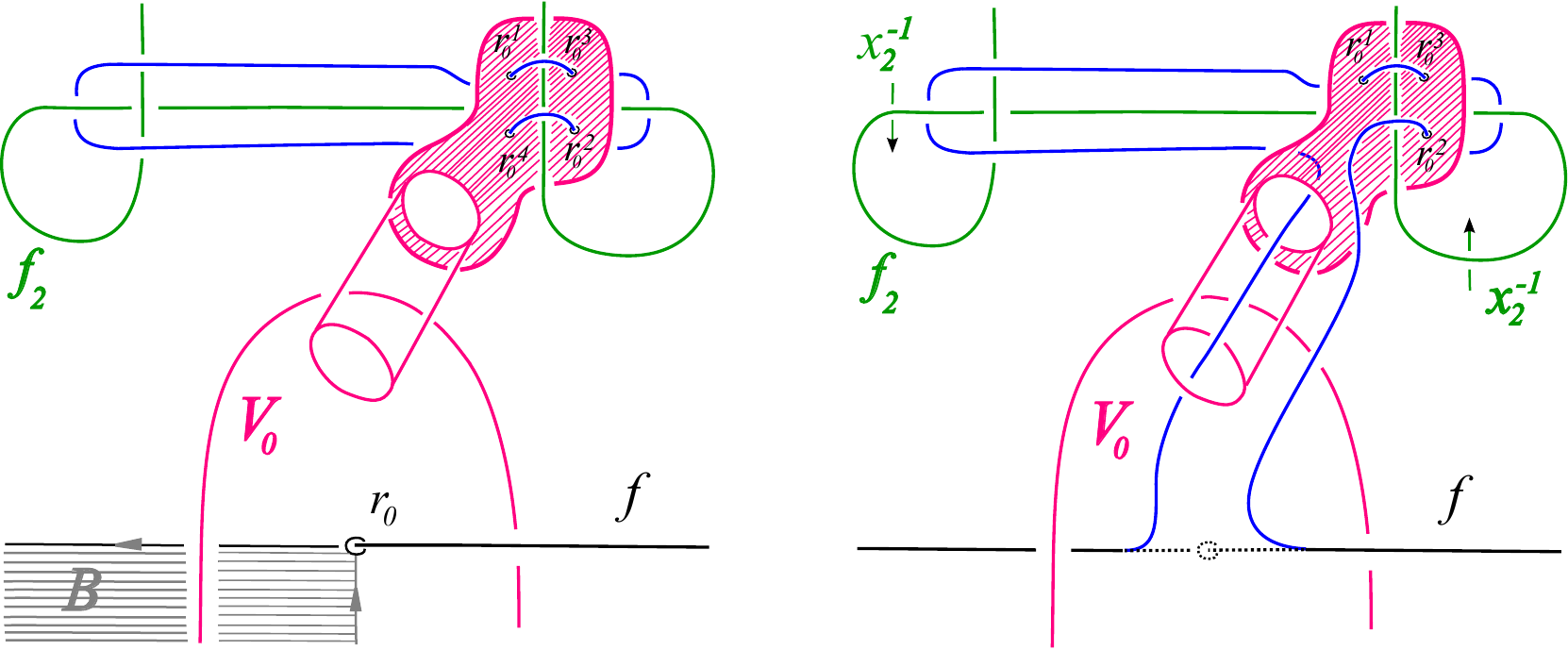}}
         \caption{\textbf{From Step 3:} Left: The Whitney disk $V_0$ (red) is tubed into the local accessory sphere $S^+_{f_2}$ (also red), creating four intersections $r_0^1, r_0^2, r_0^3, r_0^4$ between $V_0$ and the local Whitney sphere $S^W_{f_2}$ (blue).  The disk in $S^+_{f_2}$ near where the tube attaches corresponds to the bottom right-most picture in Figure~\ref{fig:Accessory-AND-Whitney-sphere-movie} except that colors are different. 
Right: $f$ (black) is tubed (blue) along $V_0$ into $S^W_{f_2}$ to eliminate $r_0$ and $r_0^4$.}
         \label{fig:tube-V0-into-accessory-sphere-1AandB}
\end{figure}

\subsubsection{Step 3}\label{sec:step-3-reduce-primary}
\begin{figure}[p]
\centering
\includegraphics[scale=.45]{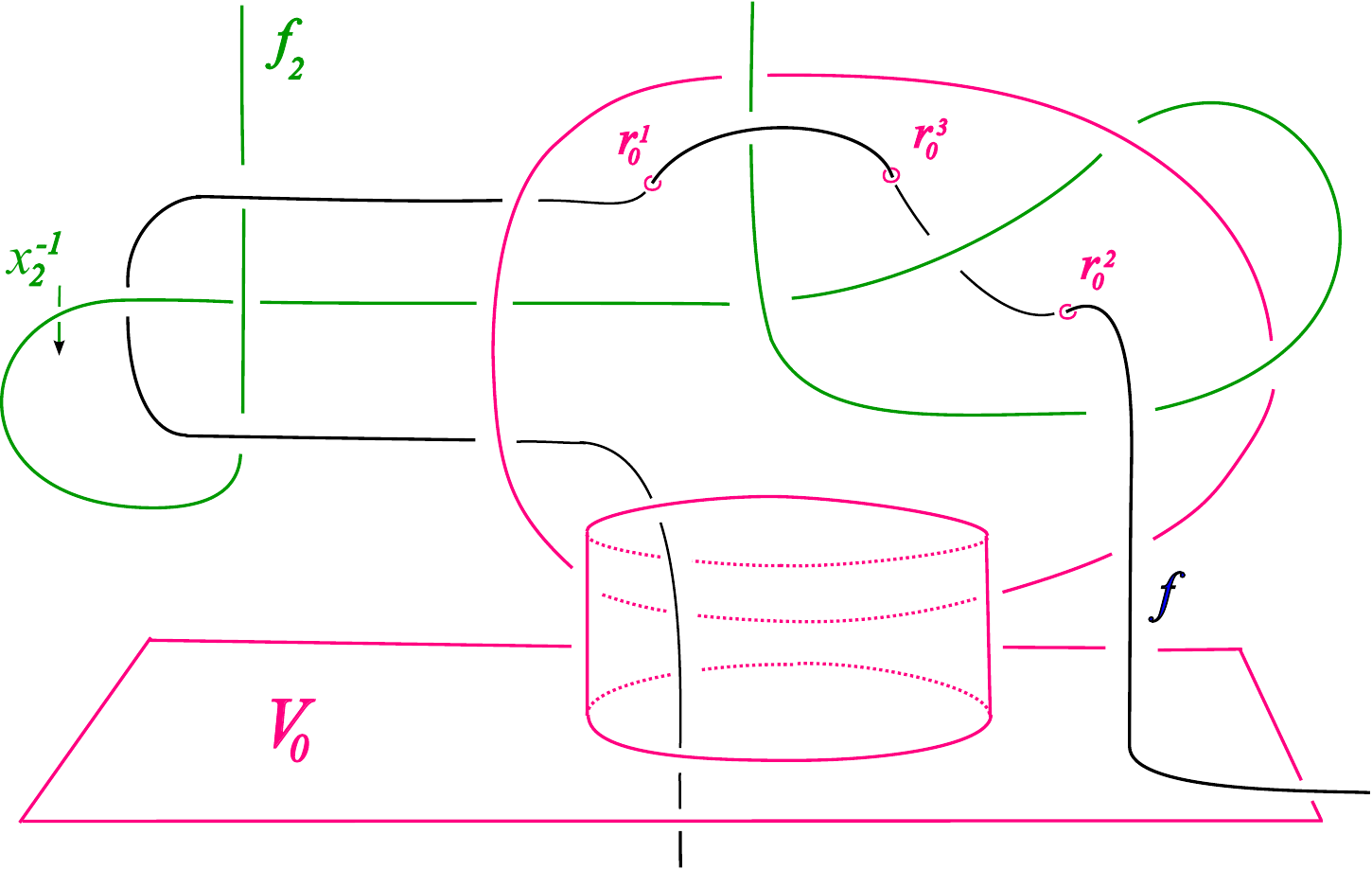}
\caption{\textbf{From Step 3:} In preparation for extending parallels of $B$ to create generalized accessory disks for $r_0^1, r_0^2, r_0^3$, this figure shows a close up of the right hand side of Figure~\ref{fig:tube-V0-into-accessory-sphere-1AandB}, with the disk in $S^+_{f_2}$ now shown as transparent, and the original part of $V_0$ now shown as horizontal, and perpendicular to the tube.}
         \label{fig:tube-V0-into-accessory-sphere-2}
\includegraphics[scale=.45]{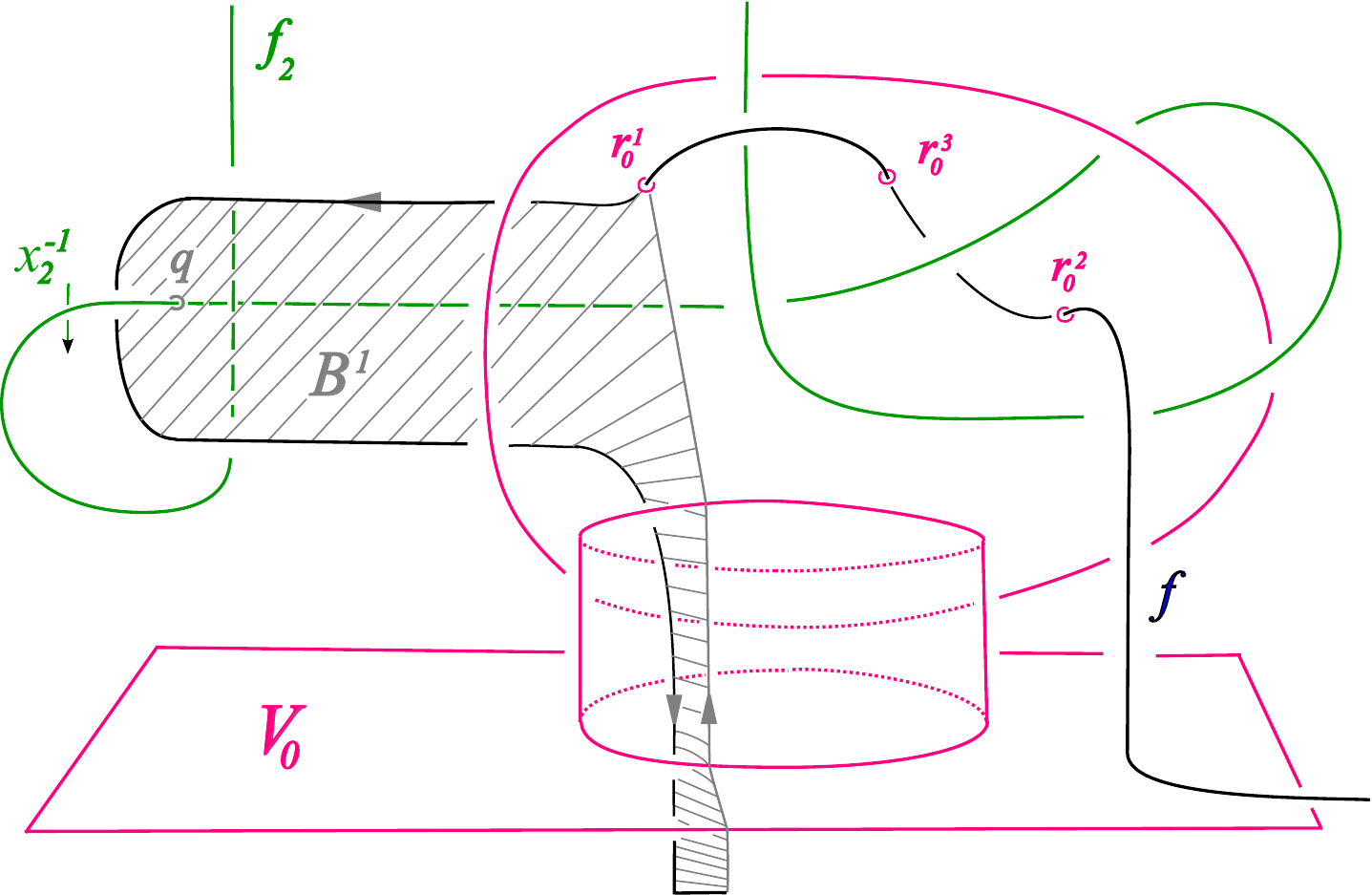}
 \caption{\textbf{From Step 3:} This figure shows the same view as Figure~\ref{fig:tube-V0-into-accessory-sphere-2}, now including part of a generalized accessory disk $B^1$ (grey) for $r_0^1$. The grey arc running up along $V_0$ indicates part of one boundary arc of $B^1$, while the other boundary arc of $B^1$ runs from $r_0^1$ to the left along $f$ and down inside the tube of $V_0$. A single intersection point $q\in B^1\cap f_2$ is visible on the left, and the rest of $B^1$ (not visible) consists of a parallel copy of $B$. The primary multiplicity of $B^1$ is $1$, since the boundary of the grey sub-disk added to $B$ to form $B^1$ represents $x_2^{-1}$, so $B^1$ has primary multiplicity one less than the primary multiplicity of $B$.}
         \label{fig:tube-V0-into-accessory-sphere-2A}
\end{figure}
\begin{figure}[ht!]
         \centerline{\includegraphics[scale=.475]{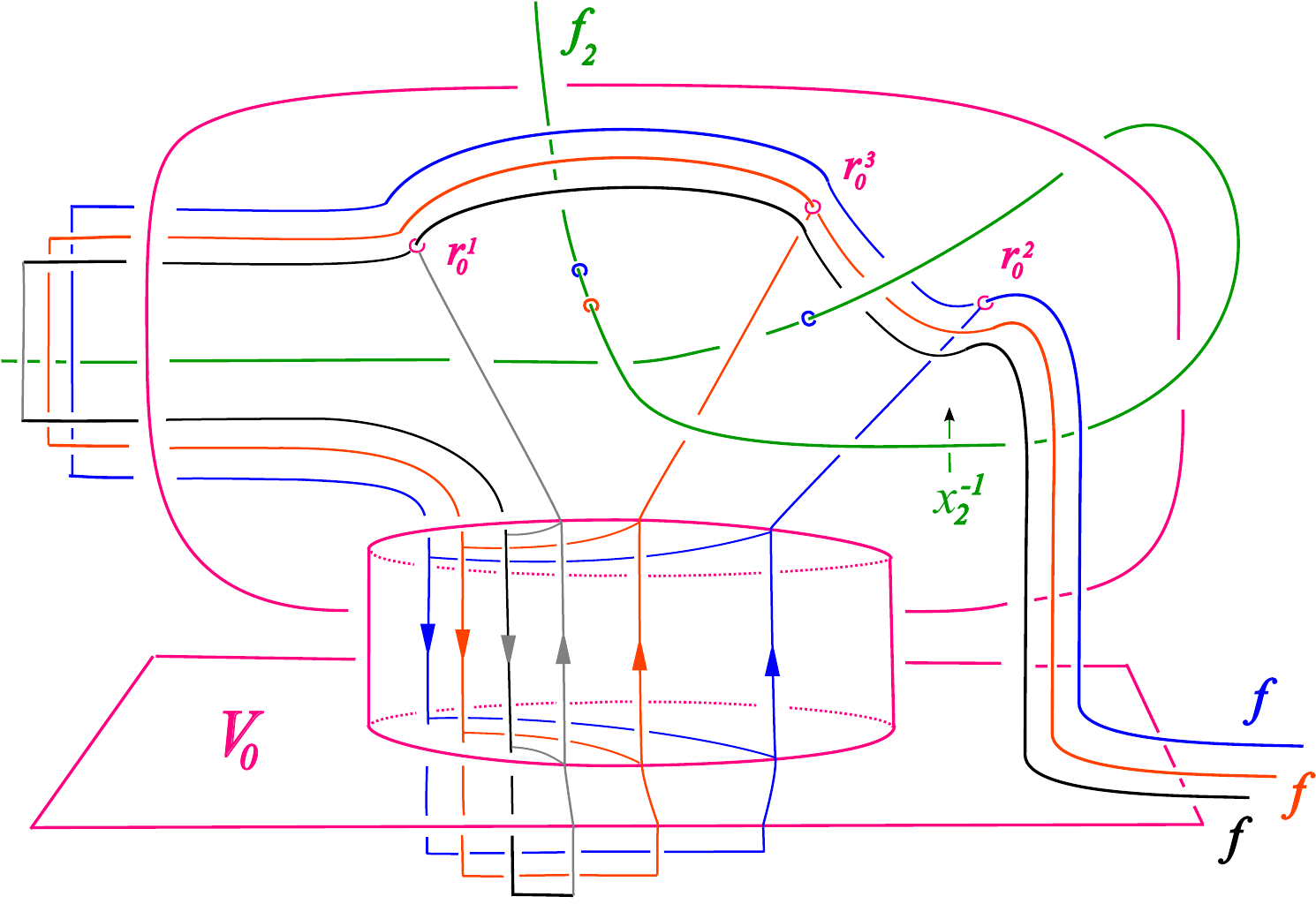}}
         \caption{\textbf{From Step 3:} The same view as Figure~\ref{fig:tube-V0-into-accessory-sphere-2A}, but now indicating parts of all three generalized accessory disks $B^1$ (grey-black), $B^2$ (blue) and $B^3$ (orange) for $r_0^1, r_0^2$ and $r_0^3$.
The picture is schematic in that two of the arcs of $f$ really lie slightly in the past and future, respectively, along with small neighborhoods in $V_0$ and the $B^j$ near the corresponding intersections $V_0\pitchfork f$. The primary multiplicities of $B^2$ and $B^3$ can be computed from this figure, keeping in mind
that each of $B^2$ and $B^3$ also has an intersection with $f_2$ parallel to $q\in B^1\cap f_2$ as on the left of Figure~\ref{fig:tube-V0-into-accessory-sphere-2A} which is not visible in this figure, and these intersections correspond to factors of $x_2^{-1}$ in $\partial B^2$ and $\partial B^3$. The single orange intersection in $B^3\cap f_2$ visible here corresponds to another factor of $x_2^{-1}$ in $\partial B^3$, hence the primary multiplicity of $B^3$ is $0$ (two less than that of $B$).
The blue positive-negative pair in $B^2\cap f_2$ indicates no further change in the linking of $\partial B^2$ with $f_2$, so the primary multiplicity of $B^2$ is $1$ (one less than that of $B$).}
         \label{fig:tube-V0-into-accessory-sphere-2B}
\end{figure}

After isotoping a small disk of $f_2$ into the $4$-ball neighborhood of $A^-_i$, perform a local trivial finger move on $f_2$, creating a local Whitney sphere $S^W_{f_2}$ and accessory spheres $S^\pm_{f_2}$, and tube $V_0$ into the new positive accessory sphere $S^+_{f_2}$ on $f_2$ via a tube that starts near $r_0$.
This creates four new intersection points $r_0^1, r_0^2, r_0^3, r_0^4$ between $V_0$ and $S^W_{f_2}$, as shown in the left side of Figure~\ref{fig:tube-V0-into-accessory-sphere-1AandB}. 

Now $r_0$ can be removed by tubing $f$ into $S^W_{f_2}$ via a tube of normal circles to $V_0$ that also eliminates $r_0^4$,
as shown in the right side of Figure~\ref{fig:tube-V0-into-accessory-sphere-1AandB}.
Note that this is a link homotopy since a Whitney sphere for $f_2$ has been added to $f$. This link homotopy is an isotopy on $f$, since the Whitney sphere is trivial (it comes from the local embedded Whitney disk which is inverse to the finger move on $f_2$).

The next step is to check that parallels of $B$ can be extended to give generalized accessory disks $B^1,B^2,B^3$ for the new intersections $\{r_0^1, r_0^2, r_0^3\}=V_0\pitchfork f$ which are framed, and disjointly embedded, with interiors disjoint from $f$. This is carried out in Figure~\ref{fig:tube-V0-into-accessory-sphere-2},
Figure~\ref{fig:tube-V0-into-accessory-sphere-2A} and Figure~\ref{fig:tube-V0-into-accessory-sphere-2B}.
The primary multiplicities of $B^1$, $B^2$ and $B^3$ are $1$, $1$ and $0$, respectively, as can be computed from the figures, using that $B$ had primary multiplicity $2$. 

It will be important that $B^1$, $B^2$ and $B^3$ do not intersect each other,  
but this is easy to see in (the slightly schematic) Figure~\ref{fig:tube-V0-into-accessory-sphere-2B} as even the projections to the present are disjointly embedded, with $B^1$ appearing in front of $B^2$, which is in turn in front of $B^3$. And the parts of $B^1$, $B^2$ and $B^3$ coming from parallels of $B$ are disjointly embedded since $B$ was framed and embedded.

At this point in the construction (after the implementations of Figure~\ref{fig:tube-V0-into-accessory-sphere-2},
Figure~\ref{fig:tube-V0-into-accessory-sphere-2A} and Figure~\ref{fig:tube-V0-into-accessory-sphere-2B}) we have disjointly embedded framed generalized accessory disks $B^1,B^2,B^3$ for $r_0^1, r_0^2,r_0^3$, with respective primary multiplicities $1,1,0$. These will eventually be converted into metabolic accessory disks on $f$ having these same primary multiplicities.

Note that $V_0$ has a new oppositely-signed pair of self-intersections inherited from $S^+_{f_2}$, but with group elements that are trivial in $\pi_1(S^4\setminus f)$.

\begin{figure}[ht!]
         \centerline{\includegraphics[scale=.5]{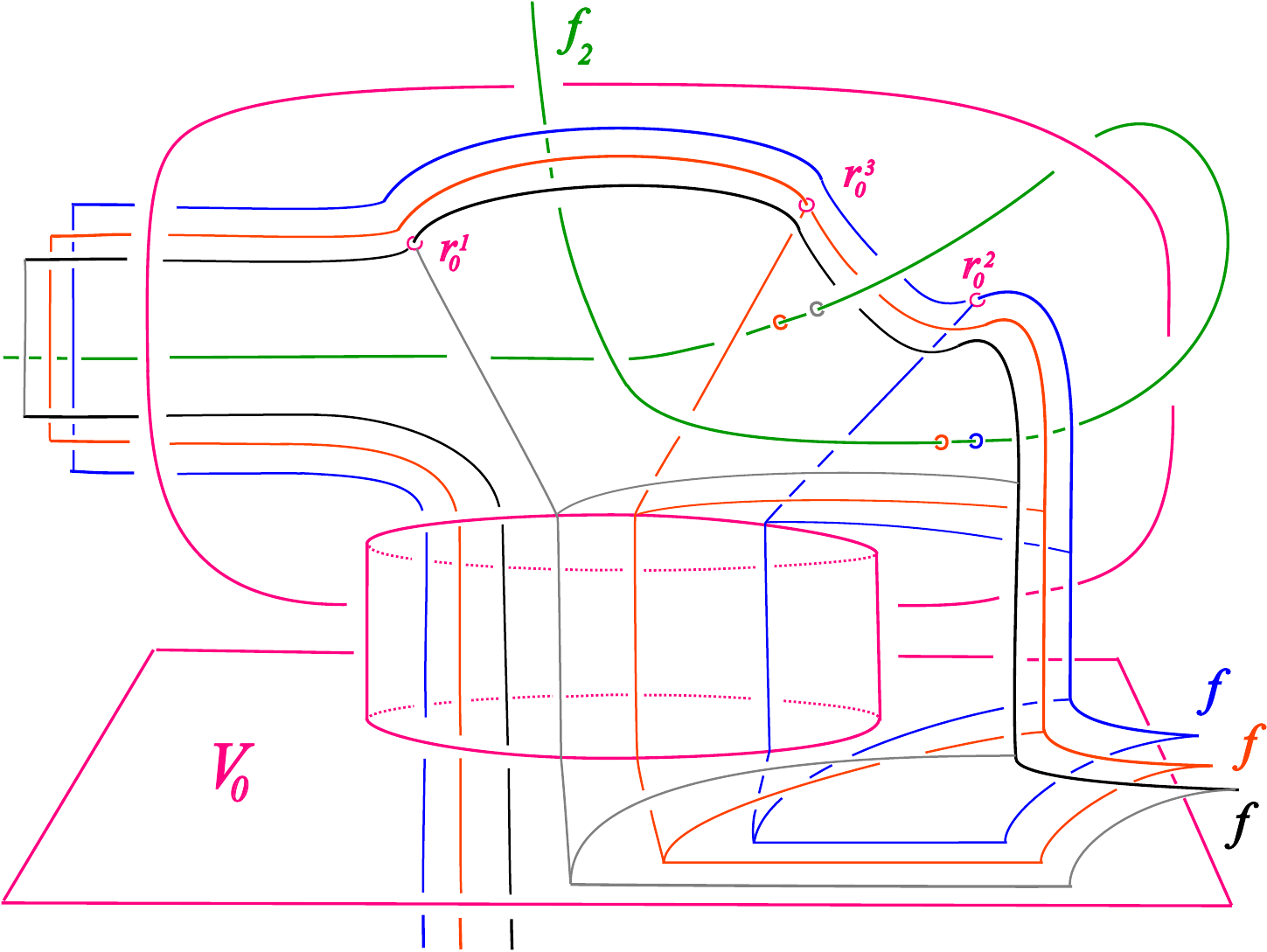}}
         \caption{\textbf{From Step 4:} A partially schematic picture of the construction of $Q^1,Q^2,Q^3$ near $r_0^1, r_0^2, r_0^3$ using the same boundary arcs along $V_0$ as for $B^1,B^2,B^3$ in Figure~\ref{fig:tube-V0-into-accessory-sphere-2B}. (Again, the picture is schematic in that two of the arcs of $f$ really lie slightly in the past and future, respectively, along with small neighborhoods in $V_0$ and the $Q^j$ of the corresponding intersections $V_0\pitchfork f$.) There is a single (grey) intersection between $Q^1$ and $f_2$; and a single (blue) intersection between $Q^2$ and $f_2$; as well as a pair of (orange) intersections between $Q^3$ and $f_2$. 
         Note that the interiors of the $Q^j$ and the $B^k$ are all disjointly embedded (the $Q^j$ are `outside' the tube, while the $B^k$ are `inside'). }
         \label{fig:tube-V0-into-accessory-sphere-3}
\end{figure}

\subsubsection{Step 4}\label{sec:step-4-reduce-primary}
Now we work on $V_1$ in a similar way as the modification of $V_0$ in Step~3: Perform another trivial finger move on $f_2$ to create another local Whitney sphere ${S^W_{f_2}}'$ and accessory sphere ${S^+_{f_2}}'$ on $f_2$, and eliminate $s_1=V_1\pitchfork f$ by tubing $V_1$ into ${S^+_{f_2}}'$, and $f$ into ${S^W_{f_2}}'$, yielding new intersections $\{s_1^1, s_1^2, s_1^3\}=V_1\pitchfork f$ (analogously to Figure~\ref{fig:tube-V0-into-accessory-sphere-1AandB} and Figure~\ref{fig:tube-V0-into-accessory-sphere-2}).

Parallel copies of the embedded quadrilateral $Q$ can now be extended to create generalized Whitney disks $Q^j$ for the pairs $r_0^j,s_1^j$ in the following way:  

Near $\{s_1^1, s_1^2, s_1^3\}=V_1\pitchfork f$ this is accomplished by the same kind of parallel local extensions that were used to create the $B^j$ from $B$ (as in Figure~\ref{fig:tube-V0-into-accessory-sphere-2B} but with $V_1$ and $s_1^j$ replacing $V_0$ and $r_0^j$, and with $Q$ and $Q^j$ replacing $B$ and $B^j$).  

Near $\{r_0^1, r_0^2, r_0^3\}=V_0\pitchfork f$, it is possible to extend the parallels of $Q$ on the ``other side'' of $V_0$ from the $B^j$ (``outside'' instead of ``inside'' the tube in the $3$-ball slice of local coordinates) in a way that does not create any intersections between the $Q^j$ and any previously created $B^k$, as shown in Figure~\ref{fig:tube-V0-into-accessory-sphere-3}.

The resulting $Q^j$ are disjointly embedded, and have interiors disjoint from all surfaces other than $f_2$.
It will not be necessary to keep track of multiplicities of the $Q^j$.

Observe that, similarly to $V_0$, now $V_1$ has a new oppositely-signed pair of self-intersections inherited from ${S^+_{f_2}}'$, with group elements that are trivial in $\pi_1(S^4\setminus f)$.

\subsubsection{Step 5}\label{sec:step-5-reduce-primary}
\begin{figure}[ht!]
         \centerline{\includegraphics[scale=.5]{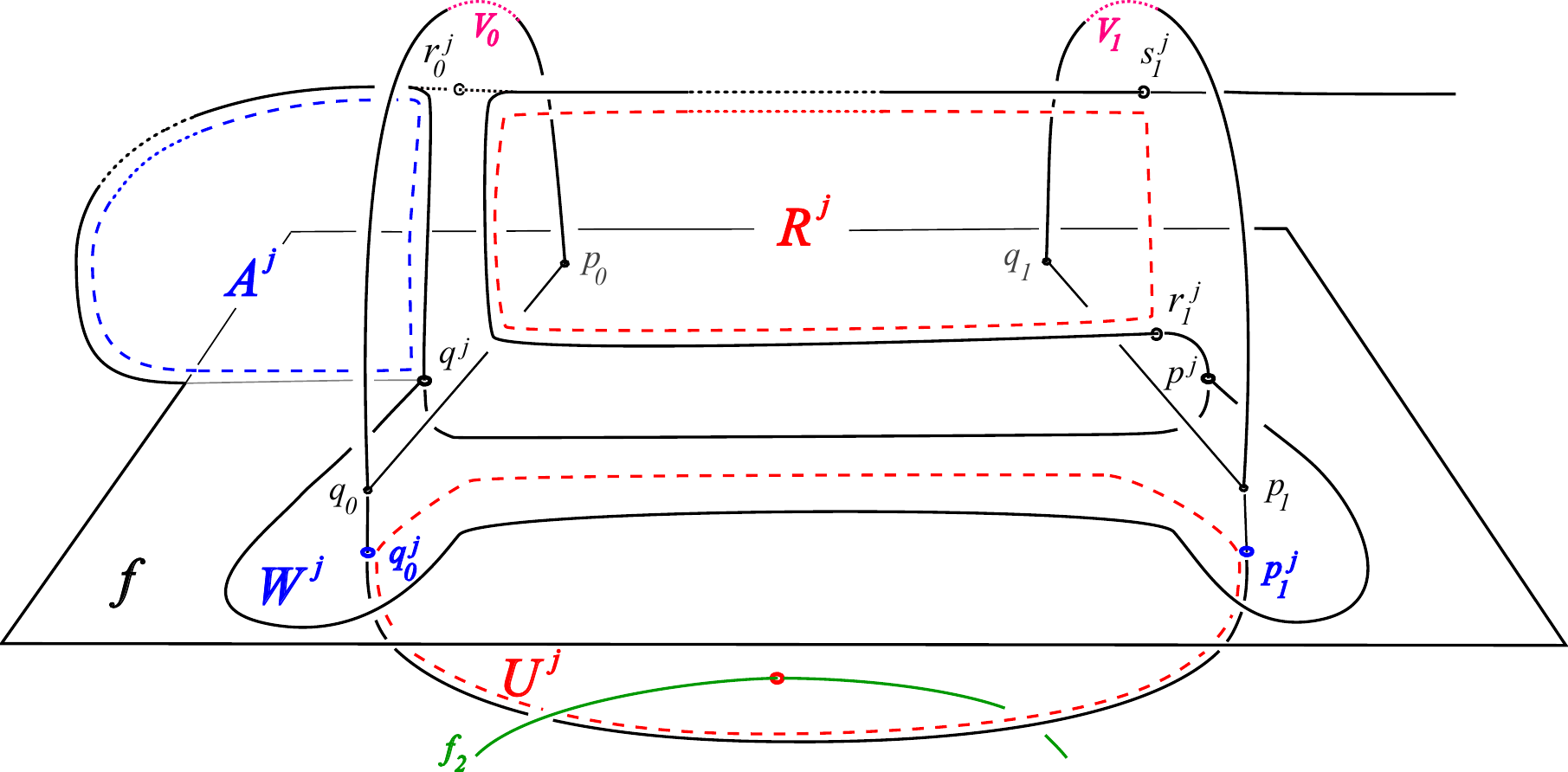}}
         \caption{\textbf{From Step 5:} A partially  condensed  view of the result of `transferring' an intersection $r_0^j\in V_0\pitchfork f$ to $r_1^j\in V_1\pitchfork f$. The new self-intersections $q^j,p^j\in f\pitchfork f$ are paired by the local Whitney disk $W^j$. A parallel copy of the Whitney disk $U^j$ from Lemma~\ref{lem:p0-q1-w-and-accessory-disks} pairs $q_0^j,p_1^j\in W^j\pitchfork f$. Suppressed from view are intersections $A^j$ and $R^j$ have with $f_2$, as well as  the self-intersections of $V_0$ and $V_1$ inherited from $S^+_{f_2}$ and ${S^+_{f_2}}'$. Also not shown are the other intersection pairs $r_0^k, s_1^k$ (or $r_1^k,s_1^k$) that $V_0$ and $V_1$ have with $f$, and their corresponding parallels of the embedded $B^k$ and $Q^k$ (or $A^k$ and $R^k$).}
         \label{fig:double-push-down-transfer-move}
\end{figure}

Now we make the interior of $V_0$ disjoint from $f$: This is accomplished by three \emph{transfer moves}, pushing down each of $r_0^1, r_0^2, r_0^3$ into $f$, then across $f$ and into $V_1$ by finger moves, creating new intersections $r_1^1, r_1^2, r_1^3$ between $V_1$ and $f$ as in Figure~\ref{fig:double-push-down-transfer-move} (see also Section~\ref{sec:transfer-move-appendix}). 

Each transfer move also creates a pair $p^j,q^j$ of self-intersections of $f$ which are paired by a local Whitney disk $W^j$ which comes equipped with a cleanly embedded, framed accessory disk $A^j$ which is formed from $B^j$. Each 
$A^j$ has interior disjoint from all other Whitney and accessory disks, and primary multiplicity equal to $0$ or $1$ since each $B^j$ had primary multiplicity $0$ or $1$ (Step~3). Each $W^j$ has a single pair of interior intersections $q_0^j,p_1^j$ with $f$ that can be paired by a parallel copy $U^j$ of the framed embedded
Whitney disk $U$ from Lemma~\ref{lem:p0-q1-w-and-accessory-disks}. 
Note that these $W^j$ can be nested disjointly, as were $V_0,V_1$ in Figure~\ref{fig:double-push-down-2-picture} (see also Figure~\ref{fig:double-push-down-1B-picture}).
 
Since each $U^j$ has just a single intersection with $f_2$, performing a $U^j$-Whitney move on $W^j$ would make $W^j$  
cleanly embedded with interior disjoint from all other disks, and with primary multiplicity $0$ and secondary multiplicity equal to $\pm 1$ (by Observation~\ref{sec:observation}).

At this point, we go ahead and do the $U^j$-Whitney moves \textbf{\emph{only on those $W^j$ whose $A^j$ has primary multiplicity $1$}}. (For each $i$ such that $m_i=2$ there are two such $W^j$, corresponding to $B^1$ and $B^2$ each having primary multiplicity $1$ in Step 3.)

In Step~8 below, the remaining $W^j$ having $A^j$ with primary multiplicity $0$ will be made cleanly imbedded in a way that ``collects'' the non-zero secondary multiplicities into canceling pairs so that the resulting metabolic Whitney disks each have secondary multiplicity $0$. 
(For each $i$ such that $m_i=2$ there is one such $W^j$, corresponding to $B^3$ having primary multiplicity $0$ in Step 3.)

At this point $V_0$ is cleanly immersed, framed, and with interior disjoint from $f_2$ and all other disks. The only self-intersections of $V_0$ are the oppositely-signed pair inherited from tubing into the accessory sphere $S^+_{f_2}$ on $f_2$ in (Step~3), and the group elements of these self-intersections are both trivial in $\pi_1(S^4\setminus f)$, so $\lambda(V_0,V_0)=0\in\Z[x^{\pm1}]$.

By construction, the accessory disk $A^{p_0}$ for $V_0$ from Lemma~\ref{lem:p0-q1-w-and-accessory-disks} is framed and cleanly embedded with interior disjoint from all other Whitney and accessory disks.

Note that $V_0$ and $A^{p_0}$ both have vanishing primary and secondary multiplicities (and in fact are both disjoint from $f_2$).

\subsubsection{Step 6}\label{sec:step-6-reduce-primary}
The transfer moves in the previous Step~5 (Figure~\ref{fig:double-push-down-transfer-move}) convert the quadrilaterals $Q^j$ into framed disjointly embedded 
Whitney disks $R^j$ pairing $s_1^j$ and $r_1^j$ (Section~\ref{sec:transfer-move-appendix}).
The interior of each $R^j$ intersects $f_2$ (near $S^+_{f_2}$ and ${S^+_{f_2}}'$ from Step~4), but is disjoint from $f$ and all other disks.
So the interior of $V_1$ can be made to be disjoint from $f$ and all other disks via Whitney moves on $V_1$ guided by the $R^j$.

By construction, $V_1$ is now cleanly immersed, framed and disjoint from all other Whitney disks and from the interiors of all accessory disks. Also, $V_1$ satisfies
$\lambda(V_1,V_1)=0\in\Z[x^{\pm 1}]$, since the only self-intersections of $V_1$ come from tubing into accessory spheres on $f_2$
which contribute oppositely-signed self-intersections with trivial group elements in $\pi_1(S^4\setminus f)$.

The accessory disk $A^{q_1}$ for $V_1$ (from Lemma~\ref{lem:p0-q1-w-and-accessory-disks}) is framed and embedded with interior disjoint from all other Whitney and accessory disks.
From Lemma~\ref{lem:p0-q1-w-and-accessory-disks}, $A^{q_1}$ has primary multiplicity $1$.

Since the only intersections that $V_1$ has with $f_2$ come from the $R^j$-Whitney moves, it follows from Observation~\ref{sec:observation} that $V_1$ has vanishing primary multiplicity. (At this point we do not need to control the secondary multiplicity of $V_1$, since Step~7 will deal with each Whitney disk whose accessory disk has primary multiplicity $1$.)

\noindent{\bf The case $m_i=0$:}
The six steps so far have been described for the case $m_i=2$. If we had started with $W_i$ with $m_i=0$, then the construction is simplified as follows:
Step~1 and Step~2 proceed exactly the same as above.  Observe that in this $m_i=0$ case, at the end of Step~2 the bigon $B$ has primary multiplicity $0$ (the same as $A^-_i$), 
so Step~3 and Step~4 can be skipped entirely, and in Step~5 the transfer move is applied to the single intersection $r_0=B\pitchfork f$. So Step~5 yields just a single $W^j$, with accessory disk $A^j$ having primary multiplicity $0$. (This $A^j$ is formed from $B$, which is a sub-disk of the original $A_i^-$, since we didn't have to do the tubing of Step~3 and Step~4 to reduce the primary multiplicity.) As in Figure~\ref{fig:double-push-down-transfer-move}, the intersection pair $W^j\pitchfork f$ admits $U^j$ intersecting $f_2$ in a single point, and will be made cleanly embedded in Step~8 along with the other $W^j$ having $A^j$ with primary multiplicity $0$.

\noindent{\bf Summary of Steps 1--6:}
Recall that we are assuming the base case that each initial Whitney disk $W_i$ has secondary multiplicity $-1\leq n_i\leq 1$, and each initial accessory disk $A^\pm_i$ has primary multiplicity $0\leq m_i\leq 2$. Simultaneously carrying out the above constructions of Steps~1 through~6 on all the initial $W_i$ with $n_i=\pm 1$ and $m_i=0,2$ has yielded the following:

\begin{enumerate}[(a)]
\item \label{item:summary-initial}
All the initial pairs $W_i,A^+_i$ either had $m_i=1$ to start with, or now have $n_i=0$ from Step~1.  Each of these initial $W_i,A^+_i$ are still cleanly embedded and disjoint from all other Whitney and accessory disks.

\item \label{item:summary-V0}
Each new Whitney disk $V_0$ created in Step~2 and modified in Step~5 is cleanly immersed, with $\lambda(V_0,V_0)=0\in\Z[x^{\pm 1}]$, and has interior disjoint from all other disks. The accessory disk $A^{p_0}$ for $V_0$ is framed and embedded with interior disjoint from all other Whitney and accessory disks.   
By construction both $V_0$ and $A^{p_0}$ have vanishing primary and secondary multiplicities, and are in fact disjoint from $f_2$. 

\item \label{item:summary-V1}
Each new Whitney disk $V_1$ created in Step~2  and and modified in Step~6 is cleanly immersed, with $\lambda(V_1,V_1)=0\in\Z[x^{\pm 1}]$. The accessory disk $A^{q_1}$ for $V_1$ is framed and embedded with interior disjoint from all other Whitney and accessory disks.
The primary multiplicity of $V_1$ is $0$ by Observation~\ref{sec:observation}, and the primary multiplicity of $A^{q_1}$ is $1$ by Lemma~\ref{lem:p0-q1-w-and-accessory-disks}.

\item \label{item:summary-new-pairs-primary-1}
The new disk pairs $W^j,A^j$ created in Step~5 \textbf{such that $A^j$ has primary multiplicity 1} are (after the $U^j$-Whitney moves) all cleanly embedded and disjoint from all other disks. These $W^j$ each have primary multiplicity $0$,
and each $W^j$ also has secondary multiplicity $\pm 1$, coming from the unit primary multiplicities of the $U^j$ which were used to get $W^j\pitchfork f=\emptyset$ (by Observation~\ref{sec:observation}).

\item \label{item:summary-new-pairs-primary-0}
Those new disk pairs $W^j,A^j$ created in Step~5 \textbf{such that $A^j$ has primary multiplicity 0} are embedded and disjoint from all other disks but still have $W^j\pitchfork f$ paired by $U^j$.
\end{enumerate}

So the \textbf{remaining ``problems''} to be dealt with are 
\begin{itemize}
\item 
the Whitney disks which have possibly non-zero secondary multiplicity and have accessory disks with primary multiplicity $1$, as in items~(\ref{item:summary-initial}),~(\ref{item:summary-V1}) and (\ref{item:summary-new-pairs-primary-1}), and

\item
the $W^j$ which still intersect $f$ and have accessory disks with primary multiplicity $0$, as in 
item~(\ref{item:summary-new-pairs-primary-0}).
\end{itemize}

These problems will be taken care of in Step~7 and Step~8 respectively.

\subsubsection{Step~7}\label{sec:step-7-double-boundary-twist}
In this step we use a local construction to kill the secondary multiplicity of each Whitney disk whose accessory disk has unit primary multiplicity. At this point in the proof such Whitney disks are either from the initial collection with arbitrary secondary multiplicity (as in \ref{item:summary-initial}), or are a $V_1$ created in Step~2 (as in \ref{item:summary-V1}), or are from the $W^j$ created in Step~5 with unit secondary multiplicity (as in \ref{item:summary-new-pairs-primary-1}). In each of these cases the corresponding accessory disk is framed and cleanly embedded, with interior disjoint from all other disks. 

Let $W,A$ be one of these Whitney disk--accessory disk pairs with 
$A$ having primary multiplicity $1$. This step will describe how to replace $W$ by a new Whitney disk $V$ with secondary multiplicity $0$ such that 
$\partial V=\partial W$, with $V$ supported near $W\cup A$ and disjoint from the interior of $A$. Since all the accessory disks at this point are disjointly embedded and with interiors disjoint from all Whitney disks, this construction can be carried out simultaneously to replace all such $W,A$ by $V,A$ without creating any new intersections among disks.

First we construct $V$ in the case that $A$ is a positive accessory disk, and has just a single positive intersection with $f_2$.
The construction described in Figure~\ref{fig:double-boundary-twist-Whitney-move-on-accessory-disk-1}
and Figure~\ref{fig:double-boundary-twist-Whitney-move-on-accessory-disk-2} shows how to add \mbox{$+(1-x)$} to $\lambda(W,f_2)$ mod $I^2$ by changing the interior of $W$ by two boundary-twists and a Whitney move guided by a Whitney disk formed by a parallel copy of $A$. Since the boundary-twists are oppositely-handed, $W$ is still framed (for details on the boundary-twisting construction see e.g.~\cite[sec.1.3]{FQ}).

\begin{figure}[ht!]
         \centerline{\includegraphics[scale=.475]{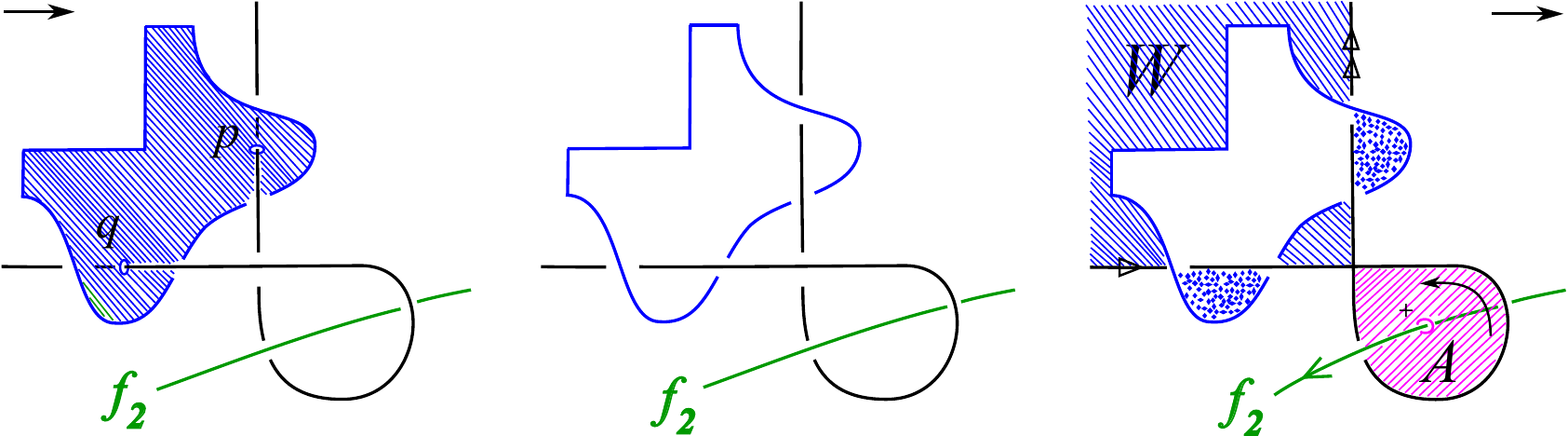}}
         \caption{Near the framed embedded positive accessory disk $A$ for $W$: A left-handed boundary-twist along $\partial_-W$ and a right-handed boundary-twist along $\partial_+W$ creates a pair of oppositely-signed intersections $p$ and $q$ in $W\pitchfork f$.}
         \label{fig:double-boundary-twist-Whitney-move-on-accessory-disk-1}
\end{figure}
\begin{figure}[ht!]
         \centerline{\includegraphics[scale=.55]{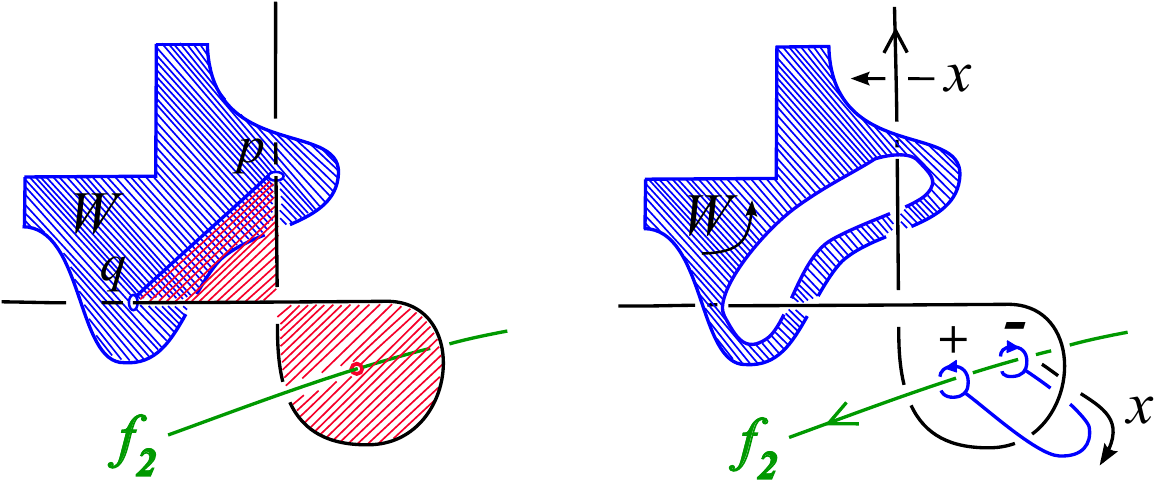}}
         \caption{Both of these pictures correspond to the right-most $3$-dimensional slice of $4$--space shown in Figure~\ref{fig:double-boundary-twist-Whitney-move-on-accessory-disk-1}.  Left: The intersections $p$ and $q$ created by the boundary-twists can be paired by the red Whitney disk formed from a parallel of $A$ (as in Section~\ref{sec:w-move-on-acc-disks}). 
         Right: The result of performing the red Whitney move on $W$. The Whitney bubble added to $W$ by the Whitney move is suppressed from view, except for an arc (blue) connecting the new intersection points between $W$ and $f_2$. The intersections $p,q\in W\pitchfork f$ have been eliminated, and (as per Observation~\ref{sec:observation}) 
the change in $\lambda(W,f_2)$ mod $I^2$ is equal to $+(1-x)$.}
         \label{fig:double-boundary-twist-Whitney-move-on-accessory-disk-2}
\end{figure}

The construction is supported in an arbitrarily small neighborhood of $A$, and also leaves a smaller neighborhood of $A$ unchanged; so this construction can be iterated to increase $\lambda(W,f_2)$ mod $I^2$ by $n(1-x)$ for any $n\geq 1$. If the left-handed and right-handed boundary-twists are switched in the construction, or if the construction is carried out using a negative accessory disk, then $\lambda(W,f_2)$ mod $I^2$ is changed by $-(1-x)$, so the secondary multiplicity of $W$ can be changed by any $n\in \Z$ and the desired $V$ can be created in this case. 
Note that the construction of $V$ from $W$ is supported near $W\cup A$, so no new intersections are created between any Whitney or accessory disks. 

Now consider the general case where $A$ has possibly more than one intersection with $f_2$ (corresponding to $P\cdot S_A$ terms in $f_2$, for $P\in I^2$ as in Lemma~\ref{lem:delta-w-disks}, and/or oppositely-signed intersection pairs created Step~3):
Each intersection between $A$ and $f_2$ will contribute
a term $\pm (1-x)$ mod $I^2$ via the construction shown in Figures~\ref{fig:double-boundary-twist-Whitney-move-on-accessory-disk-1}
and~\ref{fig:double-boundary-twist-Whitney-move-on-accessory-disk-2}. Since $\lambda(A,f_2)=1$ modulo $I$, the net affect of each iteration of the construction is still to change $\lambda(W,f_2)$ by $\pm (1-x)$ modulo $I^2$, again by Observation~\ref{sec:observation}, so the secondary multiplicity of $W$ can be reduced to zero, yielding $V$ with secondary multiplicity zero and with the same vanishing primary multiplicity as $W$.

\begin{rem}\label{rem:primary-mult-1-lemma-not-contradiction}
The reader might wonder why this construction of $V$ does not contradict the fact (from Lemma~\ref{lem:algtop}) that
$\lambda(g,f_2)\in I^2$ for any $g:S^2\hookrightarrow S^4\setminus f$? After all, it seems that such a $g$ could be formed from the two framed and cleanly immersed Whitney disks $W$ and $V$ with $\partial W=\partial V$ in the above construction, yielding $\lambda(g,f_2)=\pm(1-x)\notin I^2$? But the subtlety here is that although $W$ and $V$ are each framed, they have different (relative) framings along each of the positive and negative boundary arcs, so perturbing $W\cup V$ to be disjoint from $f$ would create intersections near $\partial W=\partial V$, and eliminating these intersections would effectively ``undo'' the construction.

 \end{rem}
 
The general effect of the double boundary-twisting construction in this step is summarized by the following lemma:
\begin{lem}\label{lem:double-boundary-twist}
Suppose $W,A$ is a cleanly immersed Whitney disk--accessory disk pair on $f$ in standard position such that $A$ has primary multiplicity $m$. Then for any integer $k$, there exists a cleanly immersed Whitney disk $V$ supported in a neighborhood of $W\cup A$, with $\partial V=\partial W$, such that
$V$ has the same primary multiplicity as $W$, and the secondary multiplicity of $V$ differs from the secondary multiplicity of $W$ by $km$.$\hfill\square$
\end{lem}

\subsubsection{Step~8}\label{sec:step-8-pair-secondary-mults}
After Step~7, all Whitney disks have vanishing primary and secondary multiplicities except for those $W^j$ created in Step~5 whose accessory disks $A^j$ have primary multiplicity $0$ (cf. \ref{item:summary-new-pairs-primary-0}). Each of these $W^j$ still has a single pair of interior intersections with 
$f$ which are paired by a Whitney disk $U^j$.  Since each $U^j$ has just a single interior intersection with $f_2$, doing a $U^j$-move on $W^j$ would make $W^j$ cleanly embedded with secondary multiplicity $\pm 1$ (by Observation~\ref{sec:observation}), so this final step of the construction will first collect these $U^j$ in ``canceling'' pairs so that doing the $U^j$-moves will yield cleanly embedded $W^j$ with secondary multiplicity $0$.

To start this construction we need to prove the following\\
\textbf{Claim:} There are an even number of these $U^j$.

\begin{proof}[Proof of Claim:]
Doing the $U^j$-moves on these $W^j$ would yield cleanly embedded pairs $W^j,A^j$.  Just for this proof of the claim denote by  
$\{W_i,A_i\}$, the collection consisting of these resulting cleanly embedded pairs $W^j,A^j$ together with the initial  Whitney disk--accessory disk pairs and
all the other Whitney disk--accessory disk pairs constructed so far in the previous steps. 

Note that the number of $U^j$ in the claim equals the number of pairs in $\{W_i,A_i\}$ such that $W_i$ has non-zero secondary multiplicity, and all of these non-zero secondary multiplicities are $\pm 1$. 
Also, the corresponding $A_i$ all have primary multiplicity $0$.

Convert all the pairs in $\{W_i,A_i\}$ into accessory disk pairs $A_i^\pm$ (see Section~\ref{sec:A-disks-from-W-disks}), and then use these $A_i^\pm$ to construct a basis of accessory spheres $S_{A_i^\pm}$ for $\pi_2(S^4\setminus f)$. Since the $A_i$ are all framed and pairwise disjointly embedded, and the $W_i$ all have vanishing $\Z[x^{\pm 1}]$-intersections, this basis of accessory spheres has the same $\Z[x^{\pm 1}]$-intersections 
as the standard accessory spheres in Lemma~\ref{lem:accessory-sphere-basis}.

Write $f_2$ in terms of the $S_{A_i^\pm}$:
$$
f_2=\sum \alpha_i^+\cdot S_{A_i^+}+\alpha_i^-\cdot S_{A_i^-},
$$
with
$$
\alpha_i^\pm=m_i^\pm+n_i^\pm(1-x)+q_i^\pm z \quad\mbox{ mod } I^3
$$
for integers $m_i^\pm,n_i^\pm,q_i^\pm$, where $m_i^\pm$ and $n_i^\pm$ are the primary and secondary multiplicities of $A_i^\pm$, respectively,
and $z:=(2-x-x^{-1})$.

Since the $W_i$ all had primary multiplicity $0$, the construction of Section~\ref{sec:A-disks-from-W-disks} has yielded $A_i^\pm$ such that $m_i^+=m_i^-$ for each $i$. Denoting these primary multiplicities by $m_i$, we have the following formula for the coefficient of $z^2$ in the expansion of $\lambda(f_2,f_2)$ in powers of $z$: 

$$
\sum[((n_i^+)^2-(n_i^-)^2)+(n_i^+ - n_i^-)m_i  +2m_i(q_i^+-q_i^-)]
$$
This sum must vanish since $\lambda(f_2,f_2)=0$ from the Kirk invariant hypothesis of Proposition~\ref{prop:I-squared}. Note that $n^+_i-n^-_i$  is the secondary multiplicity of $W_i$,
because $\lambda(W_i,f_2)=\lambda(A_i^+,f_2)-\lambda(A_i^-,f_2)$ by construction. And recall that for $m_i= 1$, we have $n^+_i-n^-_i=0$; while for $m_i=0$, we have $n^+_i-n^-_i=\pm 1$ (and $0, 1$ are the only values of $m_i$). So setting the sum equal to zero modulo $2$ shows that $\sum (n_i^+-n_i^-)$ is even, which proves the claim.
\end{proof}

\begin{rem}\label{rem:tau-zero-claim}
An alternate proof of this claim can be made using Lightfoot's main result in \cite{Lightfoot}. 
\end{rem}

\begin{figure}[ht!]
         \centerline{\includegraphics[scale=.3]{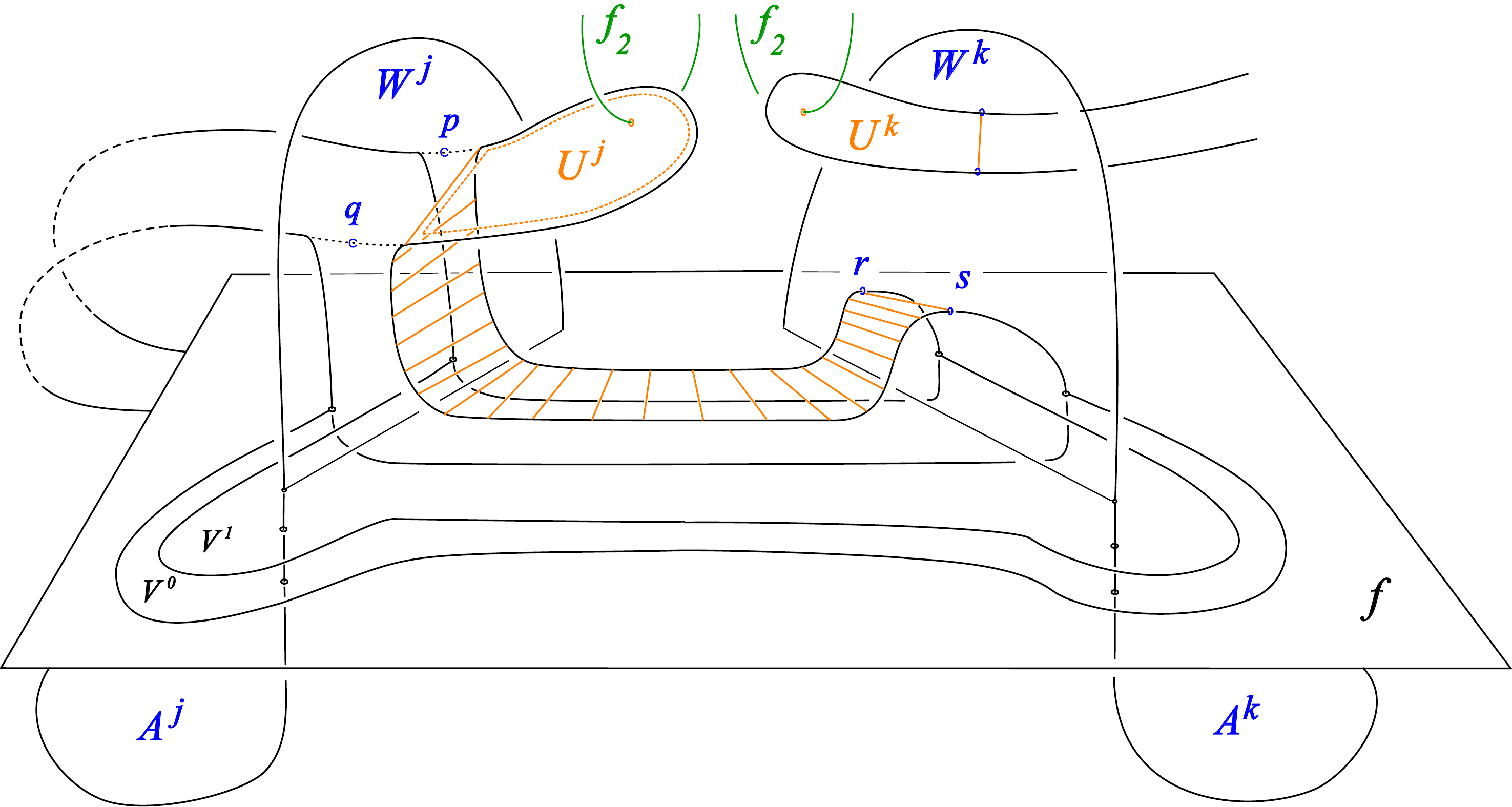}}
         \caption{The \emph{double transfer move}: Two parallel finger moves exchange $\{p,q\}= W^j\pitchfork f$ paired by the Whitney disk $U^j$ for $\{r,s\}\subset W^k\pitchfork f$ paired by an extension of $U^j$ by a band.
        Two new local Whitney disks $V^0$ and $V^1$ pair the self-intersections of $f$ created by the finger moves.}
         \label{fig:U-transfer-move-1}
\end{figure}

\subsection{The double transfer move}
Having established the claim that there are an even number of the $W^j$ created in Step~5 whose accessory disks $A^j$ have primary multiplicity $0$, each with a single $U^j$ pairing $f\pitchfork W^j$, Step~8 will proceed using a construction which ``transfers'' any $U^j$ off of its $W^j$ and onto a different $W^k$. As a result of this \emph{double transfer move}, $W^j$ will be cleanly embedded and disjoint from $f_2$, and after making $W^k$ cleanly embedded via  both the ``new'' $U^j$-move and the original $U^k$-move, $W^k$ will have vanishing secondary multiplicity (and primary multiplicity). 
This
double transfer move of a Whitney disk from one Whitney disk to another is an enhancement of the move transferring a single point from one Whitney disk to another as above in Figure~\ref{fig:double-push-down-transfer-move} of Step~5 (see also Section~\ref{sec:transfer-move-appendix}). 
At certain points it will be necessary to keep careful track of orientations, and discussion of some of these details will be deferred to Section~\ref{sec:double-transfer-orientations} so as not to break the flow of the geometric descriptions.

The construction starts as described in Figure~\ref{fig:U-transfer-move-1}, and is carried out simultaneously to all such pairs $W^j,W^k$: The intersections $\{p,q\}=W^j\pitchfork f$ paired by $U^j$ are exchanged for $\{r,s\}\subset W^k\pitchfork f$ via a parallel pair of finger moves guided by an embedded arc from $\partial W^j$ to $\partial W^k$ in $f$. This new pair $r,s$ admits an embedded Whitney disk formed by extending $U^j$ (minus a collar in $U^j$ of $\partial U^j\subset W^j$) by an embedded band near the arc (shown in the figure as the orange band just above the horizontal sheet of $f$).
This Whitney disk, which we continue to denote $U^j$, still has only a single interior intersection with $f_2$.

The parallel pair of finger moves also creates two new pairs of self-intersections of $f$ which can be paired by embedded Whitney disks $V^0$ and $V^1$, which each have a pair of interior intersections with $f$ as shown in Figure~\ref{fig:U-transfer-move-1} (where both $V^0$ and $V^1$ hang down below the horizontal sheet of $f$). As is evident in Figure~\ref{fig:U-transfer-move-1}, the boundaries $\partial V^0$ and $\partial V^1$ intersect the accessory disk boundaries $\partial A^j$ and $\partial A^k$, but these intersections can be eliminated by adding half-tubes to $A^j$ and $A^k$ near $f$ as in Figure~\ref{fig:U-transfer-move-2}.
 Adding these half-tubes creates two pairs of intersections between $f$ and each of $A^j$ and $A^k$ near $r,s$ that admit Whitney disks parallel to $U^j$.
\begin{figure}[ht!]
         \centerline{\includegraphics[scale=.3]{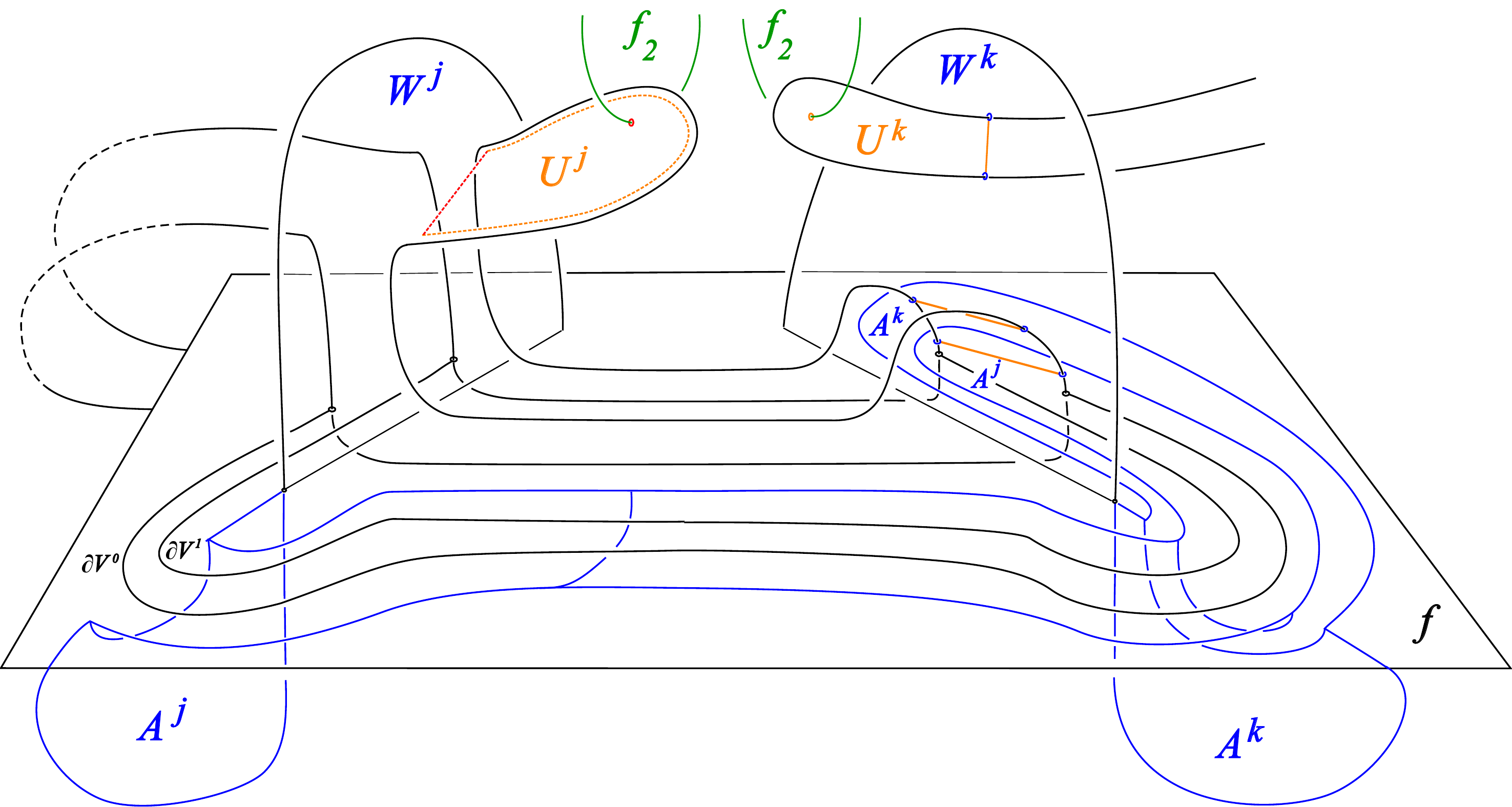}}
         \caption{After adding half-tubes to make all of $\partial V^0$, $\partial V^1$, $\partial A^j$ and $\partial A^k$ disjointly embedded: These half-tubes appear `below' the present sheet of $f$ near $A^j$ and $A^k$, but then rotate (through the future) `above' the horizontal present sheet of $f$ near $W^k$, where they each have a pair of transverse intersections with $f$ which can be paired by parallels of the Whitney disk $U^j$ pairing $r,s$ in Figure~\ref{fig:U-transfer-move-1} (here $r$ and $s$ are suppressed from view). These parallels of $U^j$ are suppressed from view except for the boundary arcs lying in $A^j$ and $A^k$, (and only part of $U^j$ is shown). Pushing the interiors of $V^0$ and $V^1$ into the past keeps them disjoint from the half-tubes.}
         \label{fig:U-transfer-move-2}
\end{figure}

Now do the Whitney moves on these two parallel copies of $U^j$ to get $A^j$ and $A^k$ cleanly embedded, and do both the $U^k$-move and $U^j$-move to get $W^k$ cleanly embedded (but keep the same notation for $A^j$, $A^k$ and $W^k$).

At this point $W^j$ has vanishing primary and secondary multiplicities (in fact $W^j\cap f_2=\emptyset$), and is equipped with the cleanly embedded accessory disk $A^j$. Also $W^j\cup A^j$ is disjoint from all other disks by construction.

As explained below in Section~\ref{sec:double-transfer-orientations} the freedom in choosing the guiding arc between $\partial W^j$ and $\partial W^k$ allows for control of whether or not the orientation (relative to $W^k$) of the resulting $U^j$ is preserved or flipped by the double transfer move, and as a result it can be arranged that the effect of then doing the $U^j$-move on $W^k$ creates a canceling contribution to that of the $U^k$-move on the secondary multiplicity of the resulting cleanly embedded $W^k$.
So, assuming this, we also now have that $W^k$ has vanishing primary and secondary multiplicities, and is equipped with the cleanly embedded accessory disk $A^k$. Also $W^k\cup A^k$ is disjoint from all other disks by construction.

Before converting $V^0$ and $V^1$ into metabolic Whitney disks with vanishing primary and secondary multiplicities we first need to construct appropriate accessory disks $A^0$ and $A^1$ for $V^0$ and $V^1$. Recall that the $\Z[x^{\pm1}]$-intersections among metabolic accessory disks can be arbitrary (Definition~\ref{def:metabolic-collection}). We describe the construction of $A^0$; the same steps yield $A^1$:

Choose an embedded accessory circle $a^0\subset f$ for the self-intersection of $f$ at the corner of $V^0$ near $\partial W$ such that $a^0$ is disjoint from all other boundaries of Whitney and accessory disks. Since $\pi_1(S^4\setminus f)\cong\Z$, there exists an immersed disk $A^0$ bounded by $a^0$ such that the interior of $A^0$ is disjoint from $f$. 
Then make the interior of $A^0$ disjoint from each Whitney disk (including $W^j$ and $W^k$) by tubing $A^0$ into parallel copies of the dual accessory sphere as needed.  Observe that the interior of $A^0$ can be assumed to be disjoint from $V^0$ and $V^1$ (even though they don't yet have accessory spheres), since $V^0$ and $V^1$ are supported near arcs in $f$.
(As usual, we do not rename $A^0$ after changing its interior by adding these spheres.) 

Now make $A^0$ disjoint from $A^j$ and $A^k$ by adding parallel copies of the Whitney spheres $S_{W^j}$ and $S_{W^k}$ to $A^0$ as needed. This disjointness will allow us to form secondary Whitney disks from  $A^j$ and $A^k$ that will be used to get $f$ disjoint from the interiors of $V^0$ and $V^1$.

Simultaneously carry out the same steps to construct an accessory disk $A^1$ for $V^1$ such that $A^1$ is disjoint from $A^j$ and $A^k$ as well as all Whitney disks (except the corner-point $\partial A^1\cap\partial V^1$).

It remains to clean up  $V^0$ and $V^1$. To get the interiors of $V^0$ and $V^1$ disjoint from $f$, form Whitney disks $U^0$ for $V^0\pitchfork f$ and $U^1$ for $V^1\pitchfork f$ by connecting parallels of (the framed cleanly embedded) $A^j$ and $A^k$ as in Figure~\ref{fig:transfer-move-3-double-A-disks}. (Note that the interior of $V^0$ can be perturbed to be disjoint from $U^1$.) 
Here we need that the interior intersections $V^0$ and $V^1$ have with $f$ are of opposite signs,
as can be assumed by the discussion in Section~\ref{sec:double-transfer-orientations}.

\begin{figure}[ht!]
         \centerline{\includegraphics[scale=.25]{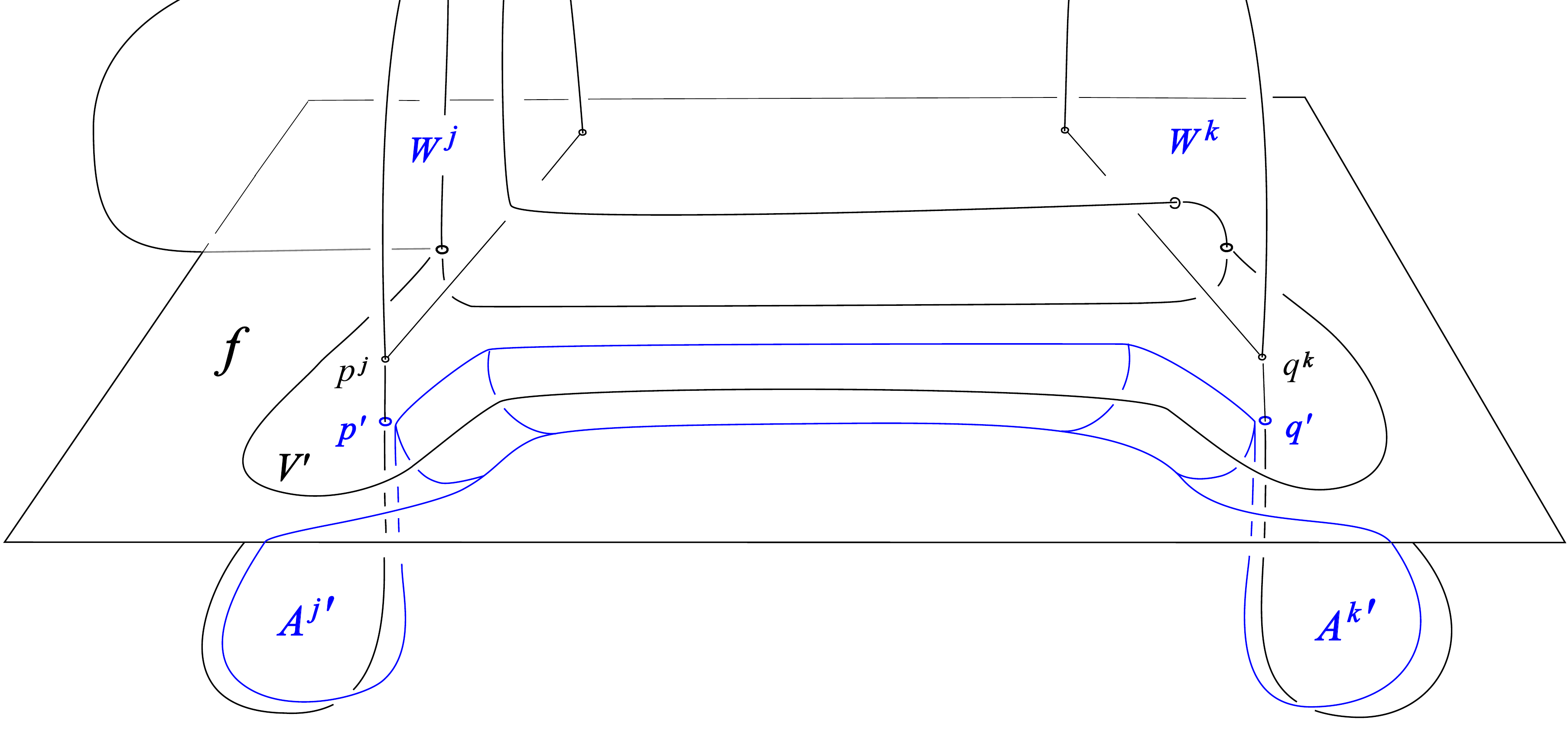}}
         \caption{In the setting of Figures~\ref{fig:U-transfer-move-1} and~\ref{fig:U-transfer-move-2}, here with $V',U'$ denoting $V^0,U^0$ and/or $V^1,U^1$: The Whitney disk $V'$ intersects $f$ in $p'$ and $q'$, which are paired by a secondary Whitney disk $U'$ (blue) formed by combining parallels ${A^j}',{A^k}'$ of the accessory disks $A^j,A^k$ for $p^j,q^k\in f \pitchfork f$.}
         \label{fig:transfer-move-3-double-A-disks}
\end{figure}

Note that $U^0$ and $U^1$ each have primary multiplicity $0$, since $A^j$ and $A^k$ both have primary multiplicity $0$ 
(the Whitney moves done on $A^j$ and $A^k$ did not affect their primary multiplicities by Observation~\ref{sec:observation}).

Now doing the $U^0$- and $U^1$-Whitney moves on $V^0$ and $V^1$ makes $V^0$ and $V^1$ both cleanly embedded with vanishing primary and secondary multiplicities (by Observation~\ref{sec:observation}).

Since $A^j$ and $A^k$ are disjoint from the new $A^0$ and $A^1$ by construction, and were already disjoint from all previous Whitney disks and accessory disks,
 it follows that $V^0$ and $V^1$ have vanishing $\Z[x^{\pm 1}]$-intersections with all Whitney and accessory disks, as required for metabolic Whitney disks.

To complete Step~8 of the construction it remains to check that this double transfer move can be simultaneously carried out for all pairs $W^j,A^j$ and $W^k,A^k$ created in Step~5 with $A^j$ and $A^k$ having primary multiplicity $0$. To see this, observe that the parallel finger moves in Figure~\ref{fig:U-transfer-move-1} (and the bands added to $U^j$ and parallels of $U^j$)
are supported near an arc $a^{jk}$ in $f$ from $W^j$ to $W^k$ (together with short arcs in $W^j$ and $W^k$). Such an $a^{jk}$ always exists because the complement in $f$ of the disjointly embedded boundaries of all the Whitney disk-accessory disk pairs and all $U^j$ is \emph{connected} (the inverse image of each union of a Whitney disk boundary with the boundary of one of its accessory disks is an embedded arc in the domain 2-sphere of $f$, as is the inverse image of each $U^j$, and these arcs are all pairwise disjoint).

\subsubsection{Controlling orientations and signs in the double transfer move}\label{sec:double-transfer-orientations}

In the setting of Step~8, we have Whitney disks $W^j$ such that each $W^j\pitchfork f$ is paired by an embedded Whitney disk $U^j$ which contains only a single interior intersection with $f_2$. The link map $(f,f_2)\imra S^4$ is oriented, and by convention the positive generator $x$ of $\pi_1(S^4\setminus f)\cong\Z$ is represented by any positive meridian to $f$.
Fix orientations on the $W^j$, and then orient the $U^j$ by taking the boundary arc $\partial U^j\subset f$ to run from the negative to the positive intersection in $W^j$. With this convention, if $\lambda(U^j,f_2)=\pm x^N$, then performing the $U^j$-move on $W^j$ makes $W^j$ cleanly embedded with $\lambda(W^j,f_2)=\pm(1-x)x^N$, with the signs preserved. That is, the sign $\pm$ of the intersection between $U^j$ and $f_2$ equals the sign of the secondary multiplicity $\pm 1$ of $W^j$ after the $U^j$-move.

So given such $W^j$ and $W^k$ (with orientations as in the previous paragraph) we want to control the double transfer move so that after $U^j$ has been moved onto $W^k$ the sign of $U^j\cap f_2$ is \emph{opposite} to the sign of $U^k\cap f_2$. Then doing the $U^j$- and $U^k$-moves on $W^k$ will result in $W^k$ having secondary multiplicity $0$. We also need to arrange that the interior intersections $V^0$ and $V^1$ have with $f$ are of opposite signs. 
We describe next how this can be accomplished with the configuration shown in Figure~\ref{fig:U-transfer-move-1}.

The choice of guiding arc $a^{jk}\subset f$ for the double transfer move determines whether the sign of $U^j\cap f_2$ is preserved or switched after $U^j$ transferred onto $W^k$, depending on
which boundary arcs $\partial_\pm W^j$ or $\partial_\pm W^k$ are joined by $a^{jk}$, and also on which sides in $f$ of $\partial_\pm W^j$ or $\partial_\pm W^k$ are connected by $a^{jk}$.
Also, observe from Figure~\ref{fig:U-transfer-move-1} that if the signs of the self-intersections of $f$ at $A^j$ and $A^k$ are of opposite sign, then so are the interior intersections that $V^0$ and $V^1$ have with $f$. 
By the construction in Section~\ref{sec:A-disks-from-W-disks}, we may assume that the signs of the self-intersections of $f$ at $A^j$ and $A^k$ are of opposite sign, so by choosing $a^{jk}$ appropriately it can be arranged that the sign of $U^j\cap f_2$ is opposite to the sign of $U^k\cap f_2$ after the double transfer move, and that the interior intersections that $V^0$ and $V^1$ have with $f$ are of opposite sign, with the configuration as in Figure~\ref{fig:U-transfer-move-1}.


\subsection{Checking metabolic properties}\label{sec:check-metabolic-disks}

The constructions in Step~7 and Step~8 resolved the remaining problems listed in items~(\ref{item:summary-initial}),~(\ref{item:summary-V1}) and (\ref{item:summary-new-pairs-primary-1}), and in 
item~(\ref{item:summary-new-pairs-primary-0}) of our summary of steps 1-6 before Section~\ref{sec:step-7-double-boundary-twist} without affecting any other disks.
So the only remaining property of the $W_i,A_i$ constructed so far that needs to be adjusted is that the definition of a metabolic collection requires \emph{positive} accessory disks. This adjustment can be carried out using the construction in Section~\ref{sec:A-disks-from-W-disks} which forms a positive accessory disk out of a negative accessory disk together with a parallel of the Whitney disk.

This completes the proof of Proposition~\ref{prop:I-squared} in the special case that each accessory disk $A^\pm_i$ has primary multiplicity $0\leq m_i\leq 2$, and each Whitney disk $W_i$ either has secondary multiplicity $-1\leq n_i\leq 1$, or has arbitrary $n_i$ if $m_i=1$.
To complete the proof we explain next how the eight steps of the construction can be extended to handle accessory disks with arbitrary (positive) primary multiplicities and Whitney disks with arbitrary secondary multiplicities.


\subsection{Extending the construction to the general case}\label{sec:general-case-reduce-primary}
To handle the case that the $A_i$ from the initial collection $\{W_i,A_i\}$ may have primary multiplicity $m_i>2$, still assuming that the $W_i$ have secondary multiplicity $-1\leq n_i\leq 1$, observe first that the above construction can proceed ``as is'' through Step~1 and Step~2 with $2$ replaced by $m_i$. (So instead of two, there would be $m_i$ strands of $f_2$ visible in the lower left of Figure~\ref{fig:double-push-down-1-solid-picture} and the upper left of Figure~\ref{fig:double-push-down-2-picture}.)
In Step~3, the construction (Figure~\ref{fig:tube-V0-into-accessory-sphere-1AandB}) replaces $r_0\in V_0\pitchfork f$ admitting framed embedded $B$ with the three $r_0^j\in V_0\pitchfork f$ admitting parallel framed embedded $B^j$ having primary multiplicities decreased by $1$ or $2$ from that of $B$. So iterating 
this step as needed on the $r_0^j$ will eventually lead to many intersections in $V_0\pitchfork f$ all admitting framed embedded generalized
accessory disks with primary multiplicities $1$ or $0$.  After applying the same iterations to $s_1^j\in V_1\pitchfork f$ in Step~4, the rest of the construction proceeds as in the $m_i=2$ case above:
We have as many disjoint framed embedded parallels of $B$ and $Q$ as we want, since they all come from the original standard accessory disk $A_i^-$. And in this $m_i>2$ case the construction at Step~5 will yield more of the new $W^j,A^j$ pairs, and there will be more $R^j$-Whitney moves on $V_1$ in Step~6, but these last two steps can proceed as before. The local construction of Step~7 proceeds as before on each Whitney disk whose accessory disk has primary multiplicity $1$, there will just be more of them now. And similarly, both the claim and the double transfer moves in Step~8 proceed as before, now applied to more Whitney disks.

\begin{figure}[ht!]
         \centerline{\includegraphics[scale=.325]{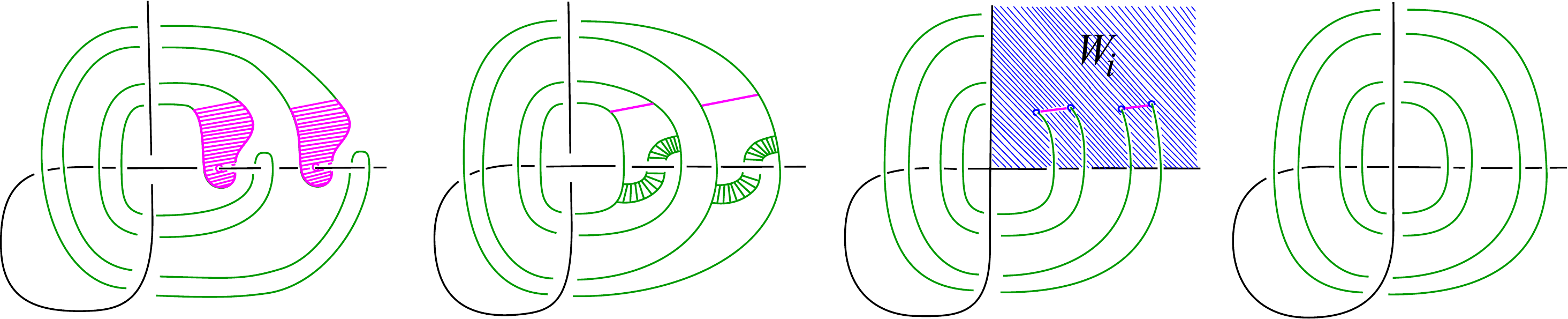}}
         \caption{This figure corresponds to the bottom row in Figure~\ref{fig:x-1-term-A-Delta}, but here with two intersection pairs $\pm(1-x)\cdot S_{A_i^-}\cap W_i$ (the case $n_i=\pm 2$), each admitting a framed embedded
          Whitney disk (purple) having only a single interior intersection with $f$.}
         \label{fig:x-1-term-A-delta-double}
\end{figure}

\begin{figure}[p]
\centering
\includegraphics[scale=.425]{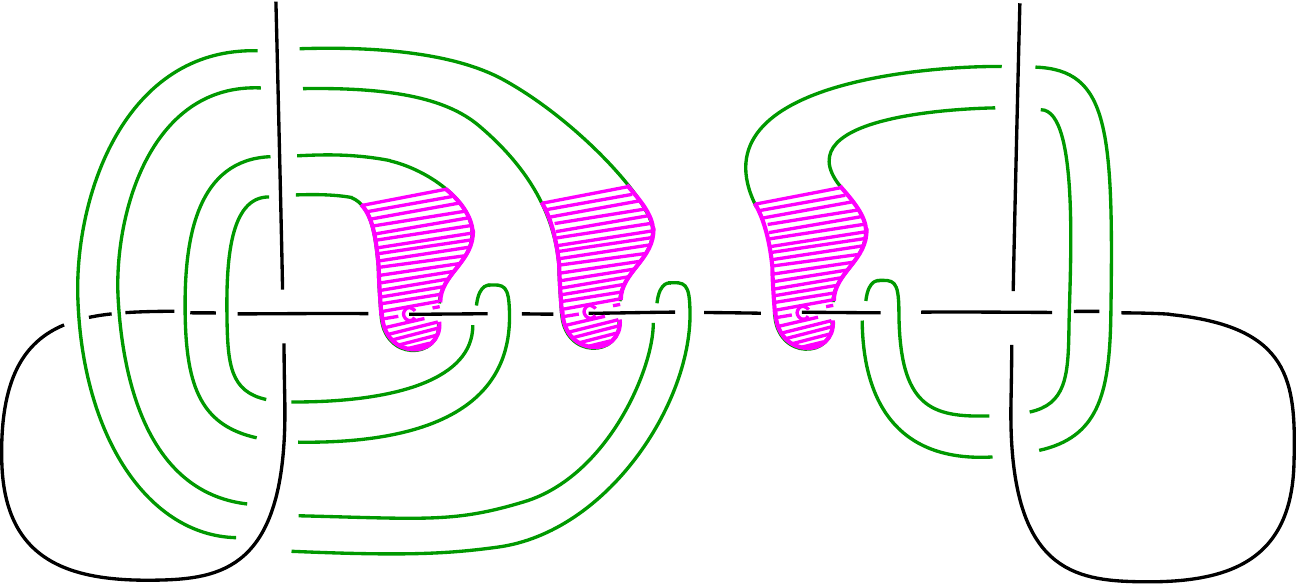}
\caption{This figure corresponds to the left-most picture in Figure~\ref{fig:x-1-term-A-delta-double}, but showing here the case of an additional positive accessory sphere pair along with the two negative accessory sphere intersection pairs.}
         \label{fig:x-1-term-A-delta-triple}
\includegraphics[scale=.4]{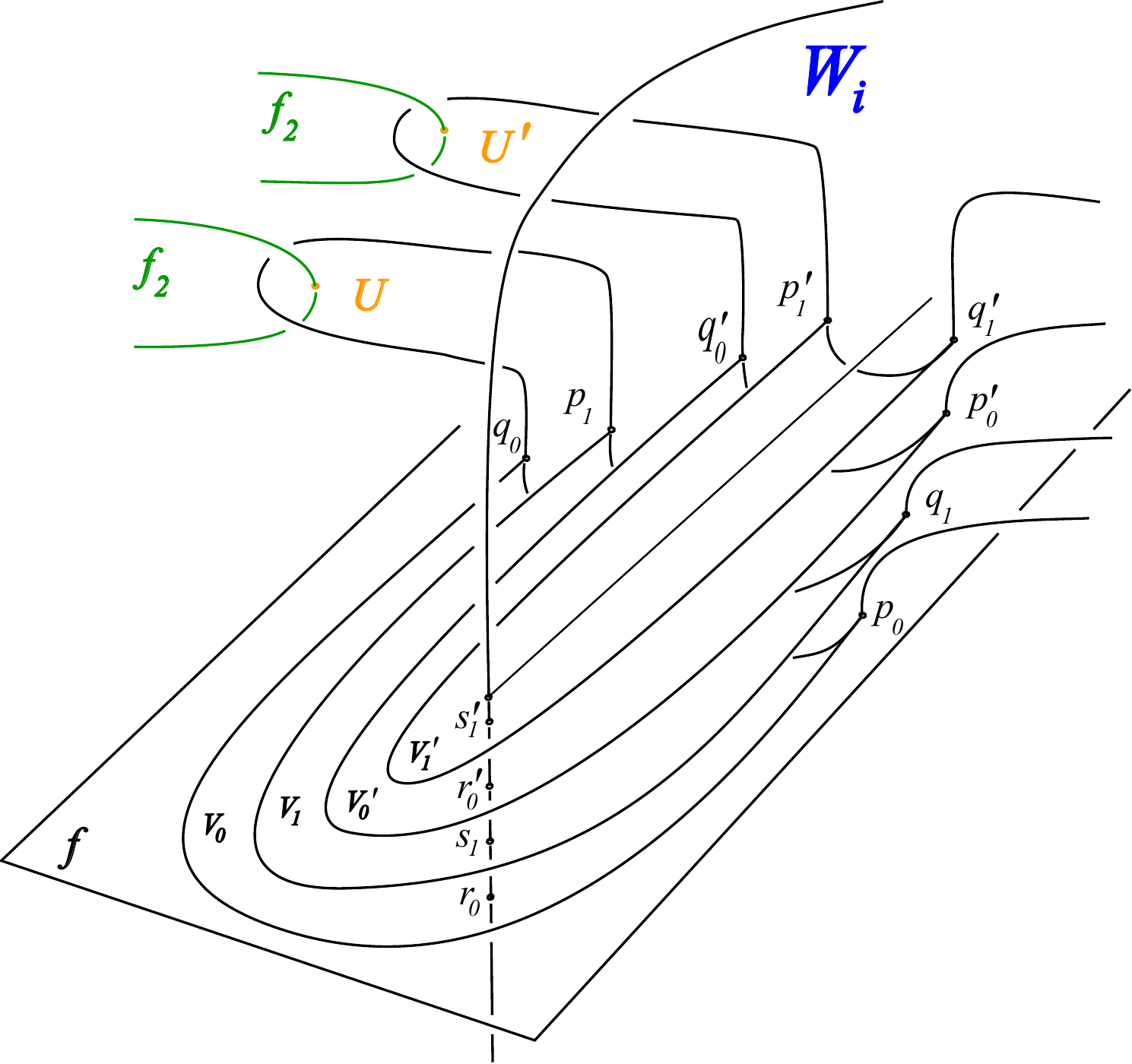}
\caption{After simultaneously applying the constructions of Figures~\ref{fig:x-1-term-A1},~\ref{fig:x-1-term-B} and~\ref{fig:x-1-term-A0-A1-U} to Figure~\ref{fig:x-1-term-A-delta-double}, the resulting four pairs of self-intersection pairs of $f$ admit two pairs of nested Whitney disks as in Figure~\ref{fig:double-push-down-1-solid-picture}. Shown here is the case $n_i=2$, but nested Whitney disks exist for any number of pairs of pairs; e.g.~starting with Figure~\ref{fig:x-1-term-A-delta-triple} would yield three nested pairs of Whitney disks. Recall that by construction we take these nested Whitney disks around the negative self-intersection of $f$, even when starting with intersections including pairs coming from positive accessory spheres $\pm(1-x)\cdot S_{A_i^+}\cap W_i$ (as in Figure~\ref{fig:x-1-term-A-delta-triple}).}
         \label{fig:double-push-down-1B-picture}
\end{figure}
Now considering the case that the secondary multiplicity of an initial $W_i$ has absolute value $|n_i|>1$, there will be $|n_i|$ intersection pairs $\pm(1-x)\cdot S_{A_i^\pm}\cap W_i$, each admitting a framed embedded Whitney disk (like $\Delta$ in Figure~\ref{fig:x-1-term-A1}) which has only a single interior intersection with $f$. See Figure~\ref{fig:x-1-term-A-delta-double} for the case of two negative accessory sphere pairs, and Figure~\ref{fig:x-1-term-A-delta-triple} for an additional positive accessory sphere pair. It is not difficult  to see that Step~1 (including Lemma~\ref{lem:p0-q1-w-and-accessory-disks}) can be applied simultaneously to each intersection pair disjointly: Through Figure~\ref{fig:x-1-term-B}, the construction of Step~1 is supported in a neighborhood of the Whitney disk $\Delta$ together with an arc along $W_i$ running from $\partial \Delta$ to $\partial_+ W_i$ ($\Delta$ itself is contained in the neighborhood of an arc). And the accessory and Whitney disks of Lemma~\ref{lem:p0-q1-w-and-accessory-disks} in Figure~\ref{fig:x-1-term-A0-A1-U} can be constructed ``side by side'' along the boundary of $W_i$ when starting with more pairs $\pm(1-x)\cdot S_{A_i^\pm}\cap W_i$ as in Figure~\ref{fig:x-1-term-A-delta-double}. To visualize this it may be helpful to consider the $n_i=2$ case and compare the input to Step~1 shown in Figure~\ref{fig:delta-first-example} and Figure~\ref{fig:x-1-term-A-delta-double} with the result that would be input to Step~2 shown in Figure~\ref{fig:double-push-down-1B-picture}.



Considering the situation at the start of Step~2, it is clear that an additional pair of self-intersection pairs $p_0',q_0'$ and $p_1',q_1'$, to the pair $p_0,q_0$ and $p_1,q_1$ in
Figure~\ref{fig:double-push-down-1-solid-picture}, admits nested pairs of Whitney disks $V_0'$ and $V_1'$ parallel to $V_0$ and $V_1$, as shown in Figure~\ref{fig:double-push-down-1B-picture}. It is also clear that such nested Whitney disks exist for any number of such additional self-intersection pairs.

\begin{figure}[ht!]
         \centerline{\includegraphics[scale=.375]{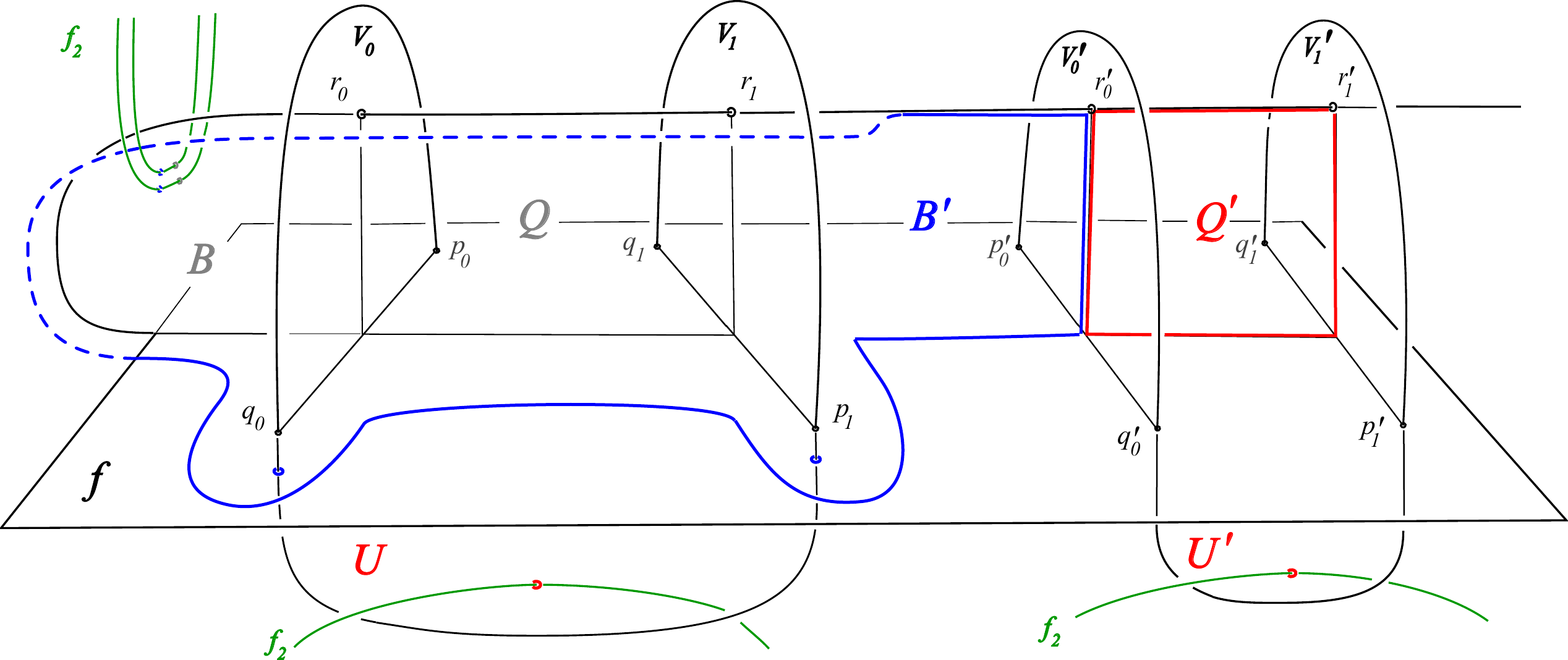}}
         \caption{The boundary of $B'$ is shown in blue, with the dotted sub-arcs indicating where $\partial B'$ has been pushed out of the present time coordinate along with most of the interior of $B'$ (to avoid intersecting $V_0$ and $V_1$). Near $q_0$ and $p_1$ the interior of $B'$ `hangs down underneath' $f$ (in the present), and the pair of transverse intersections between $f$ and $B'$ (near $q_0$ and $p_1$) can be removed by a Whitney move on $B'$ guided by a parallel of the Whitney disk $U$ from Lemma~\ref{lem:p0-q1-w-and-accessory-disks}.}
         \label{fig:B-prime}
\end{figure}
In order to continue with Step~2 we need to find an appropriate $Q'$ and $B'$ for $r_0'= V_0'\pitchfork f$ and $s_1'=V_1\pitchfork f$.
The quadrilateral Whitney disk $Q'$ sits between $V_0'$ and $V_1'$, the same as $Q$ sits between $V_0$ and $V_1$, but
a little work is required to find an appropriate generalized accessory disk $B'$ for $r_0'$:  The problem is that forming $B'$ from a parallel copy of $A_i^-$ will yield intersections between $\partial B'$ and both $\partial V_0$ and $\partial V_1$, which are transverse (in $f$) to $\partial A_i^-$. This problem is solved by first pushing $\partial B'$ off of $\partial V_0$ and $\partial V_1$ as shown in Figure~\ref{fig:B-prime}, which creates interior intersections between $B'$ and $f$, and then using a parallel copy of the Whitney disk $U$ to get the interior of $B'$ disjoint from $f$, (recall from Lemma~\ref{lem:p0-q1-w-and-accessory-disks} that $U$ is framed and embedded, with interior disjoint from $f$). This Whitney move does not change the primary multiplicity of $B'$.
Now $r_0'$ and $s_1'$ admit framed embedded $B'$ and $Q'$ disjoint from $B$ and $Q$, and the subsequent steps 3 through 6 can be carried out as before.

Although Figure~\ref{fig:B-prime} only illustrates the case $n_i=2$, the $n_i>2$ cases are similar: If there was another pair of Whitney disks $V''_0, V''_1$ to the right of $V'_0, V'_1$ in Figure~\ref{fig:B-prime}, then another generalized accessory disk $B''$ could be constructed by extending a parallel of $B'$, pushing $\partial B''$ off of
$\partial V'_0$ and $\partial V'_1$, and then using a parallel of $U'$ to get the interior of $B''$ disjoint from $f$. (The existence
of $Q''$ between $V''_0$ and $V''_1$ is clear.)
Repeating this process as needed yields disjointly embedded framed generalized accessory and Whitney disks for any $n_i$, and the rest of the steps can be carried out as before.

This completes the proof of Proposition~\ref{prop:I-squared}.
\end{proof}


\section{Appendix: Whitney disks and accessory disks}\label{sec:appendix}
This section contains some details on techniques of immersed surfaces in $4$--manifolds that are used throughout the paper, especially in Section~\ref{sec:thm-main-proof}.
More information can be found in \cite{FQ}. For further details and generality on Whitney disks and Whitney moves see e.g.~\cite[Thm.6.6]{M3}.

Unless otherwise specified, submanifolds are assumed to intersect generically.
Orientations are also assumed, but usually not specified explicitly. 
The discussion here holds in the flat topological category via the notions of 4-dimensional topological tranversality from \cite[chap.9]{FQ}.


In the case that the boundary $\partial D$ of an immersed disk $D$ is contained in the interior of an immersed surface $A$, we require and assume that $\partial D$ is embedded, and also that the interior of $D$ is disjoint from $A$ near $\partial D$, i.e.~that there exists a collar in $D$ of $\partial D$ such that the intersection of this collar with $A$ is equal to $\partial D$. In the case where the boundary of a disk $W$ passes through an interior intersection $p$ between surfaces $A$ and $B$, the three sheets are required to meet near $p$ as illustrated for the model Whitney disk $W$ in Figure~\ref{Whitney-disk-pic-and-movie}. (A \emph{sheet} of a surface is a subdisk, open or closed.)

\begin{figure}[ht!]
         \centerline{\includegraphics[scale=.4]{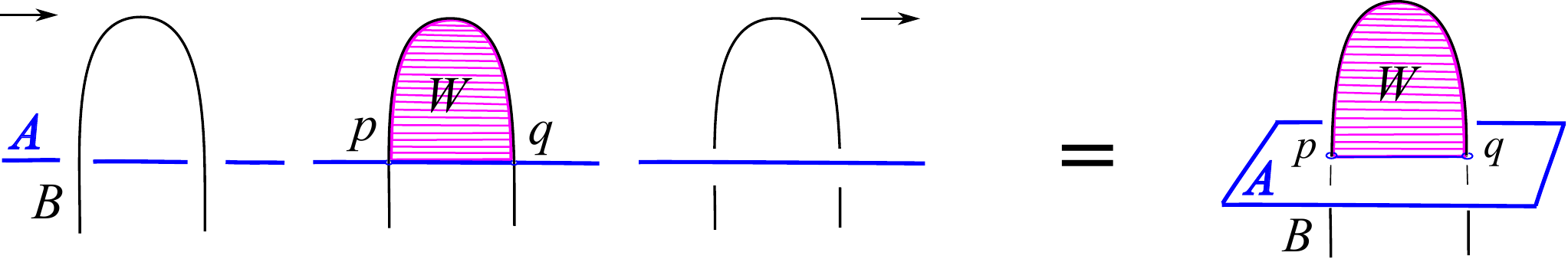}}
         \caption{An embedded model Whitney disk $W$ pairing intersections between surface sheets $A$ and $B$: In the right-most picture the $A$-sheet is contained in the `present' 3-dimensional slice of $4$--space (along with $W$), while the black arc extends into `past' and `future' to describe the $B$-sheet, as in the 3-picture `movie' on the left (with `time' moving from left to right).}
         \label{Whitney-disk-pic-and-movie}
\end{figure}

\subsection{Whitney disks}\label{sec:whitney-disks}
Let $p$ and $q$ be oppositely-signed transverse intersections between connected immersed surfaces $A$ and $B$ in a $4$--manifold $X$ (in this paper such $p$ and $q$ are called a \emph{canceling pair} of intersections), with $p$ and $q$ joined by embedded interior arcs $a\subset A$ and $b\subset B$ which are disjoint from all other singularities in $A$ and $B$. Here we allow the possibility that $A=B$, in which case the circle $a\cup b$ must be embedded and change sheets at $p$ and $q$. Any generic immersed disk $W\imra X$ bounded by such a \emph{Whitney circle} $a\cup b$ is a \emph{Whitney disk} pairing $p$ and $q$. Figure~\ref{Whitney-disk-pic-and-movie} shows a model Whitney disk in $4$--space.

From the assumption of genericity above, all Whitney disks in this paper are required to have embedded boundary. The adjective ``immersed'' will occasionally be attached to ``Whitney disk'' to remind the reader that a Whitney disk interior may not be embedded.

Throughout the following discussion of Whitney disks and Whitney moves the immersed surfaces $A$ and $B$ are allowed to have boundary in the interior of $X$, and either may itself be a Whitney disk,
as in the constructions of Section~\ref{sec:thm-main-proof}.

\begin{figure}[ht!]
         \centerline{\includegraphics[scale=.24]{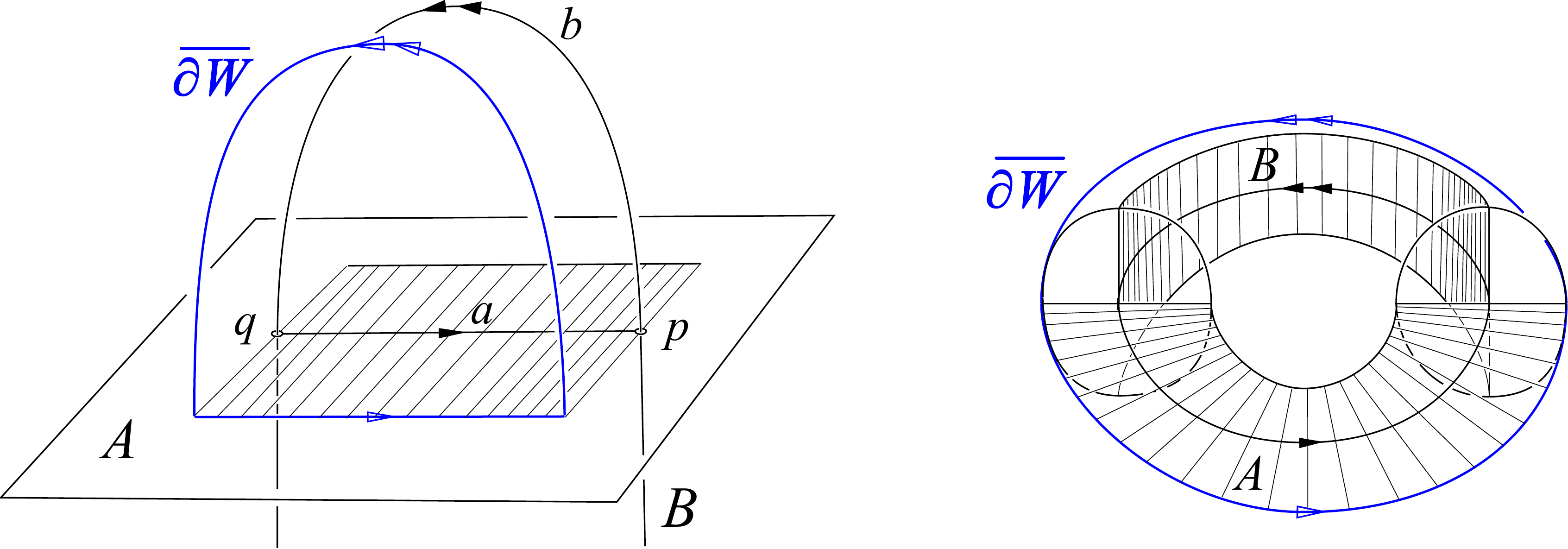}}
         \caption{Left: Near a Whitney disk $W$ pairing $p,q\in A\pitchfork B$, with $\partial W=a\cup b$, a Whitney section $\overline{\partial W}$ is shown in blue. This picture is accurate near $\partial W$; in general, the evident (but not explicitly indicated) embedded $W$ bounded by $a\cup b$ may have self-intersections as well as intersections with other surfaces. The line segments transverse to $a$ in $A$ indicate the correspondence with the right-hand picture of the normal disk-bundle $\nu_{\partial W}$ over $\partial W$.
         Right: The blue Whitney section $\overline{\partial W}$ is shown inside an embedding into $3$--space of $\nu_{\partial W}\cong S^1\times D^2$, with the sheets of $A$ and $B$ indicated by line segments transverse to $\partial W$. The $A$-sheet cuts the front solid torus horizontally, while the $B$-sheet cuts the back of the solid torus vertically.}
         \label{W-subspaces}
\end{figure}

\subsubsection{Framed Whitney disks.}\label{sec:framed-w-disks}
For oriented immersed surfaces $A,B\imra X$ in a $4$--manifold $X$, let $W$ be an immersed Whitney disk pairing intersections $p,q\in A\pitchfork B$, 
with boundary 
$\partial W= a\cup b$, for embedded arcs $a\subset A$ and $b\subset B$. 
Denote by $\nu_{\partial W}$ the restriction to $\partial W$ of the normal disk-bundle $\nu_W$ of $W$ in $X$. Since $p$ and $q$ have opposite signs, $\nu_{\partial W}$ admits a nowhere-vanishing \emph{Whitney section} 
$\overline{\partial W}$ defined by taking vectors tangent to $A$ over $a$, and extending over $b$ by vectors which are normal to $B$, as shown in the left of Figure~\ref{W-subspaces}. 

 The right side of Figure~\ref{W-subspaces} shows $\overline{\partial W}$ inside an embedding into $3$--space of $\nu_{\partial W}\cong S^1\times D^2$.
 Although this choice of embedding has $\overline{\partial W}$ corresponding to the $0$-framing of $D^2\times S^1\subset \R^3$ (as can always be arranged), the section of $\nu_{\partial W}$ determined (up to homotopy) by the canonical framing of $\nu_W$ will in general differ by ($\omega(W)$-many) full twists relative to $\overline{\partial W}$. (If $p$ and $q$ had the same sign then there would have to be a half-twist in the sheets of $A$ and $B$, so the (continuous) Whitney section $\overline{\partial W}$ could not exist.)

If $\overline{\partial W}$ extends to a nowhere-vanishing section $\overline{W}$ of $\nu_W$, then $W$ is said to be \emph{framed} (since the disk-bundle over a disk $\nu_W$ has a canonical framing, and a nowhere-vanishing normal section over an oriented surface in an oriented $4$--manifold determines a framing up to homotopy). In general, the obstruction to extending $\overline{\partial W}$ to a nowhere vanishing section of $\nu_W$ is the relative Euler number $\chi(\nu_W,\overline{\partial W})\in\Z$, called the \emph{twisting} of $W$ and denoted $\omega(W)$, so $W$ is framed if and only if $\omega(W)=0$. The twisting $\omega(W)$ can be computed by taking the intersection number of the zero section $W$ with any extension $\overline{W}$ of $\overline{\partial W}$ over $W$, so it does not depend on an orientation choice for $W$ (since switching the orientation on $W$ also switches the orientation of $\overline{W}$). And $\omega(W)$ is also unchanged by switching the roles of $a$ and $b$ in the construction of $\overline{\partial W}$, since interchanging the ``tangent to...'' and ``normal to...'' parts in the construction yields an isotopic section in $\nu_{\partial W}$ (isotopic through non-vanishing sections).

\begin{rem}\label{rem:general-whitney-section}
The twisting $\omega(W)$ of $W$, which is the element of $\pi_1(\mathrm{SO}(2))\cong\Z$ determined by a Whitney section as described above, 
can be computed using \emph{any section $\overline{\partial W}$ of $\nu_{\partial W}$ such that $\overline{\partial W}$ is in the complement of the tangent spaces of both $A$ and $B$}, since such $\overline{\partial W}$ will have the same number of rotations as a Whitney section (relative to the longitude determined by the canonical framing of $\nu_W$);
see Figure~\ref{W-subspaces}.
Parallel copies of $W$ extending such nowhere-tangent sections are used, for instance, in the constructions of Whitney spheres and accessory spheres in Section~\ref{sec:accessory-sphere-revisit}.
\end{rem}

The adjective ``parallel'' applied to a Whitney circle or Whitney disk in this paper always refers to one of the constructions described above.

\subsection{Accessory disks}\label{sec:acc-disks}

Let $p$ be a transverse self-intersection in a connected immersed surface $f:\Sigma\imra X$, where as usual we blur the distinction between $f$ and its image in the $4$--manifold $X$.
An embedded circle in $f$ which changes sheets at $p$ and is disjoint from the boundary and all other singularities of $f$ is called an \emph{accessory circle} for $p$. (Accessory circles are also sometimes called \emph{double point loops}.)

Any generically immersed disk bounded by an accessory circle for $p$ is called an \emph{accessory disk} for $p$.
In this paper accessory disks will only be used in the case that $\Sigma= S^2$ and $X=S^4$,
so every $p\in f\pitchfork f$ admits an accessory disk, for any choice of accessory circle.

\subsubsection{Framed accessory disks}\label{sec:framed-acc-disks}
\begin{figure}[ht!]
         \centerline{\includegraphics[scale=.25]{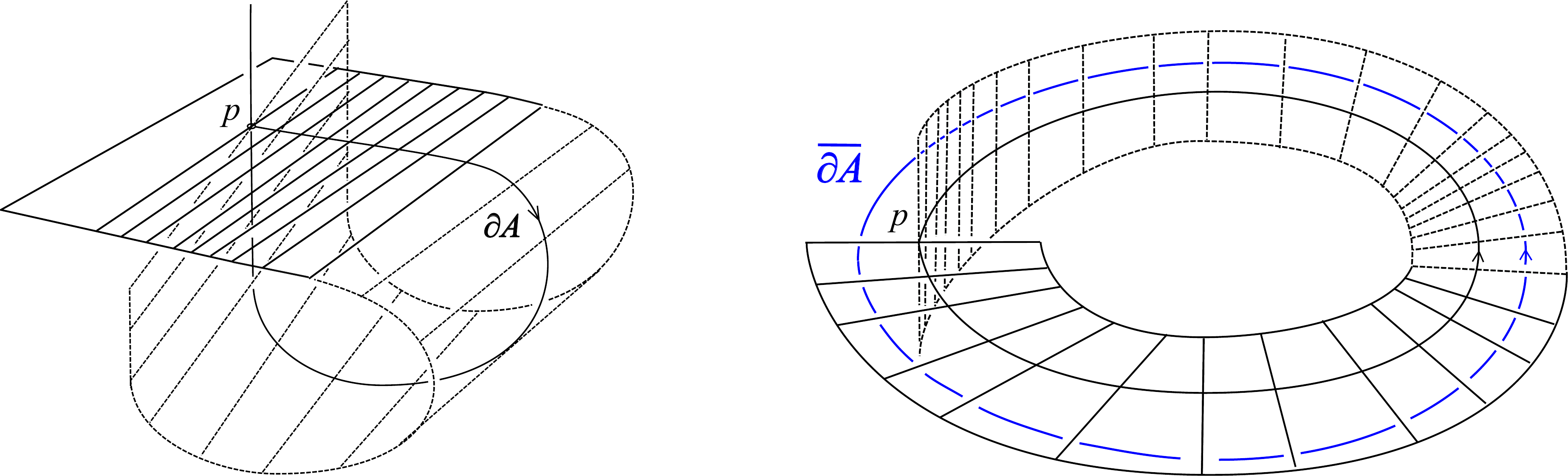}}
         \caption{Left: Near the sheets of $f$ along $\partial A$ in $X$. Right: Inside an embedding of $\nu_{\partial A}\cong D^2\times S^1$ into $3$--space, an accessory section $\overline{\partial A}$ (in blue) in the complement of the tangent space of $f$. This embedding has $\overline{\partial A}$ corresponding to the $0$-framing of $D^2\times S^1\subset \R^3$ (as can always be arranged), but the section determined by the canonical framing of $\nu_A$ may in general differ by ($\omega(A)$-many) full twists relative to $\overline{\partial A}$.}
         \label{accessory-sheets}
\end{figure}

Let $p$ be a self-intersection of $f$ admitting an accessory disk $A$ with (embedded) boundary $\partial A$. Denote by $\nu_{\partial A}$ the restriction to $\partial A$ of the normal disk-bundle $\nu_A$ of $A$ in $X$. The tangent space to $f$ along $\partial A$ determines a $1$-dimensional subspace of $\nu_{\partial A}$ at each point other than $p$, and at $p$ the two sheets 
of $f$ determine a pair of transverse $1$-dimensional subspaces of $\nu_{\partial A}$ (Figure~\ref{accessory-sheets}). Let $\overline{\partial A}$ be any nowhere-vanishing section of $\nu_{\partial A}$ such that $\overline{\partial A}$ is in the complement of the tangent space of $f$, and define the twisting $\omega(A)\in\Z$ to be the obstruction to extending $\overline{\partial A}$ over $A$, i.e.~$\omega(A)$ is the relative Euler number $\chi(\nu_A,\overline{\partial A})$ (cf.~Remark~\ref{rem:general-whitney-section} just above). Then $A$ is said to be \emph{framed} if $\omega(A)=0$.
Note that $\omega(A)$ is well-defined since the rotations (relative to $A$) of $\overline{\partial A}$ in $\nu_{\partial A}\cong S^1\times D^2$ are determined by the $1$-dimensional subspaces cut out by $f$, and $\omega(A)$ does not depend on a choice of orientation of $A$ (as observed above for Whitney disks). Any such section $\overline{\partial A}$ will be referred to as an \emph{accessory section} for $A$, and in this paper the adjective ``parallel'' applied to an accessory disk always means an extension over $A$ of an accessory section.


\subsection{Whitney disks from accessory disks}\label{sec:W-disk-from-A-disks}
\begin{figure}[ht!]
         \centerline{\includegraphics[scale=.275]{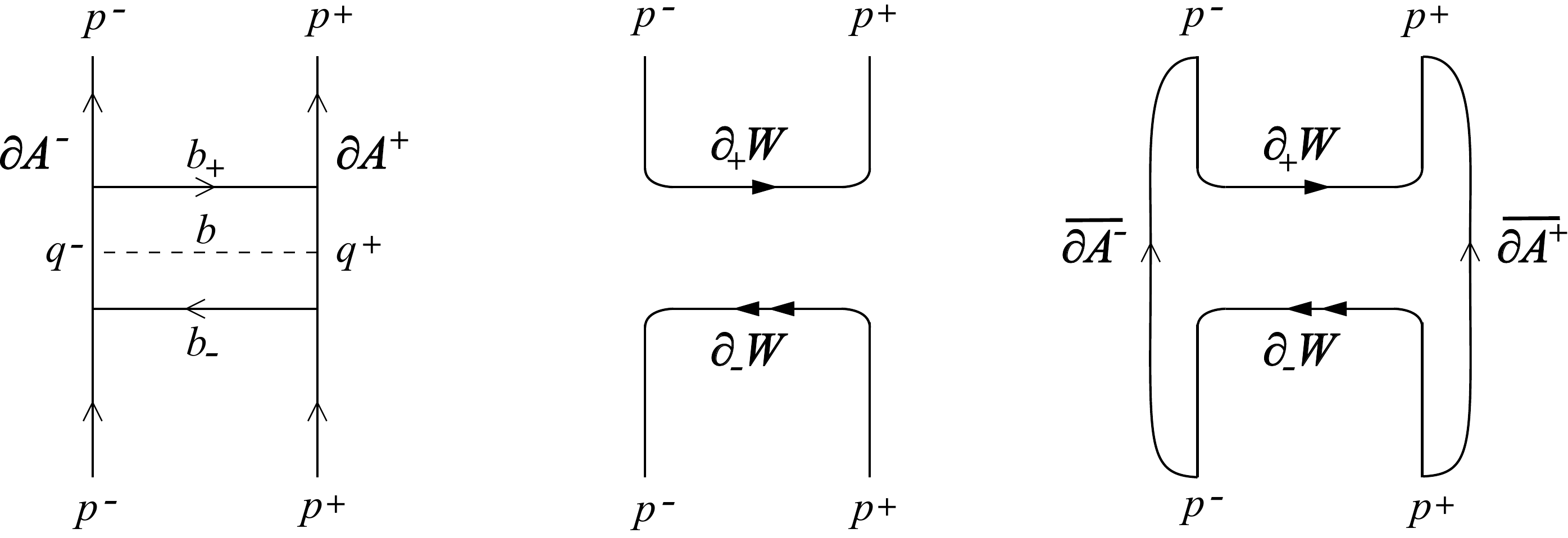}}
         \caption{The view in the domain of $f$: Banding together $\partial A^\pm$ along the dotted arc $b$ (left) yields the boundary $\partial W=\partial_+W\cup\partial_-W$ (center) of a Whitney disk $W$ constructed by `half-tubing' $A^\pm$ together (Figure~\ref{joined-accessory-disks-1}). Copies $\overline{A^\pm}$ of $A^\pm$ (perturbed rel $p^\pm$) are still accessory disks for $p^\pm$.}
         \label{banded-accessory-arcs}
\end{figure}

Here are some details on forming Whitney disks from pairs of accessory disks, as used in Section~\ref{sec:orientation-conventions} and during the double transfer move of Step~8:

Let $A^+$ and $A^-$ be accessory disks for a pair $\{p^+,p^-\}$ of oppositely-signed self-intersections of $f$ such that $\partial A^+\cap\partial A^-=\emptyset$. Choose an embedded arc $b$ in $f$
connecting a point $q^+\in\partial A^+$ to a point $q^-\in\partial A^-$, with the interior of $b$ disjoint from both 
$\partial A^\pm$. Banding together $\partial A^\pm$ in $f$ along $b$ yields a Whitney circle for $p^\pm$ which is the union of $\partial A^\pm$ minus small arcs containing $q^\pm$ together with arcs $b_+$ and $b_-$ parallel to $b$ (the preimage is shown in the left and center pictures of Figure~\ref{banded-accessory-arcs}). Denote this Whitney circle by $\partial W$ since a Whitney disk $W$ will be described shortly. Let $z^\pm$ denote inward pointing tangent vectors to $A^\pm$ at $q^\pm$. Choose a nowhere vanishing vector field $z_t$ field over $b$ such that for each $t\in[0,1]$ $z_t$ is normal to $f$, with $z_0=z^-$ and $z_1=z^+$. The arc of vectors $z_t$ determines a half-tube $H$ connecting $A^-$ and $A^+$, where arcs of $H$ from $b_-$ to $b_+$ are traced out by rotating a vector based at $b$ from $b_-$ to $b_+$ through the corresponding $z_t$ (Figure~\ref{joined-accessory-disks-1}). The Whitney disk $W$ is formed by deleting small half-disks from $A^\pm$ at $q^\pm$ and connecting the resulting disks with $H$.
\begin{figure}[ht!]
         \centerline{\includegraphics[scale=.275]{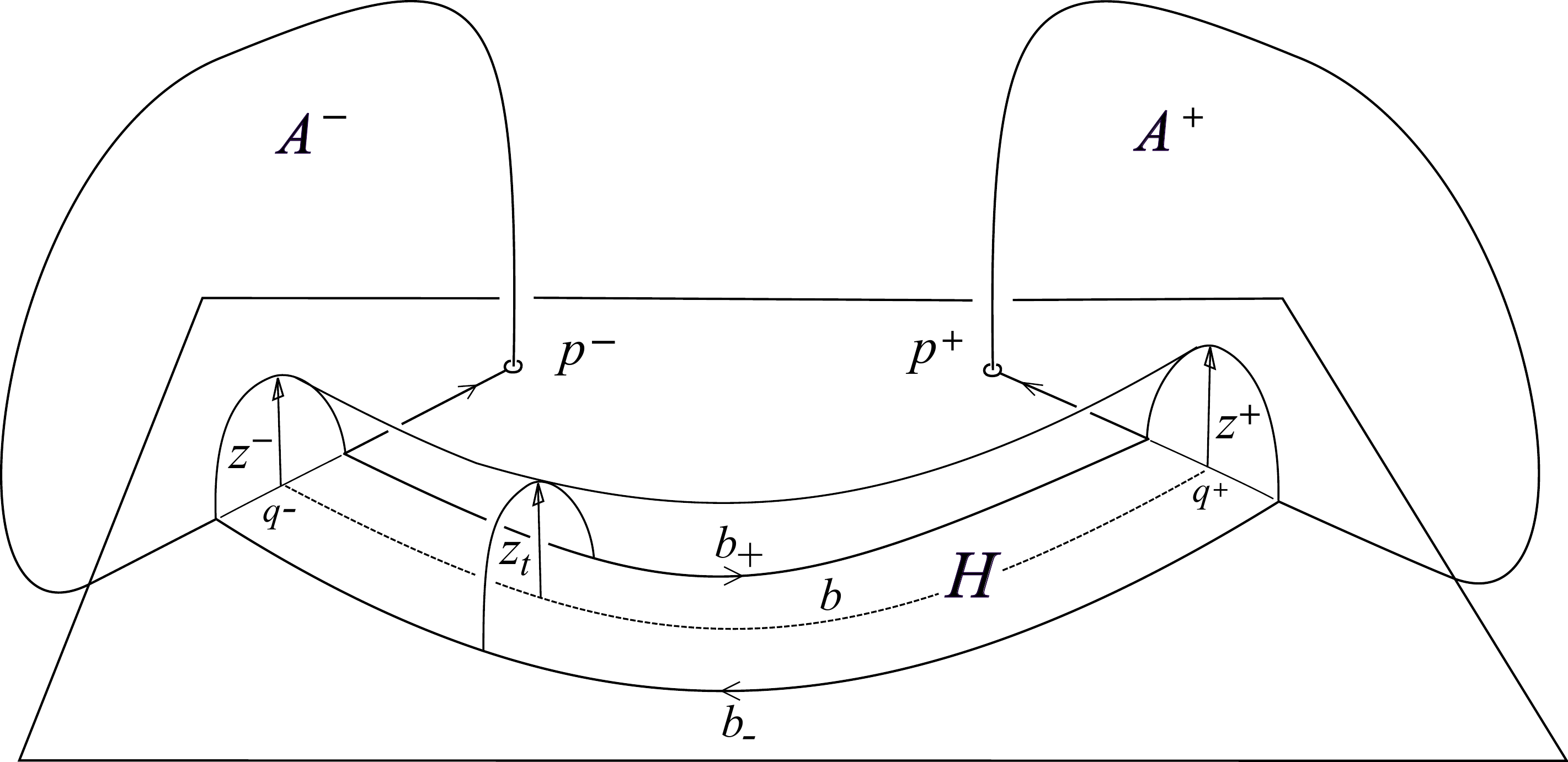}}
         \caption{Joining $A^\pm$ along the half-tube $H$ around $b$ to form a Whitney disk $W$ pairing $p^\pm$ (see Figure~\ref{banded-accessory-arcs} for boundary inverse images). If $A^\pm$ are framed, disjointly embedded, with interiors disjoint from $f$, then this picture accurately describes a $3$-dimensional slice of the local coordinates determined by a choice of $z_t$ along $b$. }
         \label{joined-accessory-disks-1}
\end{figure}

Observe that the twisting $\omega(W)$ is the sum of the twistings of $A^+$ and $A^-$: First note that since $p^\pm$ have opposite signs, the quarter turns in the right side of Figure~\ref{accessory-sheets} fit together as in the right side of Figure~\ref{W-subspaces}.
The vectors $z_t$ extend by translation to nowhere-vanishing normal vectors over all of $H\subset W$ which are also in the complement of $f$. Choosing a Whitney section by extending the restriction to $b_\pm\subset\partial W$ over the rest of $\partial W$ is the same as choosing accessory 
sections over $\partial A^+$ and $\partial A^-$. Since these sections already extend over $H$, the twisting $\omega(W)$ of $W$ is equal to the sum $\omega(A^+)+\omega(A^-)$.
(The choice of $z_t$ (which can rotate around $f$) does not affect $\omega(W)$ because a rotation of $z_t$ which creates a right-handed twist around $b_+$ also creates a left-handed twist around $b_-$, and vice versa.) 

Copies $\overline{A^\pm}$ of $A^\pm$ (perturbed rel $p^\pm$) are still accessory disks for $p^\pm$, and can be chosen so that $\overline{A^\pm}$ and $\partial W$ only intersect at $p^\pm$ (right). In particular, if the original $A^\pm$ were framed and disjointly cleanly embedded, then a regular neighborhood of the resulting triple $W,\overline{A^+},\overline{A^-}$ is diffeomorphic to  
Figure~\ref{fig:Whitney-and-accessory-disks-ORIENTED}.

This construction can be arranged to preserve or switch any given orientation on either or both of $A^\pm$ in $W$ by choosing the guiding arc $b$ to emanate from an appropriate side of $\partial A^\pm$ in $f$.

\subsection{Accessory disks from Whitney disks}\label{sec:A-disks-from-W-disks}
\begin{figure}[ht!]
         \centerline{\includegraphics[scale=.275]{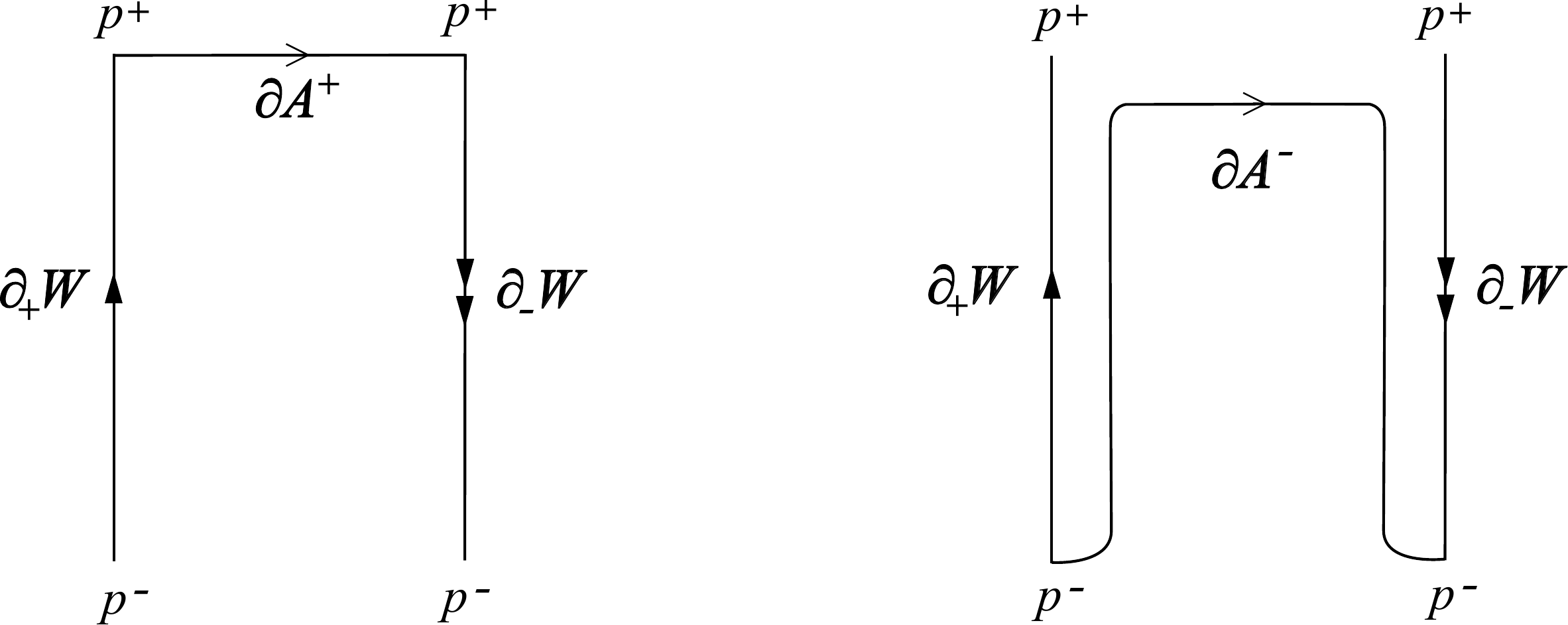}}
         \caption{The view in the domain of $f$: Forming a negative accessory disk $A^-$ (right) from parallels of the Whitney disk $W$ and positive accessory disk $A^+$ (left).}
         \label{banded-whitney-and-accessory-arc}
\end{figure}

The following observation is used in Step~8 to convert a Whitney disk-accessory disk pair into a pair of accessory disks, in Section~\ref{sec:double-transfer-orientations} to switch between positive and negative accessory disks, and in Section~\ref{sec:check-metabolic-disks} to exchange negative accessory disks for a positive accessory disks: Given a Whitney disk $W$ pairing self-intersections $p^\pm$ of $f$, and a positive accessory disk $A^+$ for $p^+$, then a negative accessory disk $A^-$ for $p^-$ can be constructed from a parallel of $W$ together with a parallel of $A^+$, as indicated in Figure~\ref{banded-whitney-and-accessory-arc} (which shows the inverse images of the boundaries in the domain of $f$). Any intersections that $W$ and $A^+$ have with any surfaces will be inherited by $A^-$, as will any self-intersections of $W$ and $A^+$. The twisting $\omega(A^-)$ of $A^-$ will be the sum $\omega(W)+\omega(A^+)$. The roles of $A^+$ and $A^-$ can be switched in this construction. 
So in particular, if $W,A^\pm$ is a metabolic pair, then this constructions yields $W,A^\mp$ which is also a metabolic pair.

\subsection{Whitney moves and transfer moves}\label{sec:W-move}
Given a Whitney disk $W$ pairing $p,q\in A\pitchfork B$, a \emph{Whitney move on $A$ guided by $W$} (also called a \emph{Whitney move on $A$ along $W$}) eliminates $p$ and $q$ by the construction illustrated in Figure~\ref{Whitney-move-blue-black-movie}: A small neighborhood of $\partial W\cap A$ in $A$ is deleted, and a \emph{Whitney bubble} consisting of two oppositely-oriented parallel copies of $W$ and a parallel of a small neighborhood of $\partial W\cap B$ is added to $A$.

\begin{figure}[ht!]
         \centerline{\includegraphics[scale=.4]{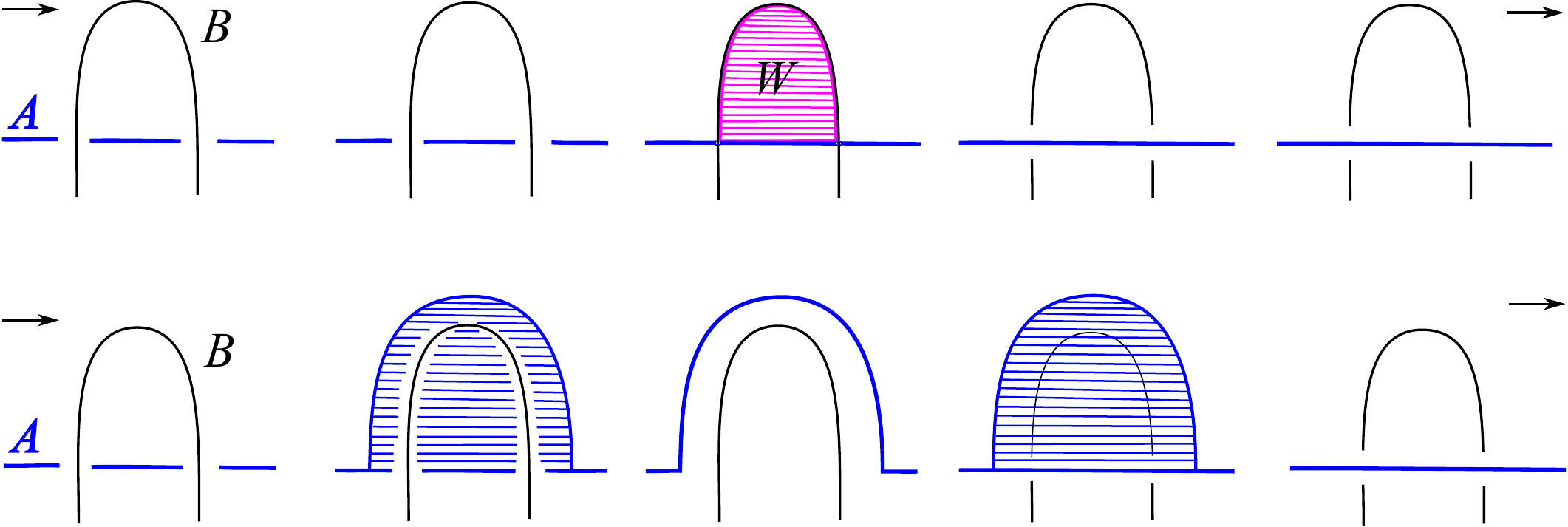}}
         \caption{Before and after a model Whitney move: The top and bottom rows of pictures show the same slices of a $4$-ball neighborhood containing surface sheets $A$ (blue) and $B$ (black). Before the Whitney move (top row), the intersections between $A$ and $B$ are paired by the framed embedded Whitney disk $W$ (purple). After applying the Whitney move to $A$ along $W$ (bottom row), $A$ and $B$ are disjoint, and the Whitney bubble added to $A$ is visible in the three center pictures.}
         \label{Whitney-move-blue-black-movie}
\end{figure}

The result of a Whitney move on $A$ along $W$ is the same as changing $A$ by a regular homotopy supported near $W$, and we frequently keep the same name/notation for a surface that has been changed by a Whitney move.

\subsubsection{Whitney moves guided by accessory disks}\label{sec:w-move-on-acc-disks}
We discuss next how parallel copies of framed accessory disks can be used to guide Whitney moves on surfaces intersecting the two sheets of $f$ near $p$, as in the constructions of accessory spheres in Section~\ref{sec:accessory-sphere-revisit} (Figure~\ref{fig:Accessory-sphere-movie}) and during Step~7 of the proof of Proposition~\ref{prop:I-squared}, see Figures~\ref{fig:double-boundary-twist-Whitney-move-on-accessory-disk-2} and~\ref{parallel-accessory-disk-movie}.

Before giving details, here is the main idea: Observe that any accessory section can be isotoped through nowhere-vanishing sections in $\nu_{\partial A}$ to be tangent to $f$ along most of $\partial A$ except along a small arc of $\partial A$ near $p$ where the section makes a ``quarter-turn'' through vectors normal to both sheets of $f$. This small arc can be pushed off of $f$ slightly into $A$ while preserving the twisting of the section,
and ``trimming'' $A$ along this pushed-in arc yields one of a family of parallel Whitney disks such as $V$ in Figure~\ref{parallel-accessory-disk-movie}.

.
\begin{figure}[ht!]
         \centerline{\includegraphics[scale=.45]{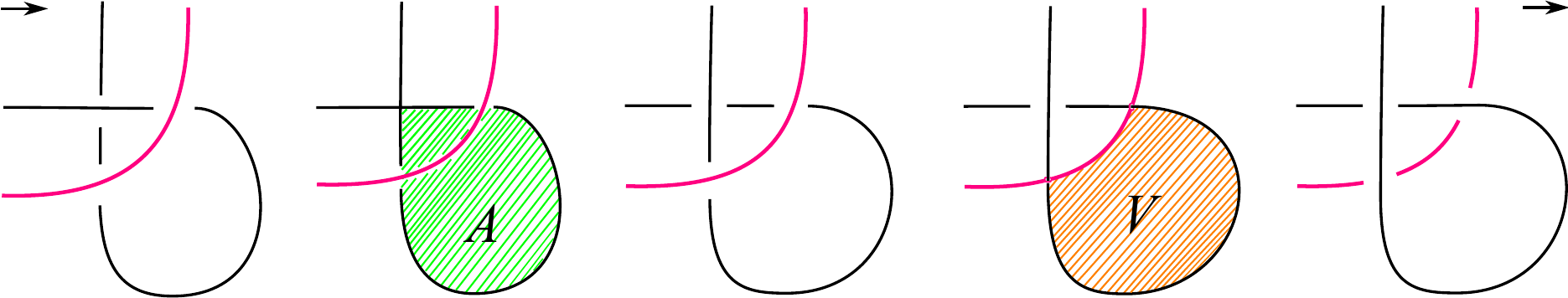}}
         \caption{Near a framed embedded accessory disk $A$ (green) on $f$ (black), intersections between a surface sheet (red) and $f$ are paired by a Whitney disk $V$ formed from a (trimmed) parallel copy of $A$.}
         \label{parallel-accessory-disk-movie}
\end{figure}

So let $p$ be a self-intersection of $f$ admitting a framed oriented immersed accessory disk $A$ with embedded boundary $\partial A$.
Thinking of $\partial A$ as the image of an immersion of an interval, $a:[0,1]\to f(\Sigma)\subset X$ with $a(0)=p=a(1)$, 
let $v_+$ denote the unit tangent vector to $f$ at $p$ pointing positively along $\partial A$, and 
let $v_-$ denote the unit tangent vector to the other sheet of $f$ at $p$ pointing back negatively along $\partial A$ (See Figure~\ref{accessory-arc-1}).
Denote the local sheets of $f$ near $p$ by $f_\pm$, where $f_\pm$ contains $v_\pm$.
Choose a tangent vector $\overline{v}_+$ to $f_+$ at $p$, such that $\overline{v}_+$ is orthogonal to $v_+$, and
extend 
$\overline{v}_+$ to a nowhere-vanishing section of the normal bundle of $a$ in $f$.
This section
extending $\overline{v}_+$ defines an embedded arc $\overline{a}:[0,1]\to f(\Sigma)\subset X$ with $\overline{a}(0)=\overline{v}_+$.
We may assume that the terminal endpoint $\overline{a}(1)$ defines a tangent vector $\overline{v}_-$ to $f_-$ at $p$ which is orthogonal to $v_-$ (See Figure~\ref{accessory-arc-1} and left side of Figure~\ref{accessory-disk-and-cut-accessory-disk-1}). 
\begin{figure}[ht!]
         \centerline{\includegraphics[scale=.35]{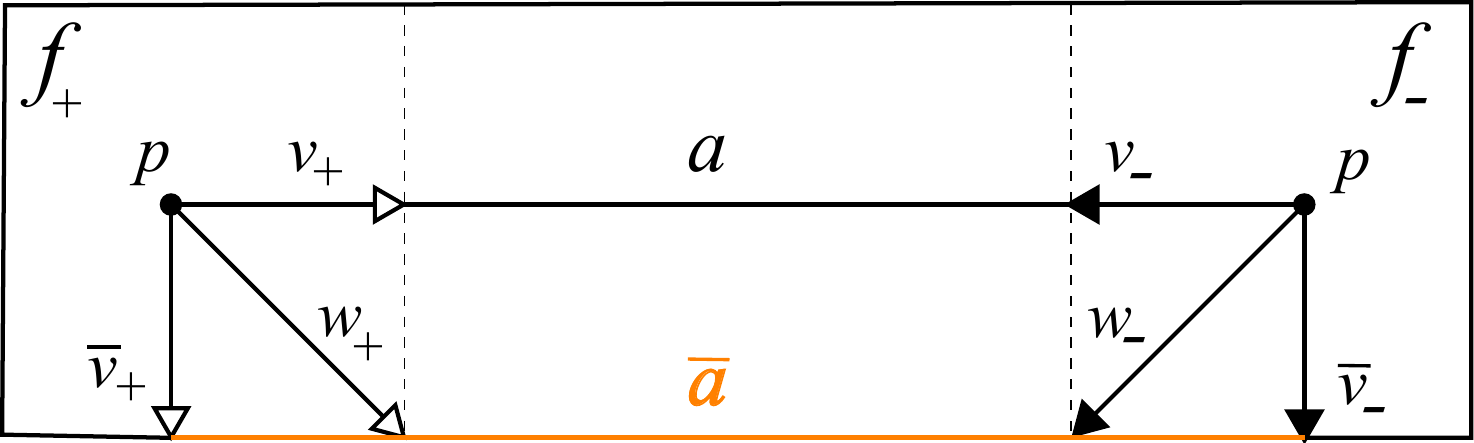}}
         \caption{The preimage of a neighborhood of the accessory circle $\partial A$.}
         \label{accessory-arc-1}
\end{figure}

\begin{figure}[ht!]
         \centerline{\includegraphics[scale=.3]{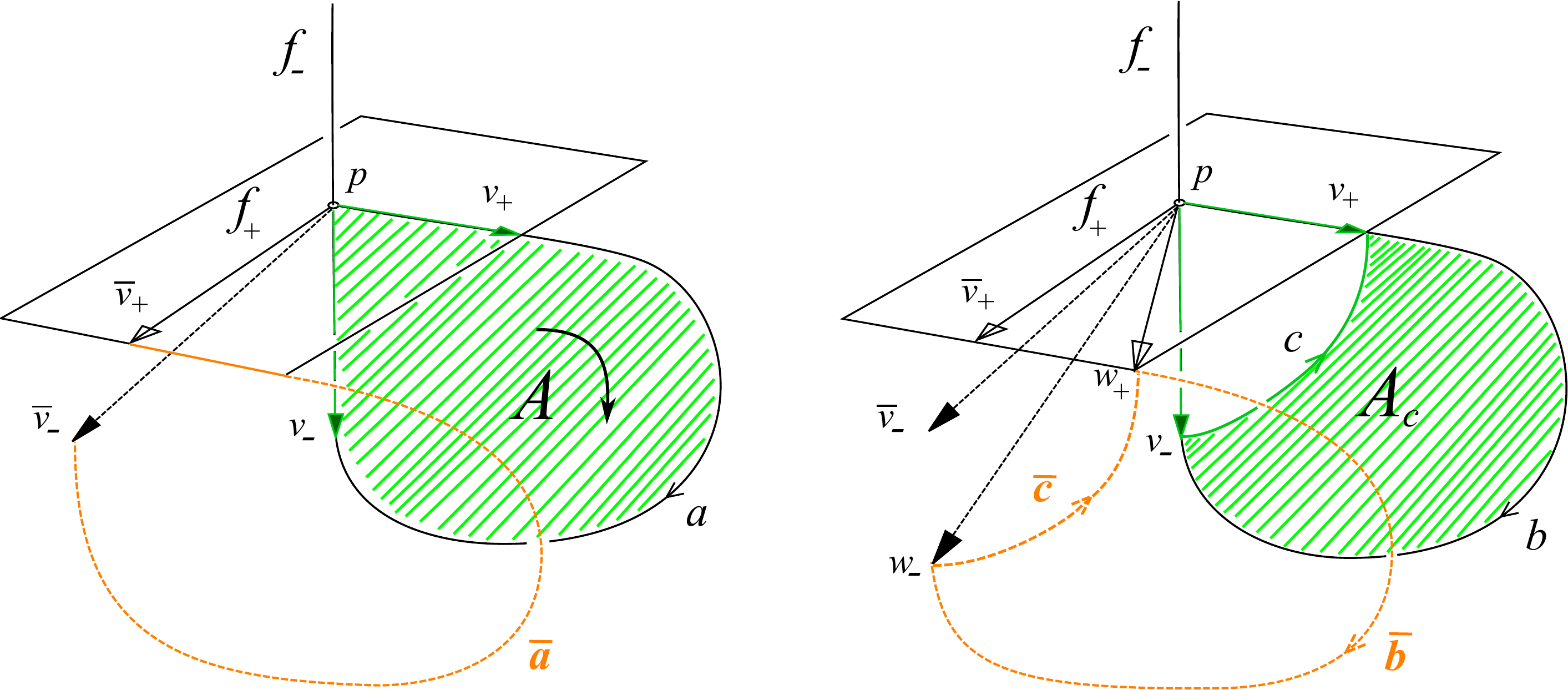}}
         \caption{Left: The accessory disk $A$ and the arc $\overline{a}\subset f$ parallel to $a=\partial A$ in $f$.
         Right: The trimmed accessory disk $A_c\subset A$ with boundary $b\cup c$, and the section $\overline{\partial A_c}=\overline{b}\cup \overline{c}$ of $\nu_{\partial A_c}$. Note that the dotted vectors $\overline{v}_-$ and $w_-$ point in the direction of an orthogonal `time' coordinate.}
         \label{accessory-disk-and-cut-accessory-disk-1}
\end{figure}

Around $p$ the vectors $v_+$ and $\overline{v}_+$ span an embedded square of $f_+$, and similarly the vectors $v_-$ and $\overline{v}_-$ span an embedded square of $f_-$. 
By transversality, the vectors $v_+,\overline{v}_+,v_-,\overline{v}_-$ span a $4$-dimensional cube of $X$ around $p$, and it may be assumed that $A$ intersects this cube in a square quadrant spanned by $v_+$ and $v_-$. Rotating $v_-$ to $v_+$ traces out an embedded corner of $A$ at $p$ bounded by the union of (the segments) $v_-$ and $v_+$ together with the arc $c$ traced out by the tips of the rotating vectors. Deleting from $A$ all of this corner except for $c$ yields a slightly smaller disk $A_c$ whose boundary $\partial A_c=b\cup c$ is the union of $c$ together with the arc $b\subset\partial A$ outside the corner (from the tip of $v_+$ to the tip of $v_-$). (See Figure~\ref{accessory-disk-and-cut-accessory-disk-1}).

Let $\nu_{\partial A_c}$ denote the restriction of the normal bundle $\nu_{A_c}$ of $A_c$ in $X$ to the boundary $\partial A_c$.  Define a nowhere-vanishing section $\overline{\partial A_c}$ of $\nu_{\partial A_c}$
as follows: Over the arc $b\subset a$ take the vectors $\overline{b}\subset\overline{a}$. Specifically, $\overline{b}$ starts at the tip of $w_+:=v_++\overline{v}_+$ (over the tip of $v_+$) and ends at the tip of $w_-:=v_-+\overline{v}_-$ (over the tip of $v_-$). To define $\overline{\partial A_c}$ over the arc $c$ of $\partial A_c$, note that rotating $w_-$ to $w_+$ in the plane spanned by $w_-$ and $w_+$ defines an arc $\overline{c}$ (traced out by the tip of the rotating vector), which determines
an arc of normal vectors to $A_c$ over $c$. (To see that $\overline{c}$ is normal to $A_c$, recall that $A$ is contained in the quadrant spanned by $v_+$ and $v_-$ in the cube around $p$, and this quadrant is disjoint from any non-trivial linear combination of $w_-$ and $w_+$.) 

We claim that $\chi(\nu_{A_c},\overline{\partial A_c})=\omega(A)$, so in particular, $\overline{\partial A_c}$ extends to a nowhere-vanishing section over $A_c$ if and only if $A$ is framed. To see the claim, fix the endpoints of $c$, and let $c_t$ denote a family of arcs determined by an isotopy of the interior of $c$ across the corner of $A$ and into the arc of $\partial A$ passing through $p$, so $c_0=c$ and $c_1=v_-\cup v_+\subset \partial A$. 
The union of the $c_t$ with $b$ forms a family of circles $c^\circ_t=c_t\cup b$ interpolating between $\partial A_c=c^\circ_0$ and $\partial A=c^\circ_1$, and these circles bound disks $A_{c_t}\subset A$ interpolating between $A_{c_0}=A_c$ and $A_{c_1}=A$.
Observe that the (fixed) arc $\overline{c}$ determines
an arc of nowhere-vanishing normal vectors to $A$ over each $c_t$ for $0\leq t\leq 1$.
Over each $c^\circ_t$ define a section $\overline{\partial A_{c_t}}$ of the restriction of $\nu_A$ to $c^\circ_t=\partial A_{c_t}$ as follows: Over $b$ (which is a common sub-arc of all the $c^\circ_t$) take $\overline{\partial A_{c_t}}=\overline{b}$, and over $c_t$ define $\overline{\partial A_{c_t}}$ to be the vectors determined by $\overline{c}$.  
Note first that although the section $\overline{\partial A_{c_1}}$ over $c^\circ_1=\partial A$ is tangent to $f$ along the arc $b\subset \partial A$, this part of $\overline{\partial A_{c_1}}$ can be perturbed to be normal to $f$ yielding an isotopic \emph{accessory section} $\overline{\partial A}\subset\nu_{\partial A}$. So $\omega(A)=\chi(\nu_A,\overline{\partial A})=\chi(\nu_A,\overline{\partial A_{c_1}})$.

On the other hand, from the isotopy $\overline{\partial A_{c_t}}$ of bundles with sections induced by the isotopy of arcs $c_t$ we have:  
$$
\chi(\nu_{A_c},\overline{\partial A_c})=\chi(\nu_{A_{c_0}},\overline{\partial A_{c_0}})=\chi(\nu_{A_{c_1}},\overline{\partial A_{c_1}})=\chi(\nu_A,\overline{\partial A_{c_1}}),
$$
so $\chi(\nu_{A_c},\overline{\partial A_c})=\omega(A)$.

Note that the interior of the arc $\overline{c}$ is also disjoint from $f$, since $f$ only intersects the span of $w_+$ and $w_-$ when at least one coefficient of $w_+$ or $w_-$ is zero. So if $A$ is framed and has interior disjoint from $f$ then parallel copies of $A_c$ will also have interiors disjoint from $f$.

It follows from this discussion that if a surface sheet $S$ intersects $f$ near $p$ in a pair of oppositely-signed points 
at the tips of $w_+$ and $w_-$, with $S$ intersecting the normal disk bundle $\nu_A$ of $A$ in the vectors corresponding to a section $\overline{c}$ as in this model, then 
then this intersection pair admits a Whitney disk which is a parallel copy of $A_c$ and has the same twisting.
This is exactly the case illustrated in Figure~\ref{parallel-accessory-disk-movie}, and used in the constructions of accessory spheres in Section~\ref{sec:accessory-sphere-revisit} (Figure~\ref{fig:Accessory-sphere-movie}) and during Step~7 of the proof of Proposition~\ref{prop:I-squared}.



\begin{figure}[ht!]
         \centerline{\includegraphics[scale=.25]{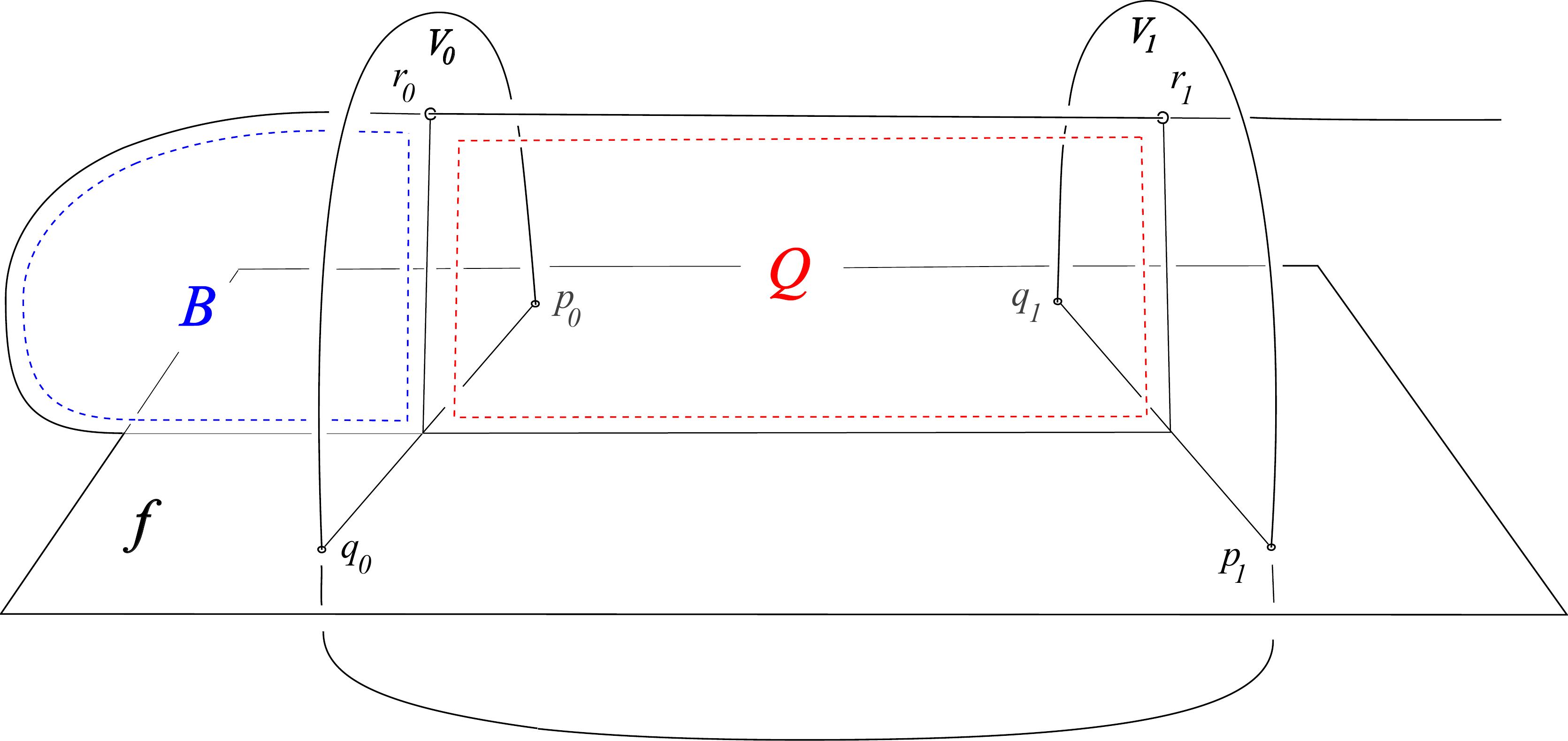}}
         \caption{Before the transfer move.}
         \label{fig:transfer-move-1}
\end{figure}

\subsubsection{Transfer moves}\label{sec:transfer-move-appendix}
\begin{figure}[ht!]
         \centerline{\includegraphics[scale=.25]{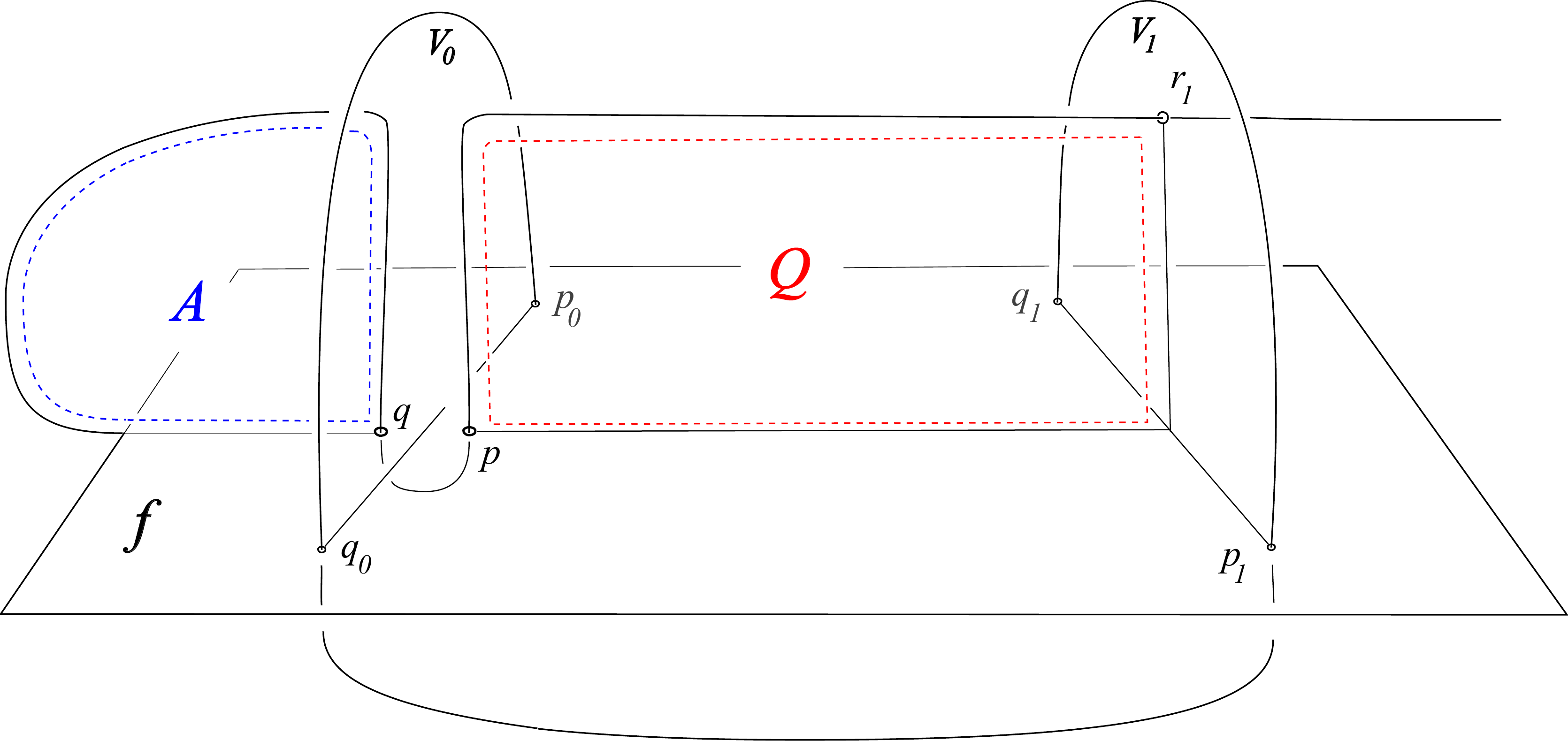}}
         \caption{After the first finger move on $f$ along $V_0$ and across $\partial V_0$.}
         \label{fig:transfer-move-1B}
\end{figure}
\begin{figure}[ht!]
         \centerline{\includegraphics[scale=.25]{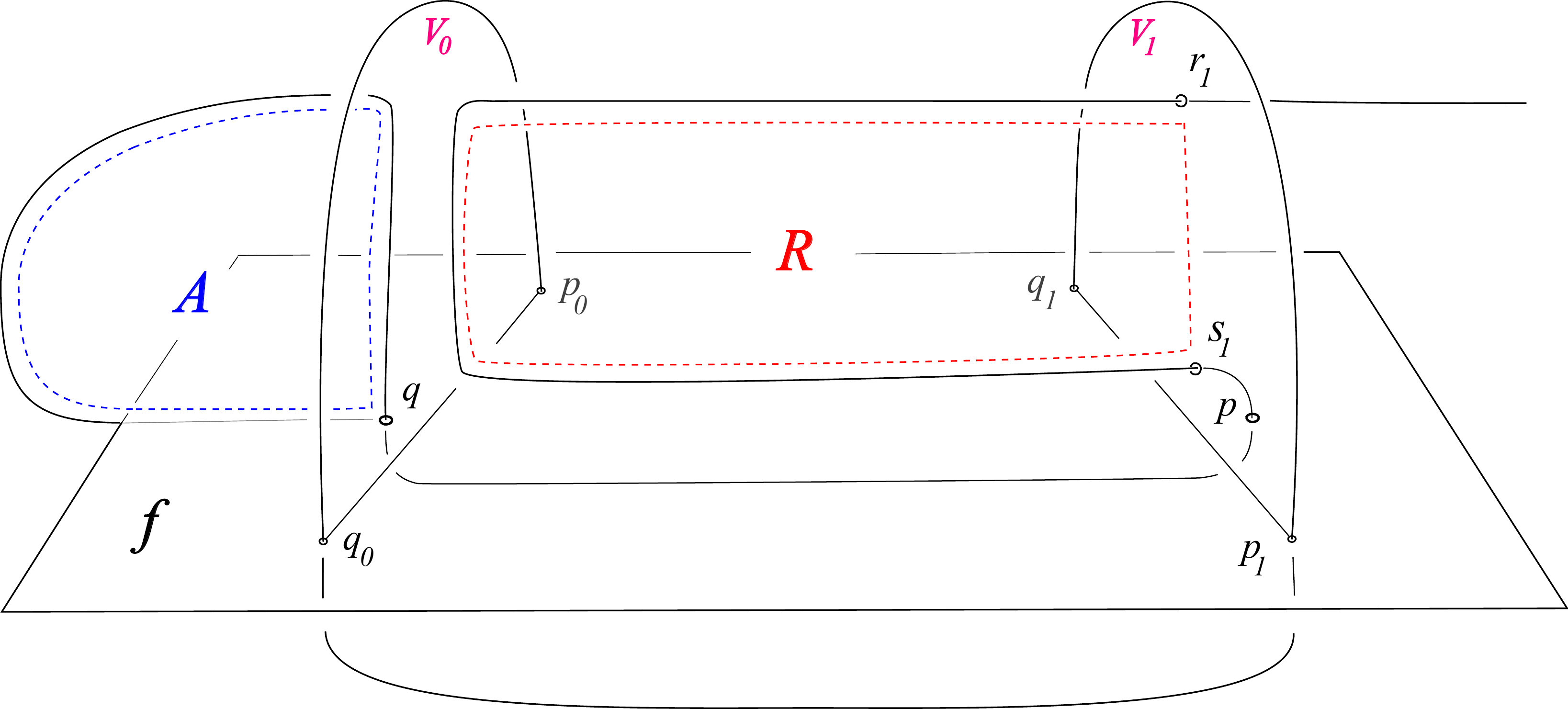}}
         \caption{After the transfer move is completed by a pushing $p$ along $f$ across $\partial V_1$, creating $s_1\in V_1\pitchfork f$, with $R$ pairing $r_1,s_1$.}
         \label{fig:transfer-move-2}
\end{figure}
\begin{figure}[ht!]
         \centerline{\includegraphics[scale=.25]{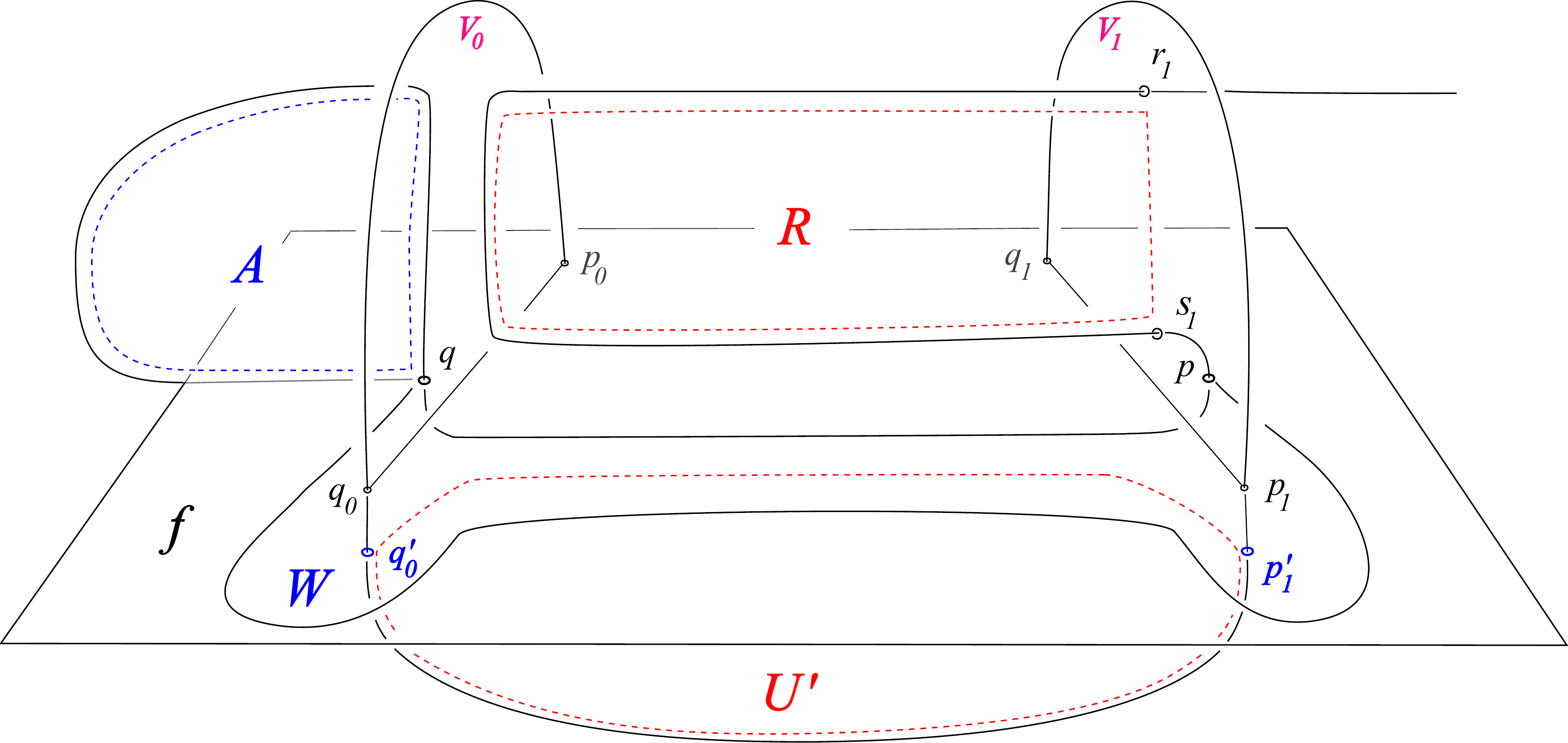}}
         \caption{Whitney disks $W$ and $U'$ pair the newly created intersection pairs $p,q\in f\pitchfork f$ and $p_1',q_0'\in W\pitchfork f$, respectively.}
         \label{fig:transfer-move-3}
\end{figure}
This subsection provides some more detail and background on the \emph{transfer move} construction which ``transfers'' an intersection between
$f$ and the interior of one Whitney disk on $f$ onto another Whitney disk, at the cost of creating new controlled intersections between Whitney disks and $f$.
The origin of this construction appears to go back to \cite{Ya}, and it was generalized and used extensively in the development of an obstruction theory for Whitney towers in \cite{CST1,ST1,ST2}. Here we focus on the setting for the transfer move in Section~\ref{sec:step-5-reduce-primary} of the current paper.

Figure~\ref{fig:transfer-move-1} shows the configuration at Step~2 of the proof of Proposition~\ref{prop:I-squared}. The transfer move exchanges $r_0\in V_0\pitchfork f$ for $s_1\in V_1\pitchfork f$ by first performing a finger move along $V_0$ across $\partial V_0$ {fig:transfer-move-1B}, and then pushing one of the new self-intersections $p\in f\pitchfork f$ across $\partial V_1$ (Figure~\ref{fig:transfer-move-2}). The finger move converts the bigon $B$ from Figure~\ref{fig:transfer-move-1}
into an accessory disk $A$ for the new intersection $q\in f\pitchfork f$ in Figure~\ref{fig:transfer-move-2}, hence we refer to $B$ as a \emph{generalized accessory disk}. 
Similarly, the quadrilateral $Q$ from Figure~\ref{fig:transfer-move-1}
gives rise to the Whitney disk $R$ pairing $r_1,s_1\in V_1\pitchfork f$ in Figure~\ref{fig:transfer-move-2}, hence we refer to $Q$ as a \emph{generalized Whitney disk}.
From the local coordinates it is clear that $A$ and $R$ are framed, and this corresponds to $B$ and $Q$ being framed in the sense that normal sections over $\partial B$ and $\partial Q$ in the complement of the sheets containing $\partial B$ and $\partial Q$ extend over $B$ and $Q$, as in the definition of framed Whitney disks (Remark~\ref{rem:general-whitney-section}) and accessory disks (Section~\ref{sec:framed-acc-disks}).

The constructions of Steps~3 and~4 of the proof of Proposition~\ref{prop:I-squared}
create multiple disjoint parallel copies of $B$ and $Q$ extended by disjointly embedded disks, and the corresponding transfer moves can be carried out simultaneously yielding multiple disjoint parallel copies of $A$ and $R$ extended by embedded disks.
During these Steps~3 and~4 the Whitney disks $V_0$ and $V_1$ are also tubed into framed immersed accessory spheres on $f_2$, but this does not obstruct the transfer moves which are supported near arcs.

Figure~\ref{fig:transfer-move-3} shows how after the transfer move the new intersections $p,q\in f\pitchfork f$ can be paired by a Whitney disk $W$ which is supported near the the union of the arc from $\partial V_0$ which guided the second finger move along $f$ together with sub-arcs of $\partial V_0$ and $\partial V_1$. (In the figure, $W$ appears to hang down underneath the horizontal plane of $f$.) And the pair of new intersections $p_1',q_0'\in W\pitchfork f$ can be paired by a Whitney disk $U'$ formed from a parallel of the Whitney disk $U$ from Lemma~\ref{lem:p0-q1-w-and-accessory-disks} of Section~\ref{sec:step-1-reduce-primary}.
Multiple disjoint parallels of such framed embedded $W$ can be constructed for multiple simultaneous transfer moves by nesting along  $\partial V_0$ and $\partial V_1$ (compare Figures~\ref{fig:double-push-down-1-solid-picture} and~\ref{fig:double-push-down-1B-picture}, but with $V$-parallels replaced by $W$-parallels). And the multiple pairs of intersections between $f$ and these parallels of $W$ can be paired by disjoint parallels of $U'$.


\end{document}